\documentclass[11pt]{amsart}
\usepackage{moreenum}
\usepackage{amssymb}
\usepackage{xspace,cmap}
\usepackage{csquotes}
\usepackage{stmaryrd}

\renewcommand{\restriction}{\mathbin\upharpoonright}   
\newtheorem*{mainthm}{Main Theorem}
\newtheorem{mthm}{Theorem}
\newtheorem{theorem}{Theorem}[section]
\newtheorem{prop}[theorem]{Proposition}
\newtheorem{claim}{Claim}[theorem]
\newtheorem{subclaim}{Subclaim}[claim]
\newtheorem{lemma}[theorem]{Lemma}
\newtheorem{cor}[theorem]{Corollary}
\newtheorem{fact}[theorem]{Fact}

\theoremstyle{definition}
\newtheorem{example}[theorem]{Example}
\newtheorem{definition}[theorem]{Definition}
\newtheorem{notation}[theorem]{Notation}
\newtheorem{conv}[theorem]{Convention}
\newtheorem{goal}[theorem]{Goal}
\newtheorem*{setup}{Setup~\thesection}

\newcounter{condition}

\theoremstyle{remark}
\newtheorem{remark}[theorem]{Remark}

\newcommand\cat[1]{{}^\curvearrowright\langle #1\rangle}
\newcommand{\fork}[2][]{{\pitchfork_{#1}}(#2)}

\newcommand*\axiomfont[1]{\textsf{\textup{#1}}}
\newcommand\zfc{\axiomfont{ZFC}}
\newcommand\gch{\axiomfont{GCH}}
\newcommand\sch{\axiomfont{SCH}}
\newcommand\ch{\axiomfont{CH}}
\def\br{\blacktriangleright}
\newcommand\<{{<}}

\DeclareMathOperator{\CPP}{CPP}
\DeclareMathOperator{\crit}{crit}
\DeclareMathOperator{\range}{Im}
\DeclareMathOperator{\Ult}{Ult}
\DeclareMathOperator{\id}{id}

\newenvironment{blk}[1]{\medskip\noindent\textbf{Building Block~#1.\relax}~}{}
\newcommand\blkref[1]{Building Block~#1}

\newcommand{\cone}[1]{{\mathbb P\mathrel{\downarrow}#1}}
\newcommand{\conea}[1]{{\mathbb A\mathrel{\downarrow}#1}}
\newcommand{\cones}[2]{{\mathbb S_{#1}\mathrel{\downarrow}#2}}
\def\sle{\preceq}
\def\SLE{\preceq}

\def\LE{\le}
\def\sq{\sqsubseteq}
\newcommand{\one}{\mathop{1\hskip-3pt {\rm l}}}

\def\s{\subseteq}
\def\forces{\Vdash}
\def\pI{\textrm{\bf I}}	
\def\pII{\textrm{\bf II}}
\DeclareMathOperator{\tp}{tp}
\DeclareMathOperator{\mtp}{mtp}
\DeclareMathOperator{\col}{Col}
\DeclareMathOperator{\cl}{cl}
\DeclareMathOperator{\otp}{otp}
\DeclareMathOperator{\rng}{Im}
\DeclareMathOperator{\acc}{acc}
\DeclareMathOperator{\nacc}{nacc}
\DeclareMathOperator{\tr}{Tr}
\DeclareMathOperator{\cf}{cf}
\DeclareMathOperator{\refl}{Refl}
\DeclareMathOperator{\ord}{Ord}

\providecommand{\myceil}[2]{\left\lceil #1 \right\rceil^{#2} }

\renewcommand{\mid}{\mathrel{|}\allowbreak}
\newcommand{\lh}{\ell}
\newcommand{\mc}{\mathop{\mathrm{mc}}\nolimits}
\newcommand{\dom}{\mathop{\mathrm{dom}}\nolimits}
\newcommand{\Col}{\mathop{\mathrm{Col}}}
\newcommand\z[1]{\mathring{#1}}

\title[Sigma-Prikry forcing III]{Sigma-Prikry forcing III:\\Down to $\aleph_\omega$}

\author{Alejandro Poveda}
\address{ Einstein Institute of Mathematics, Hebrew University of Jerusalem, Givat-Ram, 91904, Israel.}
\curraddr{Center of Mathematical Sciences and Applications, Harvard University, Cambridge, MA 02138, USA.}

\author{Assaf Rinot}
\address{Department of Mathematics, Bar-Ilan University, Ramat-Gan 5290002, Israel.}
\urladdr{http://www.assafrinot.com}

\author{Dima Sinapova}
\address{Department of Mathematics, Statistics, and Computer Science\\ University of Illinois at Chicago\\ Chicago, IL 60607-7045\\ USA.} 
\curraddr{Department of Mathematics, Rutgers University, Piscataway, NJ 08854-8019, USA.}
\urladdr{https://sites.math.rutgers.edu/~ds2005/}

\begin{document}
\begin{abstract} We prove the consistency of the failure of the singular cardinals hypothesis at $\aleph_\omega$
together with the reflection of all stationary subsets of $\aleph_{\omega+1}$. This shows that two classic results of Magidor (from 1977 and 1982)
can hold simultaneously.
\end{abstract}
\maketitle

\section{Introduction}

Many natural questions cannot be resolved by the standard mathematical axioms ($\zfc$);
the most famous example being Hilbert's first problem, the continuum hypothesis ($\ch$).
At the late 1930's, G\"odel constructed an inner model of set theory \cite{MR2514}
in which the generalized continuum hypothesis ($\gch$) holds, demonstrating, in particular, that $\ch$ is consistent with $\zfc$.
Then, in 1963, Cohen invented the method of forcing \cite{cohen1}
and used it to prove that $\neg\ch$ is, as well, consistent with $\zfc$.

In an advance made by Easton \cite{MR269497}, it was shown that any reasonable behavior of the continuum function $\kappa\mapsto 2^\kappa$ for \emph{regular} cardinals $\kappa$ may be materialized. 
In a review on Easton's paper for \emph{AMS Mathematical Reviews}, Azriel L\'evy writes:

\begin{displayquote}
The corresponding question concerning the singular $\aleph_\alpha$'s is still open, and seems to be one of the most difficult open problems of set theory in the post-Cohen era.
It is, e.g., unknown whether for all $n(n<\omega\rightarrow 2^{\aleph_n}=\aleph_{n+1}^.)$ implies $2^{\aleph_\omega}=\aleph_{\omega+1}$ or not.
\end{displayquote}

A preliminary finding of Bukovsk\'{y} \cite{MR183649} (and independently Hechler) suggested that singular cardinals may indeed behave differently,
but it was only around 1975, with Silver's theorem \cite{MR0429564} and the pioneering work of Galvin and Hajnal \cite{MR376359},
that it became clear that singular cardinals  obey much deeper constraints.
This lead to the formulation of the singular cardinals hypothesis ($\sch$)  as a (correct) relativization of $\gch$ to singular cardinals,
and ultimately to Shelah's \textit{pcf} theory \cite{MR1112424,MR1759410}.
Shortly after Silver's discovery, advances in inner model theory due to Jensen (see \cite{MR0480036})
provided a \emph{covering lemma}  between G\"odel's original model of $\gch$ and many other models of set theory,
thus establishing that any consistent failure of $\sch$ must rely on an extension of $\zfc$ involving large cardinals axioms.

\emph{Compactness} is the phenomenon where if a certain property holds for every strictly smaller substructure of a given object, then it holds for the object itself.
Countless results in topology, graph theory, algebra and logic demonstrate that the first infinite cardinal is compact.
Large cardinals axioms are compactness postulates for the higher infinite.

A crucial tool for connecting large cardinals axioms with singular cardinals was introduced by Prikry in \cite{MR262075}.
Then Silver (see \cite{MR540771}) constructed a model of $\zfc$ whose extension by Prikry's forcing
gave the first universe of set theory with a singular strong limit cardinal $\kappa$ such that $2^\kappa>\kappa^+$.
Shortly after, Magidor \cite{MR491183} proved that the same may be achieved at level of the very first singular cardinal, that is, $\kappa=\aleph_\omega$.
Finally, in 1977, Magidor answered the question from L\'evy's review in the affirmative:
\begin{mthm}[Magidor, \cite{MR491184}]\label{thm1} Assuming the consistency of a supercompact cardinal and a huge cardinal above it,
it is consistent that $2^{\aleph_n}=\aleph_{n+1}$ for all $n<\omega$,
and $2^{\aleph_\omega}=\aleph_{\omega+2}$.
\end{mthm}

Later works of Gitik, Mitchell, and Woodin pinpointed the optimal large cardinal hypothesis required for Magidor's theorem (see \cite{MR1989201,MR2768697}).

Note that Theorem~\ref{thm1} is an incompactness result;
the values of the powerset function are small below $\aleph_\omega$,
and blow up at $\aleph_\omega$.
In a paper from 1982, Magidor obtained a result of an opposite nature,
asserting that \emph{stationary reflection} --- one of the most canonical forms of compactness --- may hold at the level of the successor of the first singular cardinal:
\begin{mthm}[Magidor, \cite{MR683153}]\label{thm2} Assuming the consistency of infinitely many supercompact cardinals,
it is consistent that every stationary subset of $\aleph_{\omega+1}$ reflects.\footnote{That is,
for every subset $S\s\aleph_{\omega+1}$, if for every ordinal $\alpha<\aleph_{\omega+1}$ (of uncountable cofinality),
there exists a closed and unbounded subset of $\alpha$ disjoint from $S$,
then there exists a closed and unbounded subset of $\aleph_{\omega+1}$ disjoint from $S$.}
\end{mthm}

Ever since, it remained open whether Magidor's compactness and incompactness results may co-exist.

The main tool for obtaining Theorem~\ref{thm1} (and the failures of SCH, in general) is Prikry-type forcing (see Gitik's survey \cite{Gitik-handbook}),
however, adding Prikry sequences at a cardinal $\kappa$ typically implies the failure of reflection at $\kappa^+$.
On the other hand, Magidor's proof of Theorem~\ref{thm2} goes through L\'evy-collapsing $\omega$-many supercompact cardinals to become the $\aleph_n$'s, and in any such model SCH would naturally hold at the supremum, $\aleph_\omega$.

Various partial progress to combine the two results was made along the way.
Cummings, Foreman and Magidor \cite{cfm} investigated which sets can reflect in the classical Prikry generic extension.
In his 2005 dissertation \cite{AS},
Sharon analyzed reflection properties of extender-based Prikry forcing (EBPF, due to Gitik and Magidor \cite{Git-Mag}) and devised a way to kill one non-reflecting stationary set,
again in a Prikry-type fashion. He then described an iteration to kill all non-reflecting stationary sets, but the exposition was incomplete.

In the other direction, works of
Solovay \cite{MR0379200},
Foreman, Magidor and Shelah \cite{MR924672},
Veli\v{c}kovi\'{c} \cite{MR1174395},
Todor\v{c}evi\'{c} \cite{MR1261218},
Foreman and Todor\-\v{c}evi\'{c} \cite{MR2115072},
Moore \cite{MR2224055},
Viale \cite{MR2225888},
Rinot \cite{MR2431057},
Shelah \cite{MR2369124},
Fuchino and Rinot \cite{MR2784001},
and Sakai \cite{MR3372613}
add up to a long list of compactness principles that are sufficient to imply the $\sch$.

In \cite{PartI}, we introduced a new class of Prikry-type forcing called \emph{$\Sigma$-Prikry}
and showed that many of the standard Prikry-type forcing for violating SCH at the level of a singular cardinal of countable cofinality fits into this class.
In addition, we verified that Sharon's forcing for killing a single non-reflecting stationary set fits into this class.
Then, in \cite{PartII}, we devised a general iteration scheme for $\Sigma$-Prikry forcing.
From this, we constructed a model of the failure of SCH at $\kappa$ with stationary reflection at $\kappa^+$. 
The said model is a generic extension obtained by 
first violating the $\sch$ using EBPF, and then carrying out an iteration of length $\kappa^{++}$
of the $\Sigma$-Prikry posets to kill all non-reflecting stationary subsets of $\kappa^+$.

Independently, and around the same time, Ben-Neria, Hayut and Unger \cite{BHU} also obtained the consistency of the failure of SCH at $\kappa$ with stationary reflection at $\kappa^+$.
Their proof differs from ours in quite a few aspects;
we mention just two of them. 
First, they cleverly avoid the need to carry out iterated forcing, by invoking iterated ultrapowers, instead.
Second, instead of EBPF,
they violate $\sch$ by using Gitik's very recent forcing \cite{MR4045995} which is also applicable to cardinals of uncountable cofinality.
In addition, the authors devote an entire section to the countable cofinality case, where they get the desired reflection pattern out of a partial supercompact cardinal.
An even simpler proof was then given by Gitik in \cite{GitikRef}.

Still, in all of the above, the constructions are for a singular cardinal $\kappa$ that is very high up; more precisely, $\kappa$ is a limit of inaccessible cardinals.
Obtaining a similar construction for $\kappa=\aleph_\omega$ is quite more difficult, as it involves interleaving collapses.
This makes key parts of the forcing no longer closed, and closure is an essential tool to make use of the indestructibility of the  supercompact cardinals when proving reflection.

In this paper, we extend the machinery developed in \cite{PartI,PartII} to support
interleaved collapses, and show that this new framework captures Gitik's EBPF with interleaved collapses \cite{GitikCollapsin}.
The new class is called $(\Sigma,\vec{\mathbb{S}})$-Prikry.
Finally, by running our iteration of $(\Sigma,\vec{\mathbb{S}})$-Prikry forcings
over a suitable ground model, we establish that Magidor's compactness and incompactness results can indeed co-exist:

\begin{mainthm} Assuming the consistency of infinitely many supercompact cardinals,
it is consistent that all of the following hold:
\begin{enumerate}
\item $2^{\aleph_n}=\aleph_{n+1}$ for all $n<\omega$;
\item $2^{\aleph_\omega}=\aleph_{\omega+2}$;
\item every stationary subset of $\aleph_{\omega+1}$ reflects.
\end{enumerate}
\end{mainthm}

\subsection{Road map} 
Let us give a high-level overview of the proof of the Main Theorem,
and how it differs from the theory developed in Parts I and II of this series.
Here, we assume that the reader has some familiarity with \cite{PartI,PartII}.

\begin{itemize}
\item Recall that the class of $\Sigma$-Prikry forcing provably does not add bounded sets (due to the Complete Prikry Property, $\CPP$) and typically does not collapse cardinals (due to Linked$_0$-property, which provides a strong chain condition).
Since in this paper, we do want to collapse some cardinals, we introduce a bigger class that enables it.
Similarly to what was done in the prequels,
we shall ultimately carry out a transfinite iteration of posets belonging to this class, 
but to have control of the cardinal collapses,
we will be limiting the collapses to take place only on the very first step of the iteration.

This leads to the definition of our new class --- $(\Sigma,\vec{\mathbb{S}})$-Prikry --- 
in which the elements of $\vec{\mathbb{S}}$ are the sole responsible for adding bounded sets.
This class is presented in Section~\ref{weaklySigmaPrikrysection}.
But, before we can get there, there is something more basic to address:

\item Recall that the successive step of the iteration scheme for the $\Sigma$-Prikry  class 
involves a functor that, to each $\Sigma$-Prikry poset $(\mathbb P,\ell_{\mathbb P},c_{\mathbb P})$ and some `problem' $z$ produces a new $\Sigma$-Prikry poset $(\mathbb A,\ell_{\mathbb A},c_{\mathbb A})$ that admits a \emph{forking projection} to 
$(\mathbb P,\ell_{\mathbb P},c_{\mathbb P})$ and `solves' $z$.
In particular, each layer $\mathbb A_n$ projects to $\mathbb P_n$, but while collapses got an entry permit, it is not the case that the $\mathbb A_n$'s admit a decomposition of the form (closed enough forcing) $\times$ (small collapsing forcing).

This leads to the concept of \emph{nice projections} that we study in Section~\ref{SectionNiceProje}.
This concept plays a role in the very definition of the $(\Sigma,\vec{\mathbb{S}})$-Prikry class,
and enables us to analyze the patterns of stationary reflection taking place in various models of interest.
Specifically, a $4$-cardinal \emph{suitability for reflection} property is identified with the goal of, later on, linking 
two different notions of sparse stationary sets.

\item In Section~\ref{SectionGitikForcing}, we prove that Gitik's Extender Based Prikry Forcing with Collapses (EBPFC) fits into the $(\Sigma,\vec{\mathbb{S}})$-Prikry framework. 
We analyze the preservation of cardinals in the corresponding generic extension and, in Subsection~\ref{sec43}, we show that EBPFC is suitable for reflection.

\item Recall that in \cite[\S2.3]{PartII}, we showed how the existence of a forking projection from
$(\mathbb A,\ell_{\mathbb A},c_{\mathbb A})$ to a $\Sigma$-Prikry poset $(\mathbb P,\ell_{\mathbb P},c_{\mathbb P})$ 
can be used to infer that $(\mathbb A,\ell_{\mathbb A},c_{\mathbb A})$ is $\Sigma$-Prikry.
This went through arguing that a forking projection that has the \emph{weak mixing property} 
inherits a winning strategy for a diagonalizability game (Property $\mathcal D$) and the $\CPP$ from $(\mathbb P,\ell_{\mathbb P},c_{\mathbb P})$.
In Section~\ref{SectionNiceForking}, we briefly extend these results to the $(\Sigma,\vec{\mathbb{S}})$-Prikry class,
by imposing the niceness feature of Section~\ref{SectionNiceProje} on the definition of forking projections.
An additional stronger notion (\emph{super nice}) of forking projections is defined, but is not used before we arrive at Section~\ref{Iteration}.

Section~\ref{SectionNiceForking} concludes with a sufficient condition for nice forking projections to preserve suitability for reflection.

\item The next step is to revisit the functor $\mathbb{A}(\cdot,\cdot)$ from \cite[\S4.1]{PartII} for killing one non-reflecting stationary set.
As the elegant reduction obtained at end of \cite[\S5]{PartI} is untrue in the new context of $(\Sigma,\vec{\mathbb{S}})$-Prikry forcing,
we take a step back and look at \emph{fragile} stationary sets.
In Subsection~\ref{killingonesubsection},
we prove that, given a $(\Sigma, \vec{\mathbb{S}})$-Prikry poset $\mathbb{P}$ and a fragile stationary set $\dot{T}$, a variation of $\mathbb{A}(\cdot,\cdot)$ 
yields a $(\Sigma, \vec{\mathbb{S}})$-Prikry poset admitting a nice forking projection to $\mathbb{P}$ and killing the stationarity of $\dot{T}$.
The new functor now uses collapses in the domains of the strategies. 
To appreciate this difference, compare the proof of Theorem~\ref{kilingfragiles} with that of \cite[Theorem~6.8]{PartI}.

Then, in Subsection~\ref{subsection62}, we present a sufficient condition for linking fragile stationary sets with non-reflecting stationary sets.

\item In Section~\ref{Iteration}, we present an iteration scheme for $(\Sigma,\vec{\mathbb{S}})$-Prikry forcings.
Luckily, the scheme from \cite[\S3]{PartII} is successful at accommodating the new class,
once the existence of a coherent system of (super) nice projections is assumed.

\item In Section~\ref{ReflectionAfterIteration}, we construct the model witnessing our Main Theorem.
We first arrange a ground model $V$ of $\gch$ with an increasing sequence $\langle \kappa_n\mid n<\omega\rangle$ of supercompact cardinals that are indestructible under $\kappa_n$-directed-closed notions of forcing that preserve the $\gch$.
Setting $\kappa:=\sup_{n<\omega}\kappa_n$, we then carry out an iteration of length $\kappa^{++}$
of $(\Sigma,\vec{\mathbb{S}})$-Prikry.
The first step is Gitik's EBPFC to collapse $\kappa$ to $\aleph_\omega$,
getting $\gch$ below $\aleph_\omega$ and $2^{\aleph_\omega}=\aleph_{\omega+2}$.
Using the strong chain condition of the iteration, we fix a bookkeeping of all names of potential fragile stationary sets.
Each successor stage $\mathbb P_{\alpha+1}$ of the iteration is obtained by invoking the functor $\mathbb A(\cdot,\cdot)$ with respect to $\mathbb P_\alpha$ and a name for one fragile stationary set given by the bookkeeping list.
At limit stages $\mathbb P_\alpha$, we verify that the iteration remains $(\Sigma,\vec{\mathbb{S}})$-Prikry.
At all stages $\alpha<\kappa^{++}$, we verify that the initial $\gch$ is preserved 
and that this stage is suitable for reflection. From this, 
we infer that in the end model $V^{\mathbb P_{\kappa^{++}}}$, not only that there are no fragile stationary sets,
but, in fact, there are no non-reflecting stationary sets. So, Clauses (1)--(3) all hold in $V^{\mathbb P_{\kappa^{++}}}$.

\end{itemize}
\subsection{Notation and conventions}

Our forcing convention is that $p\le q$ means that $p$ extends $q$.
We write $\cone{q}$ for $\{ p\in\mathbb P\mid p\le q\}$.
Denote $E^\mu_{\theta}:=\{\alpha<\mu\mid \cf(\alpha)=\theta\}$. The sets $E^\mu_{<\theta}$ and $E^\mu_{>\theta}$ are defined in a similar fashion.
For a stationary subset $S$ of a regular uncountable cardinal $\mu$, we write $\tr(S):=\{\delta\in E^\mu_{>\omega}\mid S\cap\delta\text{ is stationary in }\delta\}$.
$H_\nu$ denotes the collection of all sets of hereditary cardinality less than $\nu$.
For every set of ordinals $x$, we denote $\cl(x):=\{ \sup(x\cap\gamma)\mid \gamma\in\ord, x\cap\gamma\neq\emptyset\}$,
and $\acc(x):=\{\gamma\in x\mid \sup(x\cap\gamma)=\gamma>0\}$.
We write $\ch_\mu$ to denote $2^\mu=\mu^+$ and $\gch_{<\nu}$ as a shorthand for $\ch_{\mu}$ holds for every infinite cardinal $\mu<\nu$.

For a sequence of maps $\vec{\varpi}=\langle \varpi_n\mid n<\omega\rangle$ and yet a another map $\pi$ such that $\rng(\pi)\s \bigcap_{n<\omega}\dom(\varpi_n)$,
we let $\vec{\varpi}\bullet\pi$ denote $\langle \varpi_n\circ\pi\mid n<\omega\rangle$.

\section{Nice projections and reflection}\label{SectionNiceProje}

\begin{definition}\label{OrderModuloVarpi} Given a poset $\mathbb P=(P,\le)$ with greatest element $\one$ and a map $\varpi$ with $\dom(\varpi)\supseteq P$,
we derive a poset $\mathbb P^\varpi:=(P,\le^\varpi)$ by letting $$p\LE^\varpi q\text{ iff }(p=\one\;\text{or}\; (p\LE q \text{ and }\varpi(p)=\varpi(q))).$$
\end{definition}

\begin{definition}\label{niceprojection} For two notions of forcing $\mathbb P=(P,\le)$ and $\mathbb S=(S,\sle)$
with maximal elements $\one_{\mathbb P}$ and $\one_{\mathbb S}$, respectively,
we say that a map $\varpi:P\rightarrow S$ is a \emph{nice projection from $\mathbb P$ to $\mathbb S$} iff all of the following hold:
\begin{enumerate}
\item \label{niceprojection1} $\varpi(\one_{\mathbb P})=\one_{\mathbb S}$;
\item \label{niceprojection2}  for any pair $q\LE p$ of elements of $P$, $\varpi(q)\SLE \varpi(p)$;
\item \label{niceprojection3}  for all $p\in P$ and $s\SLE\varpi(p)$,
the set $\{ q\in P\mid q\LE p\land  \varpi(q)\sle s\}$ admits a $\LE$-greatest element,\footnote{By convention, a greatest element, if exists, is unique.}
which we denote by $p+s$. Moreover, $p+s$ has the additional property that $\varpi(p+s)=s$;
\item\label{theprojection} for every $q\LE p+s$, there is $p'\LE^\varpi p$ such that $q=p'+\varpi(q)$; In particular, the map $(p',s')\mapsto p'+s'$ forms a projection from $(\mathbb{P}^\varpi\downarrow p)\times (\mathbb{S}\downarrow s)$ onto $\mathbb{P}\downarrow p$.
\end{enumerate}
\end{definition}

\begin{example}\label{exampleexact} If $\mathbb P$ is a product of the form $\mathbb S\times\mathbb T$,
then the map $(s,t)\mapsto s$ forms a nice projection from $\mathbb P$ to $\mathbb S$.
\end{example}

Note that the composition of nice projections is again a nice projection.

\begin{definition}\label{DefinitionQuotient}
Let $\mathbb P=(P,\le)$ and $\mathbb S=(S,\sle)$
be two notions of forcing and $\varpi:P\rightarrow S$ be a nice projection.
For an $\mathbb{S}$-generic filter $H$, we define \emph{the quotient forcing} $\mathbb{P}/H:=(P/H,\LE_{\mathbb{P}/H})$ as follows:
\begin{itemize}
\item $P/H:=\{p\in P\mid \varpi(p)\in H\}$;
\item for all $p,q\in P/H$, $q\LE_{\mathbb{P}/H} p$ iff there is $s\in H$ with $s\sle \varpi(q)$ such that $q+s\LE p$.
\end{itemize}
\end{definition}

\begin{remark}\label{remarkquotient}
In a slight abuse of notation, we tend to write $\mathbb{P}/\mathbb{S}$ when referring to a quotient as above, without specifying the generic for $\mathbb{S}$
or the map $\varpi$.
By standard arguments, $\mathbb{P}$ is isomorphic to a dense subposet of $\mathbb{S}\ast {\mathbb{P}}/\mathbb{S}$ (see \cite[p.~337]{Abra}).
\end{remark}

\begin{lemma}\label{lemma27} Suppose that $\varpi:\mathbb{P}\rightarrow \mathbb{S}$ is a nice projection.
Let $p\in P$ and set $s:=\varpi(p)$.
For any condition $a\in \mathbb{S}\downarrow s$,
define an ordering $\le_a$ over $\mathbb{P}^\varpi\downarrow p$ by letting $p_0\le_a p_1$
iff $p_0+a\le p_1+a$.\footnote{Strictly speaking, $\le_a$ is reflexive and transitive, but not asymmetric. But this is also always the case, for instance, in iterated forcing.} Then:
\begin{enumerate}
\item
$(\mathbb{S}\downarrow a)\times((\mathbb{P}^\varpi\downarrow p), {\leq_a})$ projects to $\mathbb{P}\downarrow (p+a)$, and
\item
$((\mathbb{P}^\varpi\downarrow p),\leq_{a})$ projects to  $((\mathbb{P}^\varpi\downarrow p), \leq_{a'})$ for all $a'\sle a$.\footnote{Taking $a=s$ we have in particular that  $\mathbb{P}^\varpi\downarrow p$ projects to  $((\mathbb{P}^\varpi\downarrow p), \leq_{a'})$.}
\item
If $\mathbb{P}^\varpi$ contains a $\delta$-closed dense set, then so does $((\mathbb{P}^\varpi\downarrow p), \leq_a)$.
\end{enumerate}
\end{lemma}
\begin{proof}
Note that for $p_0, p_1$ in $\mathbb{P}^\varpi\downarrow p$:
\begin{itemize}
\item
$p_0\le_sp_1$ iff  $p_0\leq^\varpi p_1$, and so $(\mathbb{P}^\varpi\downarrow p,\leq_{s})$ is simply $\mathbb{P}^\varpi\downarrow p$;
\item if $p_0\le_ap_1$, then $p_0\le_{a'}p_1$ for any $a'\sle a$;
\item in particular, if $p_0\le^\varpi p_1$, then $p_0\le_ap_1$ for any $a$ in $\mathbb S\downarrow s$.
\end{itemize}

The first projection is given by $(a',r)\mapsto r+a'$,
and the second projection is given by the identity.

For the last statement, denote  $\mathbb{P}_a:=((\mathbb{P}^\varpi\downarrow p), \leq_a)$ and let $D$ be a  $\delta$-closed dense subset of $\mathbb{P}^\varpi$. We claim that $D_a:=\{r\in \mathbb{P}^\varpi\downarrow p\mid r+a\in D\}$ is a $\delta$-closed dense subset of $\mathbb{P}_a$. For the density, if $r\in \mathbb{P}^\varpi\downarrow p$, let $q\leq^\varpi r+a$ be in $D$. Then,  by Clause~\eqref{theprojection} of 
Definition \ref{niceprojection}, $q=r'+a$ for some $r'\leq^\varpi r$, and so $r'\leq_a r$ and $r'\in D_a$. For the  closure,  suppose that $\langle p_i\mid i<\tau\rangle$ is a 
$\LE_{\mathbb{P}_a}$-decreasing sequence in $D_a$ for some $\tau<\delta$. Setting $q_i:=p_i+a$, that means that $\langle q_i\mid i<\tau\rangle$ is a $\LE^\varpi$-decreasing sequence in $D$ and so has a lower bound $q$. More precisely, $q\in D$ and for each $i$, $q\LE q_i$ and $\varpi(q)=\varpi(q_i)=a$. Let $p^*\LE^\varpi p$, be such that $p^*+a=q$. Here again we use  Clause~\eqref{theprojection} of Definition~\ref{niceprojection}. Then $p^*\in D_a$, which is the desired $\LE_{\mathbb{P}_a}$-lower bound.
\end{proof}

The next lemma clarifies the relationship between the different generic extensions that we will be considering:

\begin{lemma}\label{TowerOfExtensions}
Suppose that $\varpi:\mathbb{P}\rightarrow \mathbb{S}$, $p\in P$, $s:=\varpi(p)$ and  $\le_a$ for $a$ in $ \mathbb{S}\downarrow s$ are as in the above lemma. Let $G$ be $\mathbb{P}$-generic with $p\in G$.

Next, let $H\times G^*$ be $((\mathbb{S}\downarrow s)\times (\mathbb{P}^\varpi\downarrow p))/G$-generic over $V[G]$. 
For each $a\in H$, let $G_a$ be the $((\mathbb{P}^\varpi\downarrow p), \leq_a)$-generic filter obtained from $G^*$.
Then:
\begin{enumerate}
\item For any $a\in H$, $V[G]\s V[H\times G_a]\s V[H\times G^*]$, and $G\supseteq G_a\supseteq G^*$;
\item For any pair $a'\sle a$ of elements of $H$, $V[H\times G_{a'}]\s V[H\times G_a]$, and $G_{a'}\supseteq G_a$;
\item $G\cap (\mathbb{P}^{\varpi}\downarrow p)=\bigcup_{a\in H}G_a$. 
\end{enumerate}
\end{lemma}
\begin{proof}
For notational convenience, denote $\mathbb{P}^*:=\mathbb{P}^{\varpi}\downarrow p$  and $\mathbb{S}^*:={\mathbb{S}\downarrow s}$. 

The first two items follow from the corresponding choices of the projections in Lemma~\ref{lemma27}.
For the third item, first note that  $\bigcup_{a\in H} G_a\s \mathbb{P}^*\cap G$. Suppose that $r^*\in G\cap\mathbb{P}^*$. In $V[H]$, define $$D:=\{r\in P^*\mid(\exists a\in H) (r\leq_a r^*)\lor r\perp_{\mathbb{P}/H}r^*\}.\footnote{Since $r \in P^*$ then $\varpi(r)=\varpi(p)\in H$ and thus $r$ is a condition in $\mathbb{P}/H$.}$$
\begin{claim}
$D$ is a dense set in $\mathbb{P}^*$.
\end{claim}
\begin{proof}
Let $r\in P^*$.  If $r\perp_{\mathbb{P}/H}r^*$, then $r\in D$, and we are done. So suppose that $r$ and $r^*$ are compatible in $\mathbb{P}/H$. Let $q\in P/H$ be such that, $q\le_{\mathbb{P}/H}r$ and $q\le_{\mathbb{P}/H}r^*$. Let $a\sle\varpi(q)$ in $H$ be such that $q+a\LE r, r^*$. By Clause~\eqref{theprojection} of Definition~\ref{niceprojection} for 
 $\varpi$ we may let $r'\LE^\varpi r$, be such that  $r'+a=q+a$. In particular,  $r'\le_a r^*$, and so $r'\in D$.
\end{proof}
Now let $r\in D\cap G^*$. Since both $r,r^*\in G$, it must be that $r\leq_a r^*$ for some $a\in H$. And since $r\in G^*\s G_a$, we get that $r^*\in G_a$.
\end{proof}

\begin{lemma}\label{lemma28}
Suppose that $\varpi:\mathbb{P}\rightarrow \mathbb{S}$ is a nice projection and that $\delta<\kappa$ is a pair of infinite regular cardinals for which the following hold:
\begin{enumerate}
\item $|\mathbb{S}|<\delta$ and $\mathbb{P}^{\varpi}$ contains a $\delta$-directed-closed dense subset; 
\item After forcing with $\mathbb{S}\times \mathbb{P}^\varpi$,  $\kappa$ remains regular;
\item $V^\mathbb{P}\models``E^\kappa_{<\delta}\in I[\kappa]"$.\footnote{For the definition of the ideal $I[\kappa]$, see \cite[Definition 2.3]{ShelahBook}.}
\end{enumerate}

Let $p\in \mathbb{P}$ and set $s:=\varpi(p)$.
Then, for any $\mathbb{P}$-generic $G$ with $p\in G$,
the quotient $((\mathbb{S}\downarrow s)\times (\mathbb{P}^\varpi\downarrow p))/G$
preserves stationary subsets of $(E^\kappa_{<\delta})^{V[G]}$.
\end{lemma}
\begin{proof} For the scope of the proof denote $\mathbb{P}^*:=\mathbb{P}^\varpi\downarrow p$ and $\mathbb{S}^*:=\mathbb{S}\downarrow s$. 

Let $H\times G^*$ be $(\mathbb{S}^*\times \mathbb{P}^*)/G$-generic over $V[G]$.
For each $a\in H$,  let $G_a$ be the $((\mathbb{P}^\varpi\downarrow p), \leq_a)$-generic obtained from $G^*$.
For notational convenience we will also denote  $\mathbb{P}_a:=((\mathbb{P}^\varpi\downarrow p), \leq_a)$. 
Combining Clause~(1) of our assumptions with Lemma~\ref{lemma27}(3) we have that $\mathbb{P}_a$ contains a  $\delta$-closed dense subset, hence it is $\delta$-strategically-closed. Standard arguments imply that $\mathbb{P}^*/G_a$ is  $\delta$-strategically-closed  over $V[G_a]$.\footnote{Note that $\delta$ is still regular in $V[G_a]$, as $\mathbb{P}_a$ being $\delta$-strategically-closed over $V$.}

Suppose for contradiction that $V[G]\models ``T\s E^\kappa_{<\delta}$ is a stationary set'', but that $T$ is nonstationary in $V[H\times G^*]$. Since $|\mathbb{S}|<\delta$,  Clause~(1) above   and Lemma~\ref{TowerOfExtensions}(1) yield $(E^\kappa_{<\delta})^V=(E^\kappa_{<\delta})^{V[G^*]}=(E^\kappa_{<\delta})^{V[G]}=(E^\kappa_{<\delta})^{V[G_a]}$ for all $a\in H$. Thus, we can unambiguously  denote this set  by $E^\kappa_{<\delta}$.

\begin{claim} Let $a\in H$. Then,
$T$ is non-stationary in $V[H\times G_a]$.
\end{claim}
\begin{proof}
Otherwise, if $T$ was stationary in $V[H][G_a]$,  since $|\mathbb{S}|<\delta<\kappa$,
$$T':=\{\alpha\in E^\kappa_{<\delta}\mid (\exists r\in G_a)(b,r)\Vdash_{\mathbb{S}^*\times \mathbb{P}_a}\alpha\in\dot{T}\}$$ is a stationary set lying in $V[G_a]$, where $b\in H$. Combining Lemma~\ref{TowerOfExtensions}(1) with  the fact that $\mathbb{S}$ is small we have $I[\kappa]^{V[G]}\s I[\kappa]^{V[H\times G_a]}\s I[\kappa]^{V[G_a]}.$
Thus, Clause~(3) of our assumption yields $E^\kappa_{<\delta}\in  I[\kappa]^{V[G_a]}$.
Now, since $\mathbb{P}^*/G_a$ is $\delta$-strategically closed in $V[G_a]$,
by Shelah's theorem \cite{Sh:108}, $\mathbb{P}^*/G_a$ preserves stationary subsets of $E^\kappa_{<\delta}$ hence $T'$ remains stationary in $V[G^*]$. Once again, since $\mathbb{S}$ is a small forcing 
$T'$ remains stationary in the further generic extension $V[H\times G^*]$.
This is a contradiction with $T'\s T$ and our assumption that $T$ was non-stationary in
$V[H\times G^*]$.
\end{proof}

Then for every $a\in H$ let $C_a$ be a club in $V[H\times G_a]$ disjoint from $T$. Since $\mathbb{S}$ is a small forcing, we may assume that $C_a\in V[G_a]$. Let $\dot{C}_a$ be a $\mathbb{P}_a$-name for this club such that
\begin{itemize}
\item $p\Vdash_{\mathbb{P}_a} ``\dot{C}_a$ is a club'', and
\item $(a,p)\Vdash_{(\mathbb{S}^*\times \mathbb{P}_a)} ``\dot{C}_a\cap \dot{T}=\emptyset$''.\footnote{Here we identify both $\dot{C}_a$ and $\dot{T}$ with $(\mathbb{S}^*\times \mathbb{P}_a)$-names in the natural way. See, e.g., Lemma~\ref{TowerOfExtensions}(1).}
\end{itemize}

Since $\mathbb{S}$ is a small forcing, we may fix some $a\in H$, such that $$T_a:=\{\alpha\in E^\kappa_{<\delta} \mid \exists r\in G[ r\LE p, \varpi(r)=a\ \&\ r\Vdash_{\mathbb{P}}\check\alpha\in \dot{T}]\}$$ is stationary in $V[G]$.

\begin{claim} There is a condition $p^*\le_a p$ and an ordinal $\gamma<\kappa$ such that $(a,p^*)\Vdash_{(\mathbb{S}^*\times \mathbb{P}_a)} \gamma\in\dot{C}_a\cap \dot{T}$.
\end{claim}
\begin{proof}
Work first in $V[G]$. Let $M$ be an elementary submodel of $H_\theta$ (for a large enough regular cardinal $\theta$),
such that:
\begin{itemize}
\item $M$ contains all the relevant objects, including $\dot{C}_a$ and $(a,p)$;
\item $\gamma:=M\cap\kappa\in T_a$
\end{itemize}

Let $\chi=\cf^V(\gamma)$ and $\langle\gamma_i \mid i<\chi\rangle\in V$ be an increasing sequence with limit $\gamma$. As $\gamma\in T_a\s E^\kappa_{<\delta}$ we have  $\chi<\delta$.
Since $\gamma\in T_a$, we may fix some $r\in G$ with $\varpi(r)=a$ such that $r\Vdash_{\mathbb{P}}\check\gamma\in \dot{T}$. Also, by Clause~\eqref{theprojection} of Definition~\ref{niceprojection},  there is $r^*$ a condition in $\mathbb{P}^*$ such that $r^*+a=r$. Clearly, $r^*\in G$.

Recall that for a condition $p'\in\mathbb{P}_a$, we write $p'\in\mathbb{P}_a/G$ whenever $p'+a\in G$ (see Definition~\ref{DefinitionQuotient}). Then, for all $q\in P$, $q\forces_{\mathbb{P}}``p'\in\mathbb{P}_a/\dot{G}$'' iff $q\leq_{\mathbb{P}} p'+a$.

Since $p$ $\mathbb{P}_a$-forces $\dot{C}_a$ to be a club, for all $\beta<\kappa$, there is $\beta\leq\alpha<\kappa$, and  $p'\le_a p$ forcing $\alpha\in \dot{C}_a$. By density we can actually find such $p'$ in $\mathbb{P}_a/G$. Thus, by elementarity and  since $p\in M$, for all $i<\chi$, there is $\alpha\in M\setminus\gamma_i$, and $p'\le_a p$ in $M$, such that $p'\in \mathbb{P}_a/G$ and $p'\forces_{\mathbb{P}_a}\alpha\in \dot{C}_a$.

Fix a $\mathbb{P}$-name $\dot{M}$ and, by strengthening $r^*$ if necessary,
suppose that for some $a'\sle a$ in $H$,  $r^*+a'$ forces the above properties to hold. In particular, for all $i<\chi$ and $q\leq_{\mathbb{P}} r^*+a'$, there are $q'\leq_{\mathbb{P}}q$, $p'\leq_a p$ and $\alpha\geq\gamma_i$, such that $q'\leq_{\mathbb{P}}p'+a$, $q'\forces_{\mathbb{P}}`` p'\in \dot{M}, \alpha\in \dot{M}"$ and  $p'\forces_{\mathbb P_a}\alpha\in\dot C_a$.

Breaking this down, we get that for all $i<\chi$, if $r'\leq_a r^*$ and $b\sle a'$, then there are $b'\sle b$, $q'\leq_{a}r'$, $p'\leq_a p$ and $\gamma_i\leq\alpha<\kappa$, such that  $q'\le_a p'$, $q'+b'\forces_{\mathbb{P}}`` p'\in \dot{M}, \alpha\in \dot{M}"$ and  $p'\forces_{\mathbb P_a}\alpha\in\dot C_a$. The later also gives that $q'\forces_{\mathbb P_a}\alpha\in\dot C_a$. In other words we have the following for each $i$:

($\dagger$) for all $r'\le_a r^*$ and $b\sle a'$, there are $q'\leq_{a}r'$, $b'\sle b$, and $\alpha<\kappa$, such that
$q'+b'\forces_{\mathbb{P}}`` \alpha\in \dot{M}\setminus\gamma_i"$ and $q'\forces_{\mathbb P_a}\alpha\in\dot C_a$.

Then, since the closure of $\le_a$ is more than $|\mathbb{S}|$, we get that for each $i$:

($\dagger\dagger$) for all $r'\le_a r^*$, there are $q'\leq_{a}r'$, and a dense $D\subset {\mathbb{S}} $, such that for all $b\in D$, there is $\alpha<\kappa$, with $q'+b\forces_{\mathbb{P}}`` \alpha\in \dot{M}\setminus\gamma_i"$ and $q'\forces_{\mathbb P_a}\alpha\in\dot C_a$.

Working in $V$, construct a  $\le_a$-decreasing sequences $\langle  q_i\mid i\le\chi\rangle$ of conditions in $\mathbb P_a$ below $r^*$  and a family $\langle D_i\mid i<\chi\rangle$ of dense sets of $\mathbb{S}$ with the following properties: For each  $i<\chi$ and $b\in D_i$, there is $\alpha<\kappa$, such that:
\begin{enumerate}
\item
$q_{i+1}+b\Vdash_{\mathbb{P}} \alpha\in \dot{M}\setminus \gamma_i$, and
\item $q_{i+1}\forces_{\mathbb P_a}\alpha\in\dot C_a$.
\end{enumerate}
\noindent
At successor stages we use ($\dagger\dagger$), and at limit stages we take lower bounds.

Let $p^*:= q_\chi$. Since we can find $p^*$ as above $\le_a$-densely often below $r^*$, we may assume that $p^*+a\in G$.

Now go back to $V[G]$. For each $i<\chi$, let $b_i\in D_i\cap H$, where recall that $H$ is the induced $\mathbb{S}$-generic from $G$. Also, let $\alpha_i$ witness that $b_i\in D_i$. Then $p^*\forces_{\mathbb P_a} \alpha_i\in \dot{C}_a $, and in $V[G]$, $\alpha_i\in M\setminus \gamma_i$ (since $q_{i+1}+b_i\in G$), so $\gamma=\sup_{i<\chi}\alpha_i$. It follows that  $p^*\forces_{\mathbb P_a}\gamma\in\dot{C}_a$.

Finally, as $p^*+a\leq r^*+a=r$, we have that $p^*+a\Vdash_{\mathbb{P}}\gamma\in\dot{T}$.
Recall that the projection from $\mathbb{S}^*\times \mathbb{P}_a$ to $\mathbb{P}$ is witnessed by $(a', p')\mapsto p'+a'$ (Lemma \ref{lemma27}), hence it follows that  $(a, p^*)\Vdash_{\mathbb{S}^*\times \mathbb{P}_a}\gamma\in\dot{T}$.
So, $(a, p^*)\Vdash_{\mathbb{S}^*\times \mathbb{P}_a}\check\gamma\in\dot{T}\cap \dot{C}_a$.
\end{proof}

Choose $p^*$ as in the above lemma. That gives a contradiction with $(a,p)\Vdash_{(\mathbb{S}^*\times \mathbb{P}_a)} \dot{C}_a\cap \dot{T}=\emptyset$.
\end{proof}

\begin{definition}
For stationary subsets $\Delta,\Gamma$ of a regular uncountable cardinal $\mu$,
$\refl(\Delta,\Gamma)$ asserts that for every stationary subset $T\s \Delta$,
there exists $\gamma\in\Gamma\cap E^\mu_{>\omega}$ such that $T\cap\gamma$ is stationary in $\gamma$.
\end{definition}

We end this section by establishing a sufficient condition for $\refl(\ldots)$ to hold in generic extensions;
this will play a crucial role at the end of Section~\ref{SectionNiceForking}.

\begin{definition}\label{suitableforiteration} For infinite cardinals $\tau<\sigma<\kappa<\mu$,
we say that $(\mathbb{P},\mathbb{S},\varpi)$ is \emph{suitable for reflection with respect to $\langle\tau,\sigma,\kappa,\mu\rangle$} iff all the following hold:
\begin{enumerate}
\item \label{suitableforiteration0} $\mathbb P$ and $\mathbb S$ are nontrivial notions of forcing;
\item \label{suitableforiteration1} $\varpi:\mathbb{P}\rightarrow\mathbb{S}$ is a nice projection and $\mathbb{P}^\varpi$ contains a $\sigma$-directed-closed dense subset;\footnote{In all cases of interest, $\sigma$ will be a regular cardinal.}
\item \label{suitableforiteration2}  In any forcing extension by $\mathbb P$ or ${\mathbb{S}\times\mathbb{P}^\varpi}$, $|\mu|=\cf(\mu)=\kappa=\sigma^{++}$;
\item \label{suitableforiteration3} For any $s\in S\setminus \{\one_\mathbb{S}\}$,
there is a cardinal $\delta$ with $\tau^+<\delta<\sigma$,
such that $\mathbb{S}\downarrow s\cong\mathbb{Q}\times\col(\delta,{<}\sigma)$ for some notion of forcing $\mathbb Q$ of size $<\delta$.
\end{enumerate}
\end{definition}

\begin{lemma}\label{ebfreflection}
Let $(\mathbb{P},\mathbb{S},\varpi)$ be suitable for reflection with respect to $\langle\tau,\sigma,\kappa,\mu\rangle$.
Suppose $\sigma$ is a supercompact cardinal indestructible under forcing with $\mathbb{P}^{\varpi}$.
Then $V^{\mathbb{P}}\models \refl(E^{\mu}_{\leq\tau},E^{\mu}_{<\sigma^+}).$
\end{lemma}
\begin{proof}
By Definition~\ref{suitableforiteration}\eqref{suitableforiteration2}, it suffices to prove that
$V^{\mathbb{P}}\models \refl(E^{\kappa}_{\leq\tau},E^{\kappa}_{<\sigma^+})$.

Let $G$ be $\mathbb{P}$-generic. In $V[G]$,
let $T$ be a stationary subset of $E^{\kappa}_{\leq\tau}$. Suppose for simplicity that this is forced by the empty condition.

\begin{claim} Let $p\in G$ be such that  ${s}:=\varpi({p})$ strictly extends $\one_\mathbb{S}$.

Then,
the quotient $((\mathbb{S}\downarrow  {s})\times (\mathbb{P}^\varpi\downarrow {p}))/G$ preserves the stationarity of $T$.
\end{claim}
\begin{proof}
Using Clauses \eqref{suitableforiteration0} and \eqref{suitableforiteration1} of Definition~\ref{suitableforiteration},
let us pick any $p\in G$ for which $s:= \varpi(p)$ strictly extends $\one_{\mathbb S}$.
Back in $V$, using Clause~\eqref{suitableforiteration3} of Definition~\ref{suitableforiteration},
fix a cardinal $\delta$ with $\tau^+<\delta<\sigma$,
a notion of forcing $\mathbb Q$ of size $<\delta$,
and an isomorphism $\iota$ from $\mathbb{S}\downarrow s$ to $\mathbb{Q}\times\col(\delta,{<}\sigma)$.
Let $\iota_0,\iota_1$ denote the unique maps to satisfy $\iota(s')=(\iota_0(s'),\iota_1(s'))$.
By Example~\ref{exampleexact}, $\iota_0$ and $\iota_1$ are nice projections.
Set $\pi:=(\iota_0\circ\varpi)\restriction(\mathbb{P}\downarrow p)$ and $\varrho:=(\iota_1\circ \varpi)\restriction(\mathbb{P}\downarrow p)$, so that $\pi$ and $\varrho$ are nice projection from $\mathbb{P}\downarrow p$ to $\mathbb{Q}$ and from $\mathbb{P}\downarrow p$ to $\Col(\delta,\<\sigma)$, respectively. Note that the definition of $\pi$ depends
on our choice of $\mathbb{Q}$, which depends on our choice of $p$, and formally we defined $\pi$ as a projection from
$\mathbb{P}\downarrow p$ to $\mathbb{Q}$. In an slight abuse of notation we will write $\mathbb{P}^\pi$ rather than ${(\mathbb{P}\downarrow p)}^\pi$. More precisely,  $\mathbb{P}^\pi:=(\{q\in P\mid q\leq p\}, \leq^\pi)$.\footnote{Recall Definition~\ref{OrderModuloVarpi}.}

\begin{subclaim}\label{SubclaimProjections}
\hfill
\begin{enumerate}[label=(\roman*)]
\item $(\mathbb{S}\downarrow s)\times (\mathbb{P}^\varpi\downarrow p)$ projects onto $\mathbb{Q}\times (\mathbb{P}^\pi\downarrow p)$, and that projects onto $\mathbb{P}\downarrow p$;
\item $(\mathbb{S}\downarrow s)\times (\mathbb{P}^\varpi\downarrow p)$ projects onto $\mathbb{Q}\times (\mathbb{P}^\pi\downarrow p)$, and that projects onto $\mathbb{P}^\pi\downarrow p$;
\item $(\mathbb{S}\downarrow s)\times (\mathbb{P}^\varpi\downarrow p)$ projects onto $\col(\delta,{<}\sigma)\times (\mathbb{P}^\varpi\downarrow p)$, and that projects onto $\mathbb{P}^\pi\downarrow p$.
\end{enumerate}
\end{subclaim}
\begin{proof} (i)
For the first part, the map $(s',p')\mapsto (\iota_0(s'),p'+\iota_1(s'))$ is such a projection, where the $+$ operation is computed with respect to the nice projection $\varrho$. For the second part, the map $(q',p')\mapsto p'+q'$ gives such a projection, where the $+$ operation is computed with respect to $\pi$.\footnote{See Clause~\eqref{theprojection} of Definition~\ref{niceprojection}\eqref{theprojection} regarded with respect to $\pi$.}

(ii) For the second part, the map $(q',p')\mapsto p'$ is such a projection.

(iii) For the first part, the map $(s',p')\mapsto (\iota_1(s'),p')$ is such a projection.
For the second part, the map $(c',p')\mapsto p'+c'$ is such a projection, where the $+$ operation is with respect to the nice projection $\varrho$.
\end{proof}

By Definition~\ref{suitableforiteration}\eqref{suitableforiteration2},
in all forcing extensions with posets from Clause~(i), $\kappa$ is a cardinal which is the double successor of $\sigma$.
But then, since $|\mathbb Q|<\kappa$, it follows from the second part of Clause~(ii) that $\kappa$ is the double successor of $\sigma$
in forcing extensions by $\mathbb{P}^\pi\downarrow p$. Actually, in any forcing extension by $\mathbb{P}^\pi$.\footnote{{Note that  Subclaim~\ref{SubclaimProjections} remains valid even if we replace $p$ by any $p'\leq p$. For instance, regarding Clause~(i), we will then have that $(\mathbb{S}\downarrow \varpi(p'))\times (\mathbb{P}^\varpi\downarrow p')$ projects onto $(\mathbb{Q}\downarrow \pi(p'))\times (\mathbb{P}^\pi\downarrow p')$ and that  this latter projects onto $\mathbb{P}\downarrow p'$.}}  
Altogether, in all forcing extensions with posets from the preceding subclaim, $\kappa$ is the double successor of $\sigma$.

Let $G_q\times G^*$ be $(\mathbb{Q}\times (\mathbb{P}^\pi\downarrow p))/G$-generic over $V[G]$. Next, we want to use Lemma~\ref{lemma28}
to show that $T$ remains stationary in $V[G_q\times G^*]$, so we have to verify its assumptions hold.  The next claim along with Definition~\ref{suitableforiteration}\eqref{suitableforiteration3} yields Clause~(1) of Lemma~\ref{lemma28}:
\begin{subclaim}
$\mathbb{P}^{\pi}$ contains a $\delta$-directed-closed dense set.
\end{subclaim}
\begin{proof}
Let $D\s \mathbb{P}^\varpi$ be the $\sigma$-directed-closed dense subset given by Clause~(2) of our assumptions.\footnote{I.e., $D$ is dense and $\sigma$-directed-closed with respect to $\leq^\varpi$.} We claim that the set $$D' := \{q+c\mid q\in D,\, q\leq^\varrho p, \,c\leq_{\Col(\delta,<\sigma)} \varrho(p)\}$$ is a $\delta$-directed-closed dense subset of $\mathbb{P}^{\pi}$.  Here $q+c$ is computed with respect to the projection map $\varrho:\mathbb{P}\downarrow p\rightarrow \Col(\delta,<\sigma)$.

For density, if $q\in \mathbb{P}^\pi$, since $\varrho$ is a nice projection, let
$q'\leq^\varrho p$ be with $q'+\varrho(q)=q$. Now, let $q''\leq^\varpi q'$ be in $D$. In particular, $\varrho(q'')=\varrho(p)$, so that $q''+\varrho(q)$ is well-defined. Then $q''+\varrho(q)\leq^\pi q$,
$q''\leq^\varrho p$ and $q''+\varrho(q)\in D'$.

For directed closure, suppose that $\nu<\delta$ and $\{p_\alpha\mid \alpha<\nu\}$ is a $\leq^\pi$-directed set in $D'$. For each $\alpha<\nu$,
write $p_\alpha = q_\alpha+c_\alpha$, where $q_\alpha\in D$, $\varrho(q_\alpha)=\varrho(p)$, and $c_\alpha=\varrho(p_\alpha)\in \Col(\delta,{<}\sigma)\downarrow \varrho(p)$. Note that for each $\alpha<\nu$, $\pi(q_\alpha)=\pi(p_\alpha)$.
Also, $p_\alpha$ and $p_\beta$ being $\leq$-compatible entails  $c_\alpha\cup c_\beta\in\Col(\delta,{<}\sigma)$. Finally, since $q_\alpha,q_\beta\in D$ it follows 
that $q_\alpha$ and $q_\beta$ are $\leq^\varpi$-compatible. 

Thus, $\{q_\alpha\mid \alpha<\nu\}$ is a $\leq^\varpi$-directed set in $D$ of size $<\kappa$, so we may  let $q\in D$ be a $\leq^\varpi$-lower bound for $\{q_\alpha\mid \alpha<\nu\}$. In particular, $\varrho(q)=\varrho(p)$ and $\pi(q)=\pi(p_\alpha)$ for all $\alpha<\nu$. Additionally,
 let $c:=\bigcup_{\alpha<\nu} c_\alpha\in \Col(\delta,{<}\sigma)$. 
Then $q+c\in D'$ is the desired $\leq^\pi$-lower bound for
$\{p_\alpha\mid \alpha<\nu\}$.
\end{proof}
Clause~(2) of Lemma~\ref{lemma28}  follows from $|\mathbb{Q}|<\kappa$ and the fact that $\mathbb{P}^\pi$ forces $``\kappa=\sigma^{++}$'' (see our comments after Subclaim~\ref{SubclaimProjections}). 
Finally, for Clause~(3), we argue as follows: By \cite[Lemma~4.4]{Sh:351} and  Definition~\ref{suitableforiteration}\eqref{suitableforiteration2}, $V^{(\mathbb{Q}\times (\mathbb{P}^\pi\downarrow p))}\models ``E^{\kappa}_{<\delta}\s E^{\sigma^{++}}_{<\sigma^+}\in  I[\sigma^{++}]$'', thus  $V^{(\mathbb{Q}\times (\mathbb{P}^\pi\downarrow p))}$ thinks that $E^\kappa_{<\delta}\in I[\kappa]$. By the previous subclaim, $E^\kappa_{<\delta}$ is computed the same way both in $V^{(\mathbb{Q}\times (\mathbb{P}^\pi\downarrow p))}$ and $V^{\mathbb{P}}$. Also, remember that this latter model thinks that $``\kappa=\sigma^{++}$'', as well. Thereby if we put everything in the same canopy we infer that $V^{\mathbb{P}}\models ``E^\kappa_{<\delta}\in I[\kappa]$''.

Thereby, $T$ remains stationary in $V[G_q\times G^*]$. 
As $\mathbb Q$ is small, we may fix  $T'\s T$ such that $T'$ is in $V[G^*]$ and moreover stationary in $V[G^*]$.
As established earlier, $V[G^*]\models ``T'\s E^{\sigma^{++}}_{<\sigma^+}\in I[\sigma^{++}]"$. 
Since both $\mathbb{P}^{\pi}$ and $\mathbb{P}^\varpi$ are $\delta$-strategically closed (actually $\mathbb{P}^\varpi$ is more), we have that,  in $V[G^*]$, the quotient $(\col(\delta,{<}\sigma)\times(\mathbb{P}^\varpi\downarrow p))/G^*$ is also $\delta$-strategically closed.
So, again it follows that $(\col(\delta,{<}\sigma)\times(\mathbb{P}^\varpi\downarrow p))/G^*$ preserves the stationarity of $T'$.

Finally, since $\mathbb{S}\downarrow s\cong \mathbb{Q}\times\col(\delta,{<}\sigma)$,
the quotient
$$((\mathbb{S}\downarrow s)\times (\mathbb{P}^\varpi\downarrow p))/(\col(\delta,{<}\sigma)\times (\mathbb{P}^\varpi\downarrow p))$$ is
isomorphic to $\mathbb{Q}$, which is a small of forcing.
Altogether, $T'$ (and hence also $T$) remains stationary in the generic extension  $V[G]$. 
\end{proof}

Let  $p\in G$ be such that $s:=\varpi(p)$ strictly extends $\one_\mathbb{S}$. 
Let $H$ be the generic filter for $\mathbb{S}$ induced by $\varpi$ and $G$.
Let $G^*$ be such that $H\times G^*$ is generic for $((\mathbb{S}\downarrow s)\times (\mathbb{P}^\varpi\downarrow p))/G$.
By the above claim,
$T$ is still stationary in $V[G^*][H]$. Also, by Definition~\ref{suitableforiteration}\eqref{suitableforiteration2}, $T\s (E^{\sigma^{++}}_{\leq\tau})^{V[G^*][H]}$.

Using that $\sigma$ is a supercompact indestructible under $\mathbb{P}^\varpi$, let (in $V[G^*]$)  $$j:V[G^*]\rightarrow M$$ be a $\kappa$-supercompact embedding with  $\crit(j)=\sigma$.
We shall want to lift this embedding to $V[G^*][H]$.

Work below the condition $s$ that we fixed earlier.
Recall that $\mathbb{S}\downarrow s\cong\mathbb{Q}\times\col(\delta,{<}\sigma)$ for some poset $\mathbb Q$ of size $<\delta$ with $\tau^+<\delta<\sigma$. So, $H$ may be seen as a product of two corresponding generics, $H=H_0\times H_1$.
For the ease of notation, put $\mathbb{C}:=\Col(\delta, {<}\sigma)$.

Since $\mathbb Q$ has size $<\delta<\crit(j)$, we can lift $j$ to an embedding $$j:V[G^*][H_0]\rightarrow M'.$$
Then we lift  $j$ again to get $$j:V[G^*][H]\rightarrow N$$ in an outer generic extension of $V[G^*][H]$ by $j(\mathbb{C})/H_1$.
Since $j(\mathbb{C})/H_1$ is $\delta$-closed in $M'[H_1]$ and this latter  model is closed under $\kappa$-sequences in $V[G^*][H]$, it follows that $j(\mathbb{C})/H_1$ is also $\delta$-closed in $V[G^*][H]$.

Set $\gamma:=\sup(j``\kappa)$. Clearly, $j(T)\cap \gamma=j``T$. 
Note that, by virtue of the collapse $j(\mathbb{C})$,
$N\models ``|\kappa|=\delta\ \&\ \cf(\gamma)\leq \cf(|\kappa|)=\delta<j(\sigma)$''.

Once again,    \cite[Lemma~4.4]{Sh:351}  Definition~\ref{suitableforiteration}\eqref{suitableforiteration2} together yield $$T\s (E^{\sigma^{++}}_{\leq\tau})^{V[G^*][H]}\s (E^{\sigma^{++}}_{<\sigma^+})^{V[G^*][H]}\in I[\sigma^{++}]^{V[G^*][H]}.$$  As customary, Shelah's theorem (c.f.~\cite{Sh:108})  along with  the $\delta$-closedness of $j(\mathbb{C})/H_1$ in $V[G^*][H]$ imply that this latter forcing preserves the stationarity of  $T$. 
Now, a standard argument shows that that  $ j(T)\cap \gamma$ is stationary in  $N$. Thus,  $``\exists \alpha\in E^{j(\kappa)}_{{<}j(\sigma)} ( j(T)\cap\alpha\text{ is stationary in }\alpha)$'' holds in $N$.
So, by elementarity, in  $V[G^*][H]$, $T$ reflects at a point of cofinality $<\sigma^+$.\footnote{Actually, at a point of cofinality $<\sigma$.}
Since reflection is downwards absolute, it follows that $T$ reflects at a point of cofinality $<\sigma^+$ in $V[G]$, as wanted.
\end{proof}

\section{$(\Sigma,\vec{\mathbb{S}})$-Prikry forcings}\label{weaklySigmaPrikrysection}

We commence by recalling a few concepts from \cite[\S2]{PartII}.
\begin{definition}\label{gradedposet}
A \emph{graded poset} is a pair $(\mathbb P,\lh)$ such that $\mathbb P=(P,\le)$ is a poset, $\lh:P\rightarrow\omega$ is a surjection, and, for all $p\in P$:
\begin{itemize}
\item For every $q\le p$, $\lh(q)\geq\lh(p)$;
\item There exists $q\le p$ with $\lh(q)=\lh(p)+1$.
\end{itemize}
\end{definition}
\begin{conv} For a graded poset as above, we denote $P_n:=\{p\in P\mid \lh(p)=n\}$
and $\mathbb P_n:=(P_n\cup\{\one\},\le)$. In turn, $\mathbb P_{\ge n}$ and $\mathbb P_{>n}$ are defined analogously.
We also write $P_n^p:=\{ q\in P\mid  q\le p, \lh(q)=\lh(p)+n\}$,
and sometimes write $q\le^n p$ (and say that $q$ is \emph{an $n$-step extension} of $p$) rather than writing $q\in P^p_n$.

A subset $U\s P$ is said to be \emph{$0$-open set} iff, for all $r\in U$, $P^r_0\s U$.
\end{conv}

Now, we define the $(\Sigma,\vec{\mathbb S})$-Prikry class,
a class broader than $\Sigma$-Prikry from \cite[Definition~2.3]{PartII}.

\begin{definition}\label{SigmaPrikry}
Suppose:
\begin{enumerate}[label=(\greek*)]
\item \label{Calpha} $\Sigma=\langle \sigma_n\mid n<\omega\rangle$ is a non-decreasing sequence of regular uncountable cardinals,
converging to some cardinal $\kappa$;
\item \label{Cbeta} $\vec{\mathbb S}=\langle \mathbb S_n\mid n<\omega\rangle$ is a sequence of notions of forcing, $\mathbb S_n=(S_n,\SLE_n)$,
with $|S_n|< \sigma_n$;
\item \label{Cgamma} $\mathbb P=(P,\le)$ is a notion of forcing with a greatest element $\one$;
\item \label{Cdelta} $\mu$ is a cardinal such that $\one\forces_{\mathbb P}\check\mu=\check\kappa^+$;
\item \label{Cepsilon} $\lh:P\rightarrow\omega$ and $c:P\rightarrow \mu$ are functions;\footnote{In some applications $c$ will be a function from $P$ to some canonical structure of size $\mu$, such as $H_\mu$ (assuming $\mu^{<\mu}=\mu$).}
\item \label{Czeta} $\vec{\varpi}=\langle \varpi_n\mid n<\omega\rangle$ is a sequence of functions.
\end{enumerate}

We say that $(\mathbb P,\lh, c,\vec\varpi)$ is \emph{$(\Sigma,\vec{\mathbb S})$-Prikry}  iff all of the following hold:
\begin{enumerate}
\item\label{graded} $(\mathbb P,\lh)$ is a graded poset;
\item\label{c2} For all $n<\omega$, $\mathbb{P}_n:=(P_n\cup\{\one\},\LE)$ contains a dense subposet $\z{\mathbb{P}}_n$ which is countably-closed;
\item\label{c1} For all $p,q\in P$, if $c(p)=c(q)$, then $P_0^p\cap P_0^q$ is non-empty;
\item \label{c5} For all $p\in P$, $n,m<\omega$ and $q\LE^{n+m}p$, the set $\{r\LE ^n p\mid  q\LE^m r\}$ contains a greatest element which we denote by $m(p,q)$.
In the special case $m=0$, we shall write $w(p,q)$ rather than $0(p,q)$;\footnote{Note that $w(p,q)$ is the weakest extension of $p$  above $q$.}
\item\label{csize}  For all $p\in P$,
the set $W(p):=\{w(p,q)\mid q\LE p\}$ has size $<\mu$;
\item\label{itsaprojection} For all $p'\LE p$ in $P$, $q\mapsto w(p,q)$ forms an order-preserving map from $W(p')$ to $W(p)$;
\item\label{c6} Suppose that $U\s P$ is a $0$-open set.
Then, for all $p\in P$ and $n<\omega$, there is $q\LE^0 p$, such that, either $P^{q}_n\cap U=\emptyset$ or $P^{q}_n\s U$;
\item \label{PnprojectstoSn} For all $n<\omega$,  $\varpi_n$ is a nice projection from $\mathbb{P}_{\geq n}$
to $\mathbb S_n$,
such that, for any integer $k\ge n$, $\varpi_n\restriction \mathbb P_k$ is again a nice projection;
\item\label{moreclosedness} For all $n<\omega$, if $\z{\mathbb{P}}_n$ is a witness for Clause~\eqref{c2} then $\z{\mathbb{P}}_n^{\varpi_n}$ is a dense and $\sigma_n$-directed-closed subposet of  $\mathbb{P}_n^{\varpi_n}:=(P_n\cup\{\one\}, \LE^{\varpi_n})$.\footnote{More verbosely, for every $p\in P_n$ there is $q\in \z{P}_n$ such that $q\LE^{\varpi_n} p$ (see Notation~\ref{OrderModuloVarpi}).}

\end{enumerate}
\end{definition}

\begin{conv}\label{ConvSigmaPrikry}
We derive yet another ordering $\le^{\vec\varpi}$ of the set $P$,
letting ${\le^{\vec\varpi}}:=\bigcup_{n<\omega}{\le^{\varpi_n}}$.
Simply put, this means that $q\LE^{\vec\varpi}p$ iff ($p=\one$),
or, ($q\LE^0p$, $\lh(p)=\lh(q)$ and $\varpi_{\lh(p)}(p)=\varpi_{\lh(q)}(q)$).
\end{conv}

\begin{conv}
We say that $(\mathbb{P},\ell,c)$ has the Linked$_0$-property if it witnesses Clause~\eqref{c1} above. Similarly, we will say that $(\mathbb{P},\ell)$ has the Complete Prikry Property (CPP) if it witnesses Clause~\eqref{c6} above.
\end{conv}

Any $\Sigma$-Prikry triple $(\mathbb{P},\ell,c)$  can be regarded as a $(\Sigma,\vec{\mathbb{S}})$-Prikry forcing $(\mathbb{P},\ell,c,\vec{\varpi})$ by letting $\vec{\mathbb{S}}:=\langle(n, \{\one_\mathbb{P}\})\mid n<\omega\rangle$ and $\vec{\varpi}$ 
be the sequence of trivial projections $p\mapsto \one_\mathbb{P}$. {Conversely, any $(\Sigma, \vec{\mathbb{S}})$-Prikry  quadruple $(\mathbb{P},\ell,c,\vec{\varpi})$ with $\vec{\mathbb{S}}$ and $\vec{\varpi}$ as above witnesses that $(\mathbb{P},\ell,c)$ is $\Sigma$-Prikry.}
In particular, all the forcings from \cite[\S3]{PartI} are examples of $(\Sigma,\vec{\mathbb{S}})$-Prikry forcings.
In Section~\ref{SectionGitikForcing}, we will add a new example to this list by showing that Gitik's EPBFC (The long Extender-Based Prikry forcing with Collapses \cite{GitikCollapsin}) falls into the $(\Sigma,\vec{\mathbb{S}})$-Prikry framework.

\medskip

Throughout the rest of the section, assume that $(\mathbb{P},\ell,c,\vec{\varpi})$ is a $(\Sigma,\vec{\mathbb{S}})$-Prikry quadruple.
We shall spell out some basic features of the components of the quadruple,
and work towards proving Lemma~\ref{l14} that explains
how bounded sets of $\kappa$ are added to generic extensions by $\mathbb{P}$.

\begin{lemma}[The $p$-tree]\label{W(p)maxAntichain} Let $p\in P$.
\begin{enumerate}
\item For every $n<\omega$, $W_n(p)$ is a maximal antichain in $\cone{p}$;\label{W(p)maxAntichain1}
\item Every two compatible elements of $W(p)$ are comparable;\label{W(p)maxAntichain2}
\item For any pair $q'\LE q$ in $W(p)$,  $q'\in W(q)$;\label{W(p)maxAntichain4}
\item $c\restriction W(p)$ is injective.\label{W(p)maxAntichain3}
\end{enumerate}
\end{lemma}
\begin{proof} The proof of \cite[Lemma~2.8]{PartI} goes through.
\end{proof}

We commence by introducing the notion of coherent sequence of nice projections, which will be important  in  Section~\ref{killingone}.

\begin{definition}\label{CoherentSystem}
The sequence of nice projections $\vec{\varpi}$ is called  \emph{coherent} iff the two hold:
\begin{enumerate}
\item \label{AdditionalAssumption1} for all $n<\omega$, if $p\in P_{\geq n}$ then $\varpi_{n}``W(p)=\{\varpi_{n}(p)\}$;
\item\label{NiceCoherent} for all $n\leq m<\omega$, $\varpi_m$ factors through $\varpi_n$; i.e., there is a map $\pi_{m,n}\colon \mathbb{S}_m\rightarrow \mathbb{S}_n$ such that $\varpi_n(p)=\pi_{m,n}(\varpi_m(p))$ for all $p\in P_{\geq m}$.
\end{enumerate}
\end{definition}

\begin{lemma}\label{AdditionalAssumption3}
Let $p\in P$. Then for each $q\in W(p)$, $n\leq \ell(q)$ and $t\sle_n \varpi_n(q)$,
$$w(p,q+t)=w(p,q).$$
\end{lemma}
\begin{proof}
 Note that $w(p,q+t)$ and $w(p,q)$ are two compatible conditions in $W(p)$ with the same length. In effect, Lemma~\ref{W(p)maxAntichain}\eqref{W(p)maxAntichain1} yields the desired.
\end{proof}

\begin{lemma}\label{LemmaAdditionalAssumptions}
Assume that   $\vec{\varpi}$ is coherent.

For all $n<\omega$, $p\in P_{\geq n}$ and $t\sle_{n} \varpi_{n}(p)$, the following  hold:
\begin{enumerate}
\item for each $q\in W(p+t)$, $q=w(p,q)+t$;\label{AdditionalAssumption2}
\item  for each $q\in W(p)$, $w(p+t,q+t)=q+t;$ \label{AdditionalAssumption4}
\item  for each $m<\omega$, $W_m(p+t)=\{q+t\mid q\in W_m(p)\}$. \label{AdditionalAssumption5}
\item \label{SumInvariant} $p+t=p+\varpi_{\ell(p)}(p+t)$;
\end{enumerate}
\end{lemma}
\begin{proof}

\eqref{AdditionalAssumption2} Let  $q\in W(p+t)$. By virtue of  Definition~\ref{CoherentSystem}\eqref{AdditionalAssumption1}, we have $\varpi_{n}(q)=\varpi_{n}(p+t)=t$. This, together with $q\leq^0 w(p,q)$, implies that $w(p,q)+t$ is well-defined and also that $q\leq^0 w(p,q)+t$.  On the other hand,  $q\leq^0 w(p,q)+t\leq p+t$, hence $w(p+t,w(p,q)+t)$ and $q$ are two compatible conditions in $W(p+t)$ that have the same length. By Lemma~\ref{W(p)maxAntichain}\eqref{W(p)maxAntichain1} it follows that $q=w(p+t,w(p,q)+t)$, hence $w(p,q)+t\LE^0 q$, as desired.

\eqref{AdditionalAssumption4}
By Definition~\ref{CoherentSystem}\eqref{AdditionalAssumption1}, $q\LE^{\varpi_n} p$, hence $q+t$ is well-defined and so $w(p+t,q+t)$ belongs to $W(p+t)$. 
Combining Clause~\eqref{AdditionalAssumption2} above with \cite[Lemma~2.9]{PartI} we obtain the following chain of equalities:
$$w(p+t,q+t)=w(p,w(p+t,q+t))+t=w(p,q+t)+t.$$
Now, combine Lemma~\ref{AdditionalAssumption3}  with $q\in W(p)$ to infer that $w(p,q+t)=q$. Altogether, this shows that $w(p+t,q+t)=q+t$.

\eqref{AdditionalAssumption5} The left-to-right  inclusion is given by \eqref{AdditionalAssumption2} and the converse  by \eqref{AdditionalAssumption4}.

\eqref{SumInvariant} Note that $p+t\LE p+\varpi_{\ell(p)}(p+t)$. Conversely, by using Clause~\eqref{NiceCoherent} of Definition~\ref{CoherentSystem}  we have that  $\varpi_{n}(p+\varpi_{\ell(p)}(p+t))=\varpi_n(p+t)=t$.
\end{proof}

\begin{prop}\label{extensions} For every condition $p$ in $\mathbb P$ and an ordinal $\alpha<\kappa$,
there exists an extension $p'\LE p$ such that $\sigma_{\lh(p')}>\alpha$.
\end{prop}
\begin{proof} Let $p$ and $\alpha$ be as above. Since $\alpha<\kappa=\sup_{n<\omega}\sigma_n$, we may find some $n<\omega$ such that $\alpha<\sigma_n$.
By Definition~\ref{SigmaPrikry}\eqref{graded}, $(\mathbb P,\lh)$ is a graded poset,
so by possibly iterating the second bullet of Definition~\ref{gradedposet} finitely many times,
we may find an extension $p'\LE p$ such that $\lh(p')\ge n$. As $\Sigma$ is non-decreasing, $p'$ is as desired.
\end{proof}

As in the context of $\Sigma$-Prikry forcings, also here,
the CPP implies the Prikry Property (PP) and the Strong Prikry Property (SPP).

\begin{lemma}\label{Prikrybasic} Let $p\in P$.
\begin{enumerate}
\item Suppose $\varphi$ is a sentence in the language of forcing.
Then there is $p'\LE^0 p$, such that $p'$ decides $\varphi$;
\item\label{SPP}  Suppose $D\s P$ is a $0$-open set which is dense  below $p$.
Then there is $p'\LE^0 p$, and $n<\omega$, such that $P^{p'}_n\s D$.\footnote{Note that if $D$ is moreover open, then $P^q_m\s D$ for all $m\geq n$.}
\end{enumerate}
Moreover, we can let $p'$ above to be a condition from $\z{\mathbb{P}}^{\varpi_{\ell(p)}}_{\ell(p)}\downarrow p$.
\end{lemma}
\begin{proof}
We only give the proof of (1), the proof of (2) is similar.
Fix $\varphi$ and $p$. Put $U_{\varphi}^+:=\{q\in P\mid q\Vdash_\mathbb{P}\varphi\}$ and  $U^-_\varphi:=\{q\in P\mid q\Vdash_\mathbb{P}\neg \varphi\}$. Both of these are $0$-open, so applying  Clause~\eqref{c6} of Definition~\ref{SigmaPrikry} twice, we get following:
\begin{claim}
For all $q\in P$ and $n<\omega$, there is $q'\LE^0 q$, such that either all $r\in P^{q'}_n$ decide $\varphi$ the same way, or no $r\in P^{q'}_n$ decides
$\varphi$.
\end{claim}
Now using the claim construct a $\leq^0$ decreasing sequence $\langle p_n\mid n<\omega\rangle$ below $p$. By using Clause~\eqref{c2} of Definition~\ref{SigmaPrikry} we may additionally assume that these are conditions in $\z{\mathbb{P}}_{\ell(p)}$. 
Letting $p'$ be a $\leq^0$-lower bound for this sequence we obtain $\leq^0$-extension of $p$  deciding $\varphi$. 
\end{proof}

\begin{cor}\label{PrikryPartI} Let $p\in P$ and $s\sle_{\ell(p)} \varpi_{\ell(p)}(p)$.
\begin{enumerate}
\item Suppose $\varphi$ is a sentence in the language of forcing.\label{PrikryPartI1}
Then there is $p'\LE^{\vec\varpi} p$ and $s'\sle_{\ell(p)} s$ such that $p'+s'$ decides $\varphi$;
\item  Suppose $D\s P$ is a $0$-open set which is dense  below $p$.\label{PrikryPartI2}
Then there are $p'\LE^{\vec\varpi} p$, $s'\sle_{\ell(p)} s$  and $n<\omega$ such that $P^{p'+s'}_n\s D$.
\end{enumerate}
Moreover, we can let $p'$ above to be a condition from $\z{\mathbb{P}}^{\varpi_{\ell(p)}}_{\ell(p)}\downarrow p$.
\end{cor}
\begin{proof}
We only show \eqref{PrikryPartI1} as  \eqref{PrikryPartI2} is similar. 
By Lemma \ref{Prikrybasic}, let $q\leq^0 p+s$ deciding $\varphi$. By Definition~\ref{SigmaPrikry}\eqref{PnprojectstoSn} the map $\varpi_n$ is a nice projection, hence there is  $p'\LE^{\vec\varpi} p$ and $s'\sle_{\ell(p)} s$ such that $p'+s'=q$ (Definition~\ref{niceprojection}\eqref{theprojection}). The moreover part follows from density of $\z{\mathbb{P}}^{\varpi_{\ell(p)}}_{\ell(p)}$  in ${\mathbb{P}}^{\varpi_{\ell(p)}}_{\ell(p)}$ (Definition~\ref{SigmaPrikry}\eqref{moreclosedness}).
\end{proof}

Working a bit more, we can obtain the following:

\begin{lemma}\label{Prikry}
Let $p\in P$. Set  $\ell:=\ell(p)$ and $s:=\varpi_n(p)$.
\begin{enumerate}
\item\label{Prikry1} Suppose $\varphi$ is a sentence in the language of forcing.
Then there  is $q\LE^{\vec\varpi}  p$ such that
$D_{\varphi,q}:=\{t\SLE_\ell s\mid(q+t\forces_{\mathbb P}\varphi)\text{ or }(q+t\forces_{\mathbb P}\neg\varphi)\}$ is dense in $\cones{\ell} s$;
\item\label{Prikry3} Suppose $D\s P$ is a $0$-open set.
Then there is $q\LE^{\vec\varpi}  p$ such that
$U_{D,q}:=\{t\SLE_\ell s\mid\forall m<\omega~(P^{q+t}_m\s D\;\text{or}\;P^{q+t}_m\cap D=\emptyset)\}$ is dense in $\cones{\ell} s$.
\item\label{Prikry2} Suppose $D\s P$ is a $0$-open set which is dense  below $p$.
Then there is $q\LE^{\vec\varpi}  p$ such that
$U_{D,q}:=\{t\SLE_\ell s\mid\exists m<\omega~P^{q+t}_m\s D\}$ is dense in $\cones{\ell} s$.
\end{enumerate}
Moreover, $q$ above belongs to $\z{\mathbb{P}}^{\varpi_\ell}_\ell\downarrow p$.
\end{lemma}
\begin{proof}
(1) By Definition~\ref{SigmaPrikry}\ref{Cbeta}, let us fix some cardinal $\theta<\sigma_\ell$ along with an injective enumeration $\langle s_\alpha\mid\alpha<\theta\rangle$ of the conditions in $\cones{\ell} s$,
such that $s_0=s$.
We will construct by recursion two sequences of conditions $\vec p=\langle p^\alpha\mid \alpha<\theta\rangle$ and $\vec s=\langle s^\alpha\mid \alpha<\theta\rangle$ for which all of the following hold:
\begin{itemize}
\item[(a)] $\vec p$ is a $\LE^{\vec\varpi}$-decreasing sequence of conditions in $\z{\mathbb{P}}^{\varpi_\ell}_\ell$ below $p$;
\item[(b)] $\vec s$ is a sequence of conditions  below $s$;
\item[(c)] for each $\alpha<\theta$, $s^{\alpha}\sle_n s_\alpha$ and $p^{\alpha}+s^{\alpha}\parallel_\mathbb{P}\varphi$.
\end{itemize}

To see that this will do,
assume for a moment that there are sequences $\vec p$ and $\vec s$ as above.
Since $\theta<\sigma_\ell$,  we may find a $\LE^{\vec\varpi}$-lower bound $q$ for $\vec p$ in $\z{\mathbb{P}}^{\varpi_\ell}_\ell$.
In particular, $q\LE^{\vec\varpi} p$.
We claim that $D_{\varphi,q}$ is dense in $\cones{\ell}s$.
To this end, let $s'\sle_\ell s$ be arbitrary.
Find $\alpha<\theta$ such that $s'=s_\alpha$.
By the hypothesis, $s^{\alpha}\sle_\ell s_\alpha$ and $p^{\alpha}+s^\alpha$ decides $\varphi$, hence $q+s^{\alpha}$ also decides it.
In particular, $s^{\alpha}$ is an extension of $s'$ belonging to $D_{\varphi,q}$.

\begin{claim}
There are sequences  $\vec p$ and $\vec s$  as above.
\end{claim}

\begin{proof}We construct the two sequences by recursion on $\alpha<\theta$.
For the base case, appeal to Corollary~\ref{PrikryPartI}\eqref{PrikryPartI1} with $p$ and $s$,
and retrieve $p^0\LE^{\vec\varpi} p$ and $s^0\sle_{n}s$ such that $p_0\in \z{P}^{\varpi_{\ell}}_\ell$ and $p^0+s^0$ indeed decides $\varphi$.

$\br $ Assume  $\alpha=\beta+1$ and that $\langle p^\gamma\mid \gamma\leq \beta\rangle$ and $\langle s^\gamma\mid \gamma\leq \beta\rangle$ have been already defined. 
Since $s_\alpha\sle_\ell s=\varpi_\ell(p^\beta)$,
it follows that $p^\beta+s_\alpha$ is a legitimate condition in $P_\ell$.
Appealing to Corollary~\ref{PrikryPartI}\eqref{PrikryPartI1} with $p^\beta$ and $s_\alpha$, let  $p^\alpha\LE^{\vec\varpi} p^\beta$ and $s^{\alpha}\sle_{\ell}s_\alpha$ be such that $p^\alpha\in \z{P}^{\varpi_\ell}_\ell$ and  $p^\alpha+s^{\alpha}$ decides $\varphi$.

$\br$ Assume $\alpha\in\acc(\theta)$ and that the sequences $\langle p^\beta\mid \beta< \alpha\rangle$ and $\langle s^\beta\mid \beta<\alpha\rangle$ have already been defined.
Appealing to Definition~\ref{SigmaPrikry}\eqref{moreclosedness}, let $p^*$  be a $\LE^{\vec\varpi}$-lower bound for $\langle p^\beta\mid \beta< \alpha\rangle$.
Finally, obtain $p^\alpha\in D$ and $s^\alpha$ by appealing to Corollary~\ref{PrikryPartI}\eqref{PrikryPartI1} with respect to $p^*$ and $s_\alpha$.
\end{proof}

This completes the proof of Clause~(1).
The proof of Clauses~\eqref{Prikry3} and \eqref{Prikry2} is similar by amending suitably Clause~(c) above. For instance, for Clause~\eqref{Prikry3} we require the following in Clause~(c): for each $\alpha<\theta$ and $n<\omega$, $s^\alpha\sle_n s_\alpha$ and either $P^{p^\alpha+s^\alpha}_n\s D$ or $P^{p^\alpha+s^\alpha}_n\cap  D=\emptyset$. For the verification of this new requirement we combine Clauses~\eqref{c2}, \eqref{c6} and \eqref{PnprojectstoSn} of Definition~\ref{SigmaPrikry} with Definition~\ref{niceprojection}\eqref{theprojection}. Similarly, to prove Clause~\eqref{Prikry2} of the lemma one uses Clause~\eqref{PrikryPartI2} of  Corollary~\ref{PrikryPartI}.
\end{proof}

We now arrive at the main result of the section:
\begin{lemma}[Analysis of bounded sets]\label{l14} \hfill
\begin{enumerate}
\item\label{l14(1)}
If $p\in P$ forces that $\dot{a}$ is a $\mathbb P$-name for a bounded subset $a$ of $\sigma_{\lh(p)}$,
then $a$ is added by $\mathbb{S}_{\lh(p)}$.
In particular,
if $\dot{a}$ is a $\mathbb P$-name for a bounded subset $a$ of $\kappa$, then, for any large enough $n<\omega$, $a$ is added by $\mathbb S_n$;
\item\label{l14(3)} $\mathbb{P}$ preserves  $\kappa$. Moreover, if $\kappa$ is a strong limit, it remains so;
\item\label{l14(2)} For every regular cardinal $\nu\geq\kappa$, if there exists $p\in P$ for which $p\forces_{\mathbb P}\cf(\nu)<\kappa$,
then there exists $q\LE^{\vec\varpi} p$ with $|W(q)|\geq\nu$;\footnote{For future reference, we point out that this fact relies only on clauses  (\ref{graded}), (\ref{csize}), (\ref{c6}), \eqref{PnprojectstoSn} and (\ref{moreclosedness}) of Definition~\ref{SigmaPrikry}.}
\item\label{l14(4)} Suppose $\one\forces_{\mathbb P}``\kappa\text{ is singular}"$. Then $\mu=\kappa^+$ if and only if, for all $p\in P$, $|W(p)|\leq\kappa$.

\end{enumerate}
\end{lemma}
\begin{proof} (1) The ``in particular'' part follows from the first part
together with Proposition~\ref{extensions}.
Thus, let us suppose that $p$ is a given condition forcing that $\dot{a}$ is a name for a subset $a$ of some cardinal $\theta<\sigma_{\lh(p)}$.

For each $\alpha<\theta$, denote the sentence $``\check\alpha\in\dot{a}"$ by $\varphi_\alpha$. Set $n:=\lh(p)$ and $s:=\varpi_n(p)$.
Combining Definition~\ref{SigmaPrikry}\eqref{moreclosedness} with Lemma~\ref{Prikry}\eqref{Prikry1},
we may recursively obtain a $\LE^{\varpi_n}$-decreasing sequence of conditions
$\vec p=\langle p^\alpha\mid \alpha<\theta\rangle$ with a lower bound, such that, for each $\alpha<\theta$,
$p^\alpha\LE^{\varpi_n} p$ and
$D_{\varphi_\alpha,p^\alpha}$ is dense in $\cones{n}{s}$.
Then let $q\in P_n$ be $\LE^{\varpi_n}$-below all elements of $\vec p$.
It follows that for every $\alpha<\theta$,
$$D_{\varphi_\alpha,q}=\{t\SLE_n s\mid (q+t\forces_{\mathbb P} \varphi_\alpha)\text{ or }(q+t\forces_{\mathbb P}\neg\varphi_\alpha)\}$$
is dense in $\cones{n}{s}$.

Now, let $G$ be a $\mathbb{P}$-generic filter with $p\in G$.
Let $H_n$ be the $\mathbb{S}_n$-generic filter induced by ${\varpi}_n$ from $G$,
and work in $V[H_n]$.
It follows that, for every $\alpha<\theta$,
for some $t\in H_n$, either $(q+t\forces_{\mathbb{P}} {\check\alpha}\in\dot{a})$ or $(q+t\forces_{\mathbb{P}} \check\alpha\notin\dot{a})$.
Set $$b:=\{\alpha<\theta\mid\exists t\in H_n[q+t\forces_{\mathbb{P}} \check\alpha\in\dot{a}]\}.$$

As $q\le^{\vec\varpi}p$,
we infer that  $\varpi_n(q)=\varpi_n(p)=s\in H_n$,
so that $q\in P/H_n$. 

\begin{claim} $q\forces_{\mathbb{P}/H_n} {b}=\dot{a}_{H_n}$.
\end{claim}
\begin{proof}
Clearly, $q\forces_{\mathbb{P}/H_n} {b}\s \dot{a}_{H_n}$.
For the converse, let $\alpha<\theta$ and $r\LE_{\mathbb{P}/H_n} q$ be such that $r\forces_{\mathbb{P}/H_n} \check\alpha\in\dot{a}_{H_n}$.
By the very Definition~\ref{DefinitionQuotient}, there is $t_0\in H_n$ with $t_0\sle_n {\varpi}_n(r)$ such that $r+t_0\LE q$. By extending $t$ if necessary, we may moreover assume that $r+t_0\forces_\mathbb{P} \check\alpha\in\dot{a}$. Set $q_0:=r+t_0$.

By the choice of $q$, there is $t_1\in H_n$ such that $q+t_1\parallel_{\mathbb{P}} \check\alpha\in\dot{a}$. Set $q_{1}:=q+t_1$. Let $t\in H_n$ be such that $t\sle_n t_0,t_1$.
Recalling Definition~\ref{SigmaPrikry}\eqref{moreclosedness},
${\varpi}_n$ is nice, so $t\sle_n {\varpi}_n(q_0), {\varpi}_n(q_1)$.
By  Definition~\ref{niceprojection}\eqref{theprojection}, 
$q_0+t$ witnesses the compatibility of $q_0$ and $q_1$, hence $q+t_1 \forces_{\mathbb{P}}\check\alpha\in\dot{a}$, and thus $\alpha\in b$.
\end{proof}

Altogether, $\dot{a}_G\in V[H_n]$.

\medskip

(2) If $\kappa$ were to be collapsed, then, by Clause~(1), it would have been collapsed by $\mathbb{S}_n$ for some $n<\omega$.
However, $\mathbb{S}_n$ is a notion of forcing of size $<\sigma_n\leq \kappa$.

Next, suppose towards a contradiction that $\kappa$ is strong limit cardinal,
and yet, for some $\mathbb{P}$-generic filter $G$,
for some $\theta<\kappa$, $V[G]\models2^\theta\geq \kappa$.
For each $n<\omega$, let $H_n$ be the $\mathbb{S}_n$-generic filter induced by ${\varpi}_n$ from $G$.
Using Clause~(1), for every $a\in\mathcal P^{V[G]}(\theta)$, we fix $n_a<\omega$ such that $a\in V[H_{n_a}]$.

$\br$ If $\kappa$ is regular, then there must exist some $n<\omega$
for which $|\{a\in\mathcal P^{V[G]}(\theta)\mid n_a=n\}|\ge\kappa$.
However $\mathbb S_n$ is a notion of forcing of some size $\lambda<\kappa$,
and so by counting nice names, we see it cannot add more than $\theta^\lambda$ many subsets to $\theta$, contradicting the fact that $\kappa$ is strong limit.

$\br$ If $\kappa$ is not regular, then $\Sigma$ is not eventually constant,
and $\cf(\kappa)=\omega$, so that, by K\"onig's lemma,
$V[G]\models2^\theta\geq \kappa^+$.
It follows that exists some $n<\omega$
for which $|\{a\in\mathcal P^{V[G]}(\theta)\mid n_a=n\}|>\kappa$,
leading to the same contradiction.

(3) Suppose $\theta,\nu$ are regular cardinals with $\theta<\kappa\leq\nu$, $\dot f$ is a $\mathbb P$-name for a function from $\theta$ to $\nu$,
and $p\in P$ is a condition forcing that the image of $\dot f$ is cofinal in $\nu$.
Denote $n:=\lh(p)$ and $s:=\varpi_n(p)$.
By Proposition~\ref{extensions},  we may extend $p$ and assume that $\sigma_{n}>\theta$. 

For all $\alpha<\theta$, set $D_\alpha:=\{r\LE p\mid \exists\beta<\nu,\, r\forces_{\mathbb{P}} \dot{f}(\check{\alpha})=\check{\beta}\}$.
As $D_\alpha$ is $0$-open and  dense  below $p$, by combining Lemma~\ref{Prikry}\eqref{Prikry2} with the $\sigma_n$-directed closure of $\z{\mathbb{P}}^{\varpi_n}_n$ (see Definition~\ref{SigmaPrikry}\eqref{moreclosedness}),  
we may recursively define a $\LE^{\vec\varpi}$-decreasing sequence of conditions $\langle q^\alpha\mid \alpha\le\theta\rangle$  below $p$ such that,
for every $\alpha<\theta$, $U_{D_\alpha,q^\alpha}$ is dense in $\cones{n}s$.
Set $q:=q^\theta$,
and note that
$$U_{D_\alpha,q}:=\{t\sle_n s\mid \exists m<\omega[P^{q+t}_n\s D_\alpha]\}$$ is dense in $\cones{n}s$ for all $\alpha<\theta$. In particular, the above sets are non-empty. For each $\alpha<\theta$, let us fix $t_\alpha\in U_{D_\alpha,q}$ and $m_\alpha<\omega$ witnessing this. We now show that $|W(q)|\geq \nu$.
Let $A_\alpha:=\{\beta<\nu\mid \exists r\in P^{q+t_\alpha}_{m_\alpha} [r\forces_{\mathbb P}\dot{f}(\check{\alpha})=\check{\beta}]\}$.
By Lemma~\ref{W(p)maxAntichain}\eqref{W(p)maxAntichain1}, we have $$A_\alpha=\{\beta<\nu\mid  \exists r\in W_{m_\alpha}(q+t_\alpha)\,[r\forces_{\mathbb P}\dot{f}(\check{\alpha})=\check{\beta}]\}.$$
Let $A:=\bigcup_{\alpha<\theta}A_\alpha$.
Then, $$|A|\leq\sum_{m<\omega, t\sle_n s}|W_{m}(q+t)|\leq \max\{\aleph_0,|S_n|\} \cdot |W(q)|.\footnote{Observe that, for each $t\sle_n s$, $|W(q+t)|\leq|W(q)|$.}$$
Also, by clauses $(\alpha)$ and $(\beta)$ of  Definition~\ref{SigmaPrikry} and our assumption on $\nu$, $\max\{\aleph_0,|S_n|\}<\sigma_n<\nu$. It follows that if $|W(q)|<\nu$, then $|A|<\nu$, and so $\sup(A)<\nu$. Thus, $q$ forces that the range of $\dot f$ is bounded below $\nu$, which leads us to a contradiction. 
Therefore, $|W(q)|\geq\nu$, as desired.

(4) The left-to-right implication is obvious using  Definition~\ref{SigmaPrikry}\eqref{csize}. Next, suppose that, for all $p\in P$, $|W(p)|\leq\kappa$.
Towards a contradiction, suppose that there exist $p\in P$ forcing that $\kappa^+$ is collapsed.
Denote $\nu:=\kappa^+$.
As by assumption $\one\forces_{\mathbb P}``\kappa\text{ is singular}"$, this means that $p\forces_{\mathbb P}\cf(\nu)<\kappa$,
contradicting Clause~\eqref{l14(2)} of this lemma.
\end{proof}

We end this section recalling the concept of \emph{property $\mathcal{D}$}. This notion was introduced in \cite[\S2]{PartII} and usually captures how various forcings satisfy the Complete Prirky Property (i.e., Clause~\eqref{c6} of Definition~\ref{SigmaPrikry}):

\begin{definition} We say that $\vec r=\langle r_\xi \mid \xi<\chi\rangle$ is a \emph{good enumeration} of a set $A$ iff
$\vec r$ is injective, 
$\chi$ is a cardinal,
and $\{ r_\xi\mid \xi<\chi\}=A$.
\end{definition}

\begin{definition}[Diagonalizability game]\label{DiagonalizabilityGame}\label{Diagonalizability}
Given $p\in P$, $n<\omega$, and a good enumeration
$\vec r=\langle r_\xi\mid \xi<\chi\rangle$ of $W_n(p)$,
we say that $\vec q=\langle q_\xi\mid \xi<\chi\rangle$ is \emph{diagonalizable} (with respect to $\vec{r}$)  iff the two hold:
\begin{enumerate}[label=(\alph*)]
\item $q_\xi\le^0 r_\xi$ for every $\xi<\chi$;
\item there is $p'\le^0 p$ such that for every $q'\in W_n(p')$, $q'\le^0 q_\xi$, where $\xi$ is the unique index to satisfy $r_\xi=w(p, q')$.
\end{enumerate}
Besides, if $D$ is a dense subset of $\mathbb{P}_{\ell_\mathbb{P}(p)+n}$, $\Game_\mathbb{P}(p,\vec r,D)$ is a game of length $\chi$ between two players $\pI$ and $\pII$, defined as follows:
\begin{itemize}
\item At stage $\xi<\chi$, $\pI$ plays a condition $p_\xi\leq^0 p$ compatible with $r_\xi$,
and then $\pII$ plays $q_\xi\in D$ such that $q_\xi\leq p_\xi$ and $q_\xi\leq^0 r_\xi$;
\item $\pI$ wins the game iff the resulting sequence $\vec q=\langle q_\xi\mid \xi<\chi\rangle$ is diagonalizable.
\end{itemize}
In the special case that $D$ is all of $\mathbb{P}_{\ell_\mathbb{P}(p)+n}$, we omit it, writing $\Game_\mathbb{P}(p,\vec r)$.
\end{definition}

\begin{definition}[Property $\mathcal{D}$]\label{propertyD}
We say that a graded poset $(\mathbb{P},\ell_\mathbb{P})$ has \emph{property $\mathcal{D}$}
iff for any $p\in P$, $n<\omega$ and any good enumeration  $\vec r=\langle r_\xi\mid \xi<\chi\rangle$ of $W_n(p)$, $\pI$ has a winning strategy for the game
$\Game_\mathbb{P}(p,\vec r)$.
\end{definition}
\begin{conv}
In a mild abuse of terminology, we shall say that $(\mathbb{P},\ell,c,\vec{\varpi})$ has property $\mathcal{D}$ whenever the pair $(\mathbb{P},\ell)$ has property $\mathcal{D}$.
\end{conv}

\section{Extender Based Prikry Forcing with collapses}\label{SectionGitikForcing}

In this section we present Gitik's notion of forcing from \cite{GitikCollapsin},
and analyze its properties.
Gitik came up with this notion of forcing in September 2019, during the week of the \emph{15th International Workshop on Set Theory in Luminy},
after being asked by the second author whether it is possible to interleave collapses in the Extender Based Prikry Forcing (EBPF) with long extenders \cite[\S3]{Git-Mag}. The following theorem summarizes the main properties of the generic extensions by Gitik's forcing $\mathbb P$:

\begin{theorem}[Gitik]\label{MotisModel}
All of the following hold in $V^{\mathbb{P}}$:
\begin{enumerate}
\item All cardinals $\geq\kappa$ are preserved;
\item $\kappa=\aleph_\omega$, $\mu=\aleph_{\omega+1}$  and $\lambda= \aleph_{\omega+2}$;
\item $\aleph_\omega$ is a strong limit cardinal;
\item $\gch_{<\aleph_\omega}$, provided that $V\models\gch_{<\kappa}$;
\item $2^{\aleph_{\omega}}=\aleph_{\omega+2}$, hence the $\sch_{\aleph_\omega}$ fails.
\end{enumerate}
\end{theorem}
For people familiar with \cite{GitikCollapsin}, some of the proofs in this section can be skipped. Yet, since this forcing notion is fairly new, we do include some proofs.  Most notably, for us it is important to verify the existence of various nice projections and reflections properties in Corollary~\ref{Motisuitableforreflection} and Lemma~\ref{keysubclaim} below. In addition, unlike the exposition of this forcing from \cite{GitikCollapsin}, the exposition here shall not assume the $\gch$.

\begin{setup} Throughout this section our setup  will be as follows: 
\begin{itemize}
\item $\vec{\kappa}=\langle \kappa_n\mid n<\omega\rangle$ is a strictly increasing sequence  of cardinals;
\item $\kappa_{-1}:=\aleph_0$, $\kappa:=\sup_{n<\omega}\kappa_n$, $\mu:=\kappa^+$ and $\lambda:=\mu^+$;
\item $\mu^{<\mu}=\mu$ and $\lambda^{<\lambda}=\lambda$;\label{setupGitik}
\item for each $n<\omega$, $\kappa_n$ is $(\lambda+1)$-strong;
\item $\Sigma:=\langle\sigma_n\mid n<\omega\rangle$, where, for each $n<\omega$, $\sigma_n:=(\kappa_{n-1})^+$;\footnote{In particular, $\sigma_0=\aleph_1$.}
\end{itemize}
\end{setup}

In particular, we are assuming that, for each $n<\omega$, there is a  $(\kappa_n,\lambda+1)$-extender $E_n$ whose associated  embedding $j_n: V\rightarrow M_n$ is such that $M_n$ is a transitive class, $^{\kappa_n}M_n\s M_n$, $V_{\lambda+1}\subseteq M_n$ and $j_n(\kappa_n)>\lambda$.

For each $n<\omega$, and each $\alpha<\lambda$, set $$E_{n,\alpha}:=\{X\subseteq \kappa_n\mid \alpha\in j_n(X)\}.$$
Note that $E_{n,\alpha}$ is a non-principal $\kappa_n$-complete ultrafilter over $\kappa_n$, provided that $\alpha\ge\kappa_n$.
Moreover, in the particular case of $\alpha=\kappa_n$, $E_{n,\kappa_n}$ is also normal. For ordinals $\alpha<\kappa_n$ the measures $E_{n,\alpha}$ are principal so the only reason to consider them is for a more neat presentation.

For each $n<\omega$, we shall consider an ordering $\leq_{E_n}$ over $\lambda$, as follows:
\begin{definition}
For each $n<\omega$, set $$\leq_{E_n}:=\{(\beta,\alpha)\in \lambda\times\lambda\mid \beta\leq \alpha, \,\wedge\, \exists f\in{}^{\kappa_n}\kappa_n\; j_n(f)(\alpha)=\beta\}.$$
\end{definition}
It is routine to check that $\leq_{E_n}$ is reflexive, transitive and antisymmetric, hence $(\lambda, \leq_{E_n})$ is a partial order. In case $\beta \leq_{E_n} \alpha$, we shall fix in advance a witnessing map $\pi_{\alpha, \beta}:\kappa_n \to \kappa_n$.
Also, in the special case where $\alpha=\beta$, by convention, $\pi_{\alpha,\alpha}$ is the identity map $\id$.
Observe that $\leq_{E_n}\restriction(\kappa_n\times \kappa_n)$ is exactly the $\in$-order over $\kappa_n$ so that when we refer to $\leq_{E_n}$ we will really be speaking about the restriction of this order to $\lambda\setminus \kappa_n$. The most notable  property of the poset $(\lambda,\leq_{E_n})$ is that it is $\kappa_n$-directed: that is, for every $a\in[\lambda]^{<\kappa_n}$ there is $\alpha<\lambda$ such that $\beta\leq_{E_n}\alpha$ for all $\beta\in a$. This and other nice features of $(\lambda,\leq_{E_n})$ are proved at the beginning of \cite[\S2]{Gitik-handbook} under the $\gch$.  
A proof without the $\gch$ can be found in \cite[\S10.2]{Pov}.

\begin{remark}\label{RemarkDirectness}
For future reference, it is worth mentioning that 
all the relevant properties of $(\lambda,\leq_{E_n})$ reflect down to $(\mu,\leq_{E_n}\upharpoonright \mu\times\mu)$. In particular, it is true that  every $a\in[\lambda]^{<\kappa_n}$ may be enlarged to an $a^*$ such that $\kappa_n,\mu\in a^*$ and $a^*\cap \mu$ contains a $\leq_{E_n}$-greatest. For details, see \cite[\S2]{Gitik-handbook}.
\end{remark}

\subsection{The forcing}\label{SectionTheforcing}
 Before giving the definition of Gitik's forcing we shall first introduce the basic building block modules $\mathbb{Q}_{n0}$ and $\mathbb{Q}_{n1}$. 
 To that effect, for each $n<\omega$, let us fix a map $s_n\colon \kappa_n\rightarrow\kappa_n$ representing $\mu$ using the \emph{normal generator}, $\kappa_n$. Specifically, for each $n<\omega$, $j_n(s_n)(\kappa_n)=\mu$.

\begin{definition}\label{ModulesforGitiks}
For each $n<\omega$, define $\mathbb{Q}_{n1}$, $\mathbb{Q}_{n0}$ and $\mathbb{Q}_n$ as follows:
\begin{itemize}
\item[$(0)_n$] $\mathbb{Q}_{n0}:=(Q_{n0},\LE_{n0})$ is the set of $p:=(a^p, A^p, f^p, F^{0p},F^{1p},  F^{2p} )$, where:
\begin{enumerate}
\item $(a^p, A^p, f^p)$ is in the ${n0}$-module $Q^*_{n0}$ from the Extender Based Prikry Forcing (EBPF) as defined in \cite[Definition 2.6]{Gitik-handbook}. Moreover, we require that $\kappa_n,\mu\in a^p$ and that $a^p\cap \mu$ contains a  $\leq_{E_n}$-greatest element denoted by $\mc(a^p\cap \mu)$;\footnote{Recall that $(a^p,A^p,f^p)\in Q^*_{n0}$ in particular implies that $a^p$ contains a $\LE_{E_n}$-greatest element, which is typically denoted by $\mc(a^p)$. Note that since $\mu\in a^p$ then $\mc(a^p)$ is always strictly $\leq_{E_n}$-larger than $\mc(a^p\cap \mu)$.}
\item for $i<3$, $\dom(F^{ip})=\pi_{\mc(a^p), \mc(a^p\cap\mu)}[A^p]$, and for  $\nu\in\dom(F^{ip})$, setting $\nu_0:=\pi_{\mc(a^p\cap\mu),\kappa_n}(\nu)$, we have:
\begin{enumerate}
\item $F^{0p}(\nu)\in \Col(\sigma_n, {<}\nu_0)$; 
\item $F^{1p}(\nu)\in \Col(\nu_0, s_n(\nu_0))$;
\item $F^{2p}(\nu)\in \Col(s_n(\nu_0)^{++}, \<\kappa_n).$
\end{enumerate}
\end{enumerate}
The ordering $\LE_{n0}$ is defined as follows:
$q\LE_{{n0}}p$ iff $(a^q, A^q, f^q)\LE_{\mathbb{Q}^*_{n0}}(a^p, A^p, f^p)$ as in \cite[Definition 2.7]{Gitik-handbook}, and for each $\nu\in\dom(F^{iq})$, $F^{iq}(\nu)\supseteq F^{ip}(\nu')$, where $\nu':=\pi_{\mc(a^q\cap\mu),\mc(a^p\cap\mu)}(\nu)$.

\item[$(1)_n$]  $\mathbb{Q}_{n1}:=(Q_{n1},\LE_{n1})$ is the set of  $p:=(f^p, \rho^p,h^{0p},h^{1p}, h^{2p} )$, where:
\begin{enumerate}
\item $f^p$ is a function from some $x\in[\lambda]^{\leq \kappa}$ to $\kappa_n$;
\item $\rho^p<\kappa_n$ inaccessible;
\item $h^{0p}\in \Col(\sigma_n, {<}\rho^p)$;
\item $h^{1p}\in \Col(\rho^p, s_n(\rho^p))$;
\item $h^{2p}\in \Col(s_n(\rho^p)^{++}, {<}\kappa_n)$.
\end{enumerate}
The ordering $\LE_{n1}$ is defined as follows:  $q\LE_{n1} p$ iff
$f^q\supseteq f^p$, $\rho^p=\rho^q$, and for $i<3$, $h^{iq}\supseteq h^{ip}$.

\item[$(2)_n$] Set $\mathbb{Q}_n:=(Q_{n0}\cup Q_{n1},\LE_{n})$ where the ordering $\LE_{n}$ is defined as follows: for each $p,q\in Q_n$, $q\LE_{n} p$ iff
\begin{enumerate}
\item either $p,q \in Q_{ni}$, some $i \in  \{0,1\}$,  and $q \LE_{ni} p$, or
\item $q \in Q_{n1}$, $p \in Q_{n0}$ and, for some $\nu\in A^p$, $q\LE_{{n1}} p {}^\curvearrowright\langle\nu\rangle $, where
$$p{}^\curvearrowright\langle\nu\rangle:=(f^p\cup\{\langle\beta,\pi_{\mc(a^p),\beta}(\nu)\rangle\mid \beta\in a^p \rangle\},\bar{\nu}_0, F^{0p}(\bar{\nu}), F^{1p}(\bar{\nu}), F^{2p}(\bar{\nu})),$$
\end{enumerate}
and $\bar{\nu}=\pi_{\mc(a^p),\mc(a^p\cap\mu)}(\nu)$.
\end{itemize}
\end{definition}

\begin{remark}\label{RemarkMeasureOneSets} For each $n<\omega$,  
$$\{\rho<\kappa_n\mid (\kappa_{n-1})^+<\rho<s_n(\rho)<\kappa_n\ \&\ \rho\text{ inaccessible}\}\in E_{n,\kappa_n}.$$
Similarly, for $a\in [\lambda]^{<\kappa_n}$ as in $(0)_n(1)$ above  and $A\in E_{n,\mc(a)}$, $$(\star)\;\;\{\rho\in \pi_{\mc(a),\mc(a\cap \mu)}``A\mid |\{\nu\in A^p_n\mid \bar{\nu}_0=\rho_0\}|\leq s_n(\rho_0)^{+}\}\in E_{n,\mc(a\cap \mu)}.$$

In what follows we assume that the above is always the case for all $\rho<\kappa_n$ that we ever consider. Similarly, we may also assume that $s_n(\rho)$ is regular (actually the successor of a singular) and that $s_n(\rho^p)^{<\rho^p}=s_n(\rho^p)$.

 The reason we consider conditions witnessing Clause~$(\star)$ above is related with the verification of property $\mathcal{D}$ and $\rm{CPP}$ (cf.~Lemmas \ref{propDcollapses} and \ref{C8forMotis}). Essentially, when we describe the moves of $\pI$ and $\pII$ we would like to be able to take lower bounds of the top-most collapsing maps appearing in conditions played by $\pII$. Namely, we would like to take lower bounds of the $h^{2q_\xi}$'s.  Assuming $(\star)$ we will have that the number of  maps  that need to be amalgamated is at most $s_n(\nu_0)^+$, hence less than the completeness of the top-most L\'{e}vy collapse $\Col(s_n(\nu_0)^{++}, {<}\kappa_{n+1})$.
\end{remark}

\begin{remark}
The reason Gitik makes $F^{ip}_n$ dependent on the partial extender $E_n\restriction\mu$ rather than on the full extender $E_n$  is related with the verification of the chain condition (see \cite[Lemma~2.6]{GitikCollapsin}). 
Indeed, in that way the triple $\langle F^{0p}_n, F^{1p}_n, F^{2p}_n\rangle$  will represent three (partial) collapsing in the ultrapower by the measure $E_{n,\mc(a^p_n\cap \mu)}$. 
This will gua\-rantee that the map $c$ given in Definition~\ref{cfunctionforMoti} below will have $H_\mu$ as a range (see Remark~\ref{Clarifyingc}).
\end{remark}

Having all necessary building blocks, we can now define the poset $\mathbb{P}$.

\begin{definition}\label{EBPFC}
The Extender Based Prikry Forcing with collapses (EBPFC) is the poset $\mathbb{P}:=(P,\leq)$ defined by the following clauses:
\begin{itemize}
\item Conditions in $P$ are sequences $p=\langle p_n\mid n<\omega\rangle\in \prod_{n<\omega} Q_n$.
\item For all $p\in P$,
\begin{itemize}
\item There is $n<\omega$ such that $p_n\in Q_{n0}$;
\item For every $n<\omega$, if  $p_n\in Q_{n0}$ then $p_{m}\in Q_{m0}$ and $a^{p_n}\subseteq a^{p_{m}}$, for every $m\geq n$.
\end{itemize}
\item For all $p,q\in P$, $p\LE q$ iff  $p_n\LE_n q_n$, for every $n<\omega$.
\end{itemize}
\end{definition}
\begin{definition}\label{lofEBPFC}
$\ell:P\rightarrow\omega$ is defined by letting for all $p=\langle p_n\mid n<\omega\rangle$,
$$\ell(p):=\min\{n<\omega\mid p_n\in Q_{n0}\}.$$
\end{definition}

\begin{notation}\label{NotationMotis}
Given $p\in P$, $p=\langle p_n\mid n<\omega\rangle$, we will typically write
$p_n=(f^p_n, \rho^p_n, h_n^{0p}, h_n^{1p}, h_n^{2p})$ for $n<\ell(p)$, and
$p_n=(a^p_n, A^p_n, f^p_n, F_n^{0p}, F_n^{1p}, F_n^{2p})$ for $n\geq \ell(p)$. Also, for each $n\geq \ell(p)$, we shall denote $\alpha_{p_n}:=\mc(a^p_n \cap \mu)$.
\end{notation}

We already have $(\mathbb P,\lh)$ and
we will eventually check that $\one\forces_{\mathbb P}\check\mu=\check\kappa^+$ (Corollary~\ref{MuInEBPFC}).
Next, we introduce sequences $\vec{\mathbb{S}}=\langle \mathbb{S}_n\mid n<\omega\rangle$ and $\vec{\varpi}=\langle\varpi_n\mid n<\omega\rangle$,
and a map $c:P\rightarrow H_\mu$ such that $(\mathbb{P},\ell,c,\vec{\varpi})$ will be a $(\Sigma,\mathbb{\vec{S}})$-Prikry forcing having property $\mathcal{D}$.

As $\mu^\kappa=\mu$ and $2^\mu=\lambda$, 
using the Engelking-Kar\l owicz theorem, 
we  fix a sequence of functions $\langle e^i\mid i<\mu\rangle$  from $\lambda$ to $\mu$
such  that, for all $x\in[\lambda]^{\kappa}$ and every function $e:x\rightarrow\mu$, there exists $i<\mu$ with $e\s e^i$.
\begin{definition}\label{cfunctionforMoti}
For every condition $p=\langle p_n\mid n<\omega\rangle$ in $\mathbb P$, define a sequence of indices $\langle i(p_n)\mid n<\omega\rangle$ as follows:\footnote{In the next formula, $0$ stands for the constant map with value $0$.}
$$i(p_n):=\begin{cases}
\min \{i<\mu\mid f\s e^i\}, & \text{if $n<\ell(p)$};\\
\min\{i<\mu\mid e^i\upharpoonright a_n^p=0\ \&\ e^i\upharpoonright \dom(f^p_n)=f^p_n+1 \}, & \text{if $n\geq \ell(p)$}.
\end{cases}
$$
Define a map $c: P\rightarrow H_\mu$, by letting for any condition $p=\langle p_n\mid n<\omega\rangle$,
 $$c(p):=(\lh(p),\langle\rho^p_n\mid n<\lh(p)\rangle,\langle i(p_n)\mid n<\omega\rangle, \langle \vec{h}^p_n\mid n<\ell(p)\rangle,  \langle \alpha_{p_n}\mid n\geq \ell(p)\rangle, \langle \vec{G}^p_n\mid n\geq \ell(p)\rangle ),$$
where $\vec{h}^p_n:=\langle h^{ip}_n\mid i<3\rangle$ and $\vec{G}^p_n:=\langle j_{n,\alpha^p_n}(F^{ip}_n)(\alpha_{p_n})\mid i<3\rangle$.
\end{definition}
\begin{remark}\label{Clarifyingc}
	Note that $c$ is well-defined: all the entries appearing in $c(p)$ are clearly in $H_\mu$ with, perhaps, the only exception of the latter one, $\vec{G}^p_n$. 
	However,  all of these are conditions in  L\'evy collapses over $\Ult(V,E_{n,\alpha^p_n})$, an ultrapower by a $\kappa_n$-complete measure on $\kappa_n$. In particular, these collapses have $V$-cardinality $|j_n(\kappa_n)|<\kappa_{n+1}<\mu$. Hence, $\range(c)\s H_\mu$.
\end{remark}

\begin{definition}\label{DefinitionofSnMotis}
For each $n<\omega$, set
 $$S_n:=\begin{cases}
 \{\one\}, & \text{if $n=0$};\\
  \{\langle (\rho^p_k, h^{0p}_k, h^{1p}_k, h^{2p}_k)\mid k<n\rangle\mid  p\in P_n\}, & \text{if $n\geq 1$}.
 \end{cases}
 $$
For $n\geq 1$ and $s,t\in S_n$, write $s\sle_n t$ iff there are $p,q\in P_n$ with $p\LE q$ witnessing, respectively, that $s$ and $t$ are in $S_n$.

Denote $\mathbb{S}_n:=(S_n, \sle_n)$ and set $\vec{\mathbb{S}}:=\langle \mathbb{S}_n\mid n<\omega\rangle$.
\end{definition}

\begin{remark}\label{RemarkonS}\label{Pnisomorphic}
Observe that $|{S}_n|<\sigma_n$.  Moreover, for each $s\in S_n\setminus \{\one_{\mathbb{S}_n}\}$, $\mathbb{S}_n\downarrow s\cong \Col(\delta,{<}\kappa_{n-1})\times \mathbb{Q}$, where
$\mathbb{Q}$  is a notion of a forcing  of size $<\delta$ such that $\sigma_{n-1}<\delta<\kappa_{n-1}$. Specifically, if $p\in P_n$ is the condition from which $s$ arises, then $\delta=s_{n-1}(\rho_{n-1}^{p})^{++}$ and
$\mathbb{Q}$ is a product $$\mathbb{R}\times\Col(\sigma_{n-1},{<}\rho_{n-1}^{p})\times\Col(\rho_{n-1}^{p},s_{n-1}(\rho_{n-1}^{p})),$$
where $\mathbb R$ is a notion of forcing of size $\leq \kappa_{n-2}$.\footnote{In the particular case where $n=1$ the poset $\mathbb R$ is trivial.}
Also, by combining Easton's lemma with a counting of nice names, if the $\gch$ holds below $\kappa$ then $\mathbb{S}_n\downarrow s$ preserves this behavior of the power set function
 for each $s\in S_n\setminus \{\one_{\mathbb{S}_n}\}$.

On another note, observe that the the map $(q,s)\mapsto q+s$ yields an isomorphism between $(\mathbb{S}_n\downarrow\varpi_n(p))\times(\mathbb{P}^{\varpi_n}_n\downarrow p)$ and $\mathbb{P}_n\downarrow p$.\footnote{In general terms the above map simply defines a projection (see Definition~\ref{niceprojection}\eqref{theprojection}) but in the particular case of the EBPFC it moreover gives an isomorphism.}
\end{remark}

\begin{definition}\label{DefinitionvarpinMoti}
For each $n<\omega$, define $\varpi_n\colon P_{\geq n}\rightarrow S_n$  as follows:
$$\varpi_n(p):=\begin{cases}
\{\one\}, & \text{if $n=0$};\\
\langle (\rho^p_k, h^{0p}_k, h^{1p}_k, h^{2p}_k)\mid k<n\rangle, & \text{if $n\geq 1$}.
\end{cases}
$$
Set $\vec{\varpi}:=\langle\varpi_n\mid n<\omega\rangle$.
\end{definition}

The next lemma collects some useful properties about the $n0$-modules of the EBPFC (i.e, the $\mathbb{Q}_{n0}$'s) and reveals some of their connections with the corresponding modules of the EBPF (i.e, the $\mathbb{Q}^*_{n0}$'s).
\begin{lemma}\label{QnsAreclosed}
Let $n<\omega$. All of the following hold:
\begin{enumerate}
\item $\mathbb{P}_n$ projects to $\mathbb{Q}_{n0}$, and this latter projects to $\mathbb{Q}^*_{n0}$.
\item $\mathbb{Q}^*_{n0}$ is $\kappa_n$-directed-closed, while $\mathbb{Q}_{n0}$ is $\sigma_n$-directed-closed.
\item $\mathbb{S}_n$ satisfies the $(\kappa_{n-1})$-cc.
\end{enumerate}
\end{lemma}
\begin{proof}
(1) The map $p\mapsto (a^p_n,A^p_n,f^p_n,F^{0p}_n, F^{1p}_n, F^{2p}_n)$ is a projection between $\mathbb{P}_n$ and $\mathbb{Q}_{n0}$. Similarly, $(a,A,f,F^0,F^1,F^2)\mapsto (a,A,f)$  defines a projection between $\mathbb{Q}_{n0}$ and $\mathbb{Q}^*_{n0}$.

(2) The argument for the $\kappa_n$-directed-closedness of $\mathbb{Q}^*_{n0}$ is given in \cite[Lemma 10.2.40]{Pov}.
Let $D\s \mathbb{Q}_{n0}$ be a directed set of size $<\sigma_n$ and denote by $\varrho_n$ the projection between $\mathbb{Q}_{n0}$ and $\mathbb{Q}^*_{n0}$ given in the proof of item (1). Clearly, $\varrho_n[D]$ is a directed subset of $\mathbb{Q}^*_{n0}$ of size $<\sigma_n$, so that we may let $(a,A,f)$ be a $\LE_{\mathbb{Q}^*_{n0}}$-lower bound for it. By $\LE_{\mathbb{Q}^*_{n0}}$-extending $(a,A,f)$ we may assume that $\kappa_n,\mu\in a$ and that $a\cap \mu$ contains a $\leq_{E_n}$-greatest element. Set $\alpha:=\mc(a\cap\mu)$. 
For each $i<3$ and each  $\nu\in \pi_{\mc(a)\alpha}[A]$, define $F^{i}(\nu):=\bigcup_{p\in D} F^{ip}(\pi_{\alpha,\alpha_p}(\nu))$. 
Finally, $(a,A,f,F^0,F^1,F^2)$ is a  condition in $\mathbb{Q}_{n0}$ extending every $p\in D$.

(3) This is immediate from the  definition of $\mathbb{S}_n$ (Definition~\ref{DefinitionofSnMotis}).
\end{proof}

\subsection{EBPFC is $(\Sigma,\vec{\mathbb{S}})$-Prikry}\label{SectionMotiIsSigmaPrikry}
We verify that $(\mathbb{P},\ell,c,\vec{\varpi})$ is $(\Sigma,\vec{\mathbb{S}})$-Prikry having property $\mathcal{D}$. To this end, we go over the clauses of Definition~\ref{SigmaPrikry}.

\begin{conv}
For every sequence $\{A_k\}_{i\leq k\leq j}$ such that each $A_k$ is a subset of $\kappa_k$,
we shall identify $\prod_{k=i}^j A_k$ with its subset consisting only of the sequences that are moreover increasing.
\end{conv}

\begin{definition}\label{EBPFOneExtension}
Let $p=\langle p_n\mid n<\omega\rangle\in P$. Define:
\begin{itemize}
\item $p^\curvearrowright\emptyset:=p$;
\item For every $\nu\in A^p_{\lh(p)}$, $p{}^\curvearrowright \langle \nu\rangle$ is the unique condition $q=\langle q_n\mid n<\omega\rangle$, such that for each $n<\omega$:
$$q_n=\begin{cases}
p_n,&\text{if $n\neq \lh(p)$};\\
p_{\lh(p)}{}^\curvearrowright \langle \nu\rangle,& \text{otherwise}.
\end{cases} $$
\item Inductively, for all $m\ge\lh(p)$ and $\vec{\nu}=\langle \nu_{\lh(p)},\dots, \nu_m,\nu_{m+1}\rangle \in \prod_{n=\lh(p)}^{m+1} A^p_n$, set
$p{}^\curvearrowright \vec{\nu}:=(p^\curvearrowright \vec{\nu}\restriction(m+1))^\curvearrowright \langle \nu_{m+1}\rangle$.
\end{itemize}
\end{definition}

\begin{fact}\label{RemarkMotisPoset} Let $p,q\in P$.
\begin{itemize}
\item $q\LE^0 p$ iff $\ell(p)=\ell(q)$ and $q\LE_{n} p$, for each $n<\omega$;
\item $q\LE p$ iff there is $\vec{\nu}\in \prod_{n=\ell(p)}^{\ell(q)-1} A^p_n$ such that $q\LE^0 p{}^\curvearrowright \vec{\nu}$;
\item The sequence $\vec{\nu}$ above  is uniquely determined by $q$. Specifically, for each $n\in[\ell(p),\ell(q))$, $\nu_n=f^q_{n}(\mc(a^p_{n}))$. 
\end{itemize}
\end{fact}

By the very definition of the EBPFC (Definition~\ref{EBPFC}) and the function $\ell$ (Definition~\ref{lofEBPFC}), $(\mathbb{P},\ell)$ is a graded poset, hence $(\mathbb{P},\ell,c,\vec{\varpi})$ witnesses Clause~\eqref{graded}. Also, combining  Lemma~\ref{QnsAreclosed}(2) with the fact that all of the L\'{e}vy collapses considered are at least $\aleph_1$-closed, Clause~\eqref{c2} follows:

\begin{lemma}\label{C2forMoti}
For all $n<\omega$, $\mathbb{P}_n$ is $\aleph_1$-closed. \qed
\end{lemma}

We now verify that the map of Definition~\ref{cfunctionforMoti} witnesses Clause~\eqref{c1}:

\begin{lemma}\label{C5forMoti}
For all $p,q\in P$, if $c(p)=c(q)$, then $P_0^p\cap P_0^q$ is non-empty.
\end{lemma}
\begin{proof}
Let $p, q\in P$ and assume that $c(p)=c(q)$. By Definition~\ref{cfunctionforMoti}, we have  $\ell(p)=\ell(q)$ and $\rho^p_n=\rho^q_n$ for all $n<\ell(p)$. Set $\ell:=\ell(p)$ and $\rho_n:=\rho^p_n$ for each $n<\ell$. Also, $c(p)=c(q)$ yields  $\vec{h}^p_n=\vec{h}^q_n$ for each $n<\ell$, and $\alpha_{p_n}=\alpha_{q_n}$ and $\vec{G}^p_n=\vec{G}^q_n$ for each $n\geq \ell$. Put $\vec{h}_n:=\vec{h}^p_n$ and write $\vec{h}_n=(h^0_n,h^1_n,h^2_n)$. Denote by $\alpha_n$ the common value  $\alpha_{p_n}=\alpha_{q_n}$. We now define $r\in P^p_0\cap P^q_0$.

\smallskip

$\br$ If $n<\ell$ then $c(p)=c(q)$ implies $i=i(p_n)=i(q_n)$, and so $f^p_n\cup f^q_n \s e^{i}$. 
Set $r_n:=(\rho_n, f^p_n\cup f^q_n, {h}^{0}_n,h^{1}_n, h^{2}_n)$. Clearly, $r_n\in Q_{n1}$.

\smallskip

$\br$  For $n\geq \ell$  put $a^r_{\ell-1}:=\emptyset$ and $\alpha_{\ell-1}:=0$ and argue by recursion towards defining $a^r_n$. We assume by induction that $\mc(a^r_m\cap\mu)=\alpha_m$ for  $m\in [\ell-1,n)$.

 Since $i(p_n)=i(q_n)$, arguing as in \cite[Lemma 10.2.41]{Pov} we can make $(a^p_n,A^p_n,f^p_n)$ and $(a^q_n,A^q_n,f^q_n)$ compatible by taking the triple $$(a^r_{n-1}\cup a^p_n\cup a^q_n\cup \{\alpha^*\}, A^*, f^p_n\cup f^q_n),$$
 where $\alpha^*$ is some ordinal in $\lambda\setminus\bigcup_{\ell\leq m\leq n} (\mathrm{dom}(f^p_m)\cup \mathrm{dom}(f^q_m))$ such that $\alpha^*$ is $\leq_{E_n}$-above the ordinals in $a^r_{n-1}\cup a^p_n\cup a^q_n$. Also, $A^*$ is some suitable $E_{n,\alpha^*}$-large set. Since we can pick such an $\alpha^*$ as large as we wish (below $\lambda$) we may assume that it is actually above $\mu$. In particular, putting $$a^r_n:=a^{r}_{n-1}\cup a^p_n\cup a^q_n\cup \{\alpha^*\}$$ we have that $a^r_n$ has a $\leq_{E_n}$-maximal element and $\mathrm{mc}(a^r_n\cap\mu)=\alpha_{n}$.
 
 Let us now to define the $F$-component of $r_n$. Since for each $i\leq 2$ $j_{n,\alpha_n}(F^{pi}_n)(\alpha_n)=j_{n,\alpha_n}(F^{qi}_n)(\alpha_n)$ we have a $E_{n,\alpha_n}$-measure one set $B_i$ for which $F^{pi}_n\restriction B=F^{qi}_n\restriction B$.  Let $B:= B_0\cap B_1 \cap B_2$ and $A^{r}_n:= A^*\cap \pi^{-1}_{\alpha^*,\alpha_n}``B$. Clearly, $A^r_n\in E_{n,\alpha^*}$. Finally, define $F^{ri}_n:= F^{pi}_n\restriction \pi_{\alpha^*,\alpha_n}`` A^r_n$ and put $$r_n=(a^r_n, A^r_n, f^p_n\cup f^q_n, F^{r0}_n,  F^{r1}_n,  F^{r2}_n).$$
 
 After this recursive definition we obtain a condition $r=\langle r_n\mid n<\omega\rangle$  witnessing $P^p_0\cap P^q_0\neq \emptyset$.
\end{proof}

The verification of Clauses \eqref{c5}, \eqref{csize} and \eqref{itsaprojection} is the same as in \cite[Lemma 10.2.45, 10.2.46 and 10.2.47]{Pov}, respectively.
It is worth saying that regarding Clause~\eqref{csize} we actually have that $|W(p)|\leq \kappa$ for each $p\in P$.\label{C6forMoti}

We now show that $(\mathbb{P},\ell)$ has property $\mathcal{D}$ and that it  satisfies Clause~\eqref{c6}.

\begin{lemma} \label{propDcollapses}
$(\mathbb{P},\ell)$ has property $\mathcal{D}$.
\end{lemma}
\begin{proof}
Let $p\in P$, $n<\omega$ and $\vec{r}$ be a good enumeration of $W_n(p)$. Our aim is to show that \pI\, has a winning strategy in the game $\Game_\mathbb{P}(p,\vec{r})$. To enlighten the exposition we just give details for the case when $n=1$. The general argument can be composed using the very same ideas.

Write $p=\langle (f_n,\rho_n, h^0_n, h^1_n, h^2_n)\mid n<\ell\rangle{}^\smallfrown \langle (a_n,A_n,f_n,F^0_n,F^1_n,F^2_n)\mid n\geq \ell\rangle$. By Fact~\ref{RemarkMotisPoset}, we can identify $\vec{r}$  with  $\langle \nu_\xi\mid \xi<\kappa_\ell\rangle$, a good enumeration of  $A_\ell$. Specifically,  for each $\xi<\kappa_\ell$ we have that $r_\xi=p\cat {\nu_\xi}$. Using this enumeration we  define a sequence $\langle (p_\xi,q_\xi)\mid \xi<\kappa_\ell\rangle$ of moves in $\Game_\mathbb{P}(p,\vec{r})$.

To begin with, $\pI$ plays $p_0:=p$ and in response  $\pII$ plays some $q_0\leq^0 r_0$ with $q_0\leq p_0$. Note that this move  is possible, as $p_0$ and $r_0$ are compatible.

Suppose by induction  that we have defined a sequence $\langle (p_\eta, q_\eta)\mid \eta<\xi\rangle$ of moves in $\Game_\mathbb{P}(p,\vec{r})$ which moreover satisfies the following:

\begin{enumerate}
\item\label{C1propertyD} For each $n<\ell$ the following hold:
\begin{enumerate}
\item for all $\eta<\xi$,  $\rho^{p_\xi}_n=\rho_n$, $h^{0p_\xi}_n=h^0_n$, $h^{1p_\xi}_n=h^1_n$, $h^{2p_\xi}_n=h^2_n$;
\item for all $\zeta<\eta<\xi$, $f^{q_\zeta}_n\s f^{p_\eta}_n$;
\end{enumerate}
\item\label{C2propertyD} For all $\zeta<\eta<\xi$ and $n>\ell$, $(q_\eta)_n\leq_{n0} (p_\eta)_n\leq_{n0} (q_\zeta)_n$;
\item\label{C3propertyD} For all $\eta<\xi$:
\begin{enumerate}
\item $a^{p_\xi}_\ell=a_\ell,\;\; A^{p_\xi}_\ell=A_\ell,\;\; F^{0p_\xi}_\ell=F^{0}_\ell\;\text{and}\;\; F^{1p_\xi}_\ell=F^{1}_\ell;$
\item for each $\zeta<\eta$, if
${(\bar{\nu}_\zeta)}_0={(\bar{\nu}_\eta)}_0$ then $h^{2q_\zeta}_\ell\s h^{2q_\eta}_\ell.$
\end{enumerate}
\end{enumerate}
Let us show how to define the $\xi^{\rm{th}}$ move of $\pI$:
\smallskip

\underline{Successor case:} Suppose $\xi=\eta+1$.  Then put $p_\xi:=\langle {(p_\xi)}_n\mid n<\omega\rangle$, where
$${(p_\xi)}_n:=\begin{cases}
(f^{q_\eta}_n,\rho_n, h^{0}_n, h^1_n, h^2_n), & \text{if $n<\ell$;}\\
(a_n, A_n, f^{q_\eta}_n\setminus a_n, F^0_n, F^1_n, F^{2\xi}), & \text{if $n=\ell$;}\\
(q_\eta)_n, & \text{if $n>\ell$.}
\end{cases}
$$
Here $F^{2\xi}$ denotes the map with domain $\pi_{\mc(a_\ell),\alpha_{p_\ell}}``A_\ell$ defined as follows:
$$F^{\xi,2}(\bar{\nu}):=\begin{cases}
F^{2}_\ell(\bar{\nu})\cup \bigcup\{h^{q_\zeta,2}_\ell\mid \zeta<\xi,\, ({\bar{\nu}_\zeta})_0=({\bar{\nu}_\xi})_0\}, & \text{if $\nu=\nu_\xi$;}\\
F^{2}_\ell(\bar{\nu}), & \text{otherwise.}
\end{cases}
$$
By Clauses~\eqref{C3propertyD} of the induction hypothesis and our comments in Remark~\ref{RemarkMeasureOneSets}, $F^{2\xi}$ is a function.  A moment's reflection makes it clear that $p_\xi$ is a condition in $\mathbb{P}$ witnessing \eqref{C1propertyD} and \eqref{C3propertyD}(a) above. Also, $p_\xi\leq^0 p$ and $p_\xi$ is compatible with $r_\xi$, hence it is a legitimate move for $\pI$.\footnote{Note that $p_\xi\cat{\nu_\xi}\leq p_\xi, r_\xi$.} In response, $\pII$ plays $q_\xi\leq^0 r_\xi$ such that $q_\xi\leq p_\xi$. In particular, for each $n>\ell$, $(q_\xi)_n\leq_{n0} (p_\xi)_n\leq_{n0} (q_\eta)_n$, and also $F^{2\xi}(\bar{\nu}_\xi)\s h^{2q_\xi}_\ell$. This combined with the induction hypothesis yield  Clause~\eqref{C2propertyD} and \eqref{C3propertyD}(b),  which completes the successor case.

\smallskip

\underline{Limit case:} In the limit case we put $p_\xi:=\langle (p_\xi)_n\mid n<\omega\rangle$, where
$${(p_\xi)}_n:=\begin{cases}
(\bigcup_{\eta<\xi} f^{q_\eta}_n,\rho_n, h^{0}_n, h^1_n, h^2_n), & \text{if $n<\ell$;}\\
(a_n, A_n, \bigcup_{\eta<\xi} (f^{q_\eta}_n\setminus a_n), F^0_n, F^1_n, F^{2\xi}), & \text{if $n=\ell$;}\\
(q_\xi^*)_n, & \text{if $n>\ell$.}
\end{cases}
$$
Here, $F^{2\xi}$ is defined as before and $(q^*_\xi)_n$ is a lower bound for the sequence $\langle (p_\eta)_n\mid \eta<\xi\rangle$. Note that this choice is possible because the  orderings $\leq_{n0}$ are $\sigma_{\ell+1}$-directed-closed. Once again, $p_\xi$ is a legitimate move for $\pI$ and, in response, $\pII$ plays $q_\xi$. It is routine to check that \eqref{C1propertyD}--\eqref{C3propertyD} above hold.

\smallskip

After this process we get a sequence $\langle (p_\xi, q_\xi)\mid \xi<\kappa_\ell\rangle$. We next show how to form a condition $p^*\leq^0 p$ diagonalizing $\langle q_\xi\mid \xi<\kappa_\ell\rangle$.

Note that by shrinking $A_\ell$ to some  $A'_\ell$ we may assume that there are maps $\langle(h^{*0}_n,h^{*1}_n,h^{*2}_n)\mid n<\ell\rangle$ such that $h^{iq_\xi}_n=h^{*i}_n$ for all $\nu_\xi\in A'_\ell$ and $i<3$. Next, define a map $t$ with domain $A'_\ell$ such that $t(\nu):=\langle h^{0q_{\nu}}_\ell, h^{1q_{\nu}}_\ell\rangle$.\footnote{In a slight abuse of notation, here we are identifying $q_\nu$ with $q_\xi$, where $\nu=\nu_\xi$.} Since $j_\ell(t)(\mc(a_\ell))\in V_{\kappa+1}^{M_{E_\ell}}$ we can argue as in \cite[Claim~1]{GitikCollapsin} that there is $\alpha<\mu$ and a map $t'$ such that $j_\ell(t)(\mc(a_\ell))=j_\ell(t')(\alpha)$. Now let $a^*_\ell$ be  such that $a_\ell\cup\{\alpha\}\s a^*_\ell$  witnessing Clause~(1) of Definition~\ref{ModulesforGitiks}$(0)_n$. Then,
$$A:=\{\nu<\kappa_\ell\mid t\circ \pi_{\mc(a_\ell^*),\mc(a_\ell)}(\nu)=t'\circ \pi_{\mc(a_\ell^*\cap \mu),\alpha} \circ \pi_{\mc(a_\ell^*),\mc(a_\ell^*\cap \mu)}(\nu) \}$$
is $E_{\ell,\mc(a^*_\ell)}$-large. Set $A^*_\ell:= A\cap \pi_{\mc(a^*_\ell), \mc(a_\ell)}^{-1}A'_\ell$ and $$\hat{t}:=(t'\circ \pi_{\mc(a_\ell^*\cap \mu),\alpha})\upharpoonright\pi_{\mc(a^*_\ell), \mc(a^*_\ell\cap \mu)}``A^*_\ell.$$
Note that $\pi_{\mc(a^*_\ell),\mc(a_\ell)}``A^*_\ell\s A'_\ell\s A_\ell$.  Also, for each $\nu\in A^*_\ell$,
$$\hat{t}(\pi_{\mc(a^*_\ell), \mc(a^*_\ell\cap \mu)}(\nu))=t(\tilde{\nu})=\langle h^{0q_{\tilde{\nu}}}_\ell, h^{1q_{\tilde{\nu}}}_\ell\rangle,$$
where $\tilde{\nu}:=\pi_{\mc(a^*_\ell), \mc(a_\ell)}(\nu)$. For each $i<2$,  define a map $F^{*,i}_\ell$ with domain $\pi_{\mc(a^*_\ell),\mc(a_\ell)}``A^*_\ell$, such that for each $\nu\in A^*_\ell$,
$$F^{*,i}_\ell(\pi_{\mc(a^*_\ell), \mc(a^*_\ell\cap \mu)}(\nu)):=h^{iq_{\tilde{\nu}}}_\ell.$$
Similarly, define $F^{*,2}_\ell$ by taking lower bounds over the stages of the inductive construction mentioning ordinals $\nu_\eta\in \pi_{\mc(a^*_\ell),\mc(a_\ell)}``A^*_\ell$; i.e.,
$$F^{*,2}_\ell(\pi_{\mc(a^*_\ell), \mc(a^*_\ell\cap \mu)}(\nu)):= \bigcup\{h^{2q_{\nu_\eta}}\mid \nu_\eta\in \pi_{\mc(a^*_\ell),\mc(a_\ell)}``A^*_\ell,\, (\bar{\tilde{\nu}})_0=(\bar{\nu}_\eta)_0\}.\footnote{Once again, this choice is legitimate in that there are not too many $\nu$ with the same projection to the normal measure,  $\nu_0$ (see Remark~\ref{RemarkMeasureOneSets}).}$$
Next, define $p^*:=\langle p^*_n\mid n<\omega\rangle$, where
$$p^*_n:=\begin{cases}
(\bigcup_{\xi<\kappa_\ell} f^{q_\xi}_n,\rho_n, h^{*0}_n, h^{*1}_n, h^{*2}_n), & \text{if $n<\ell$;}\\
(a^*_n, A^*_n, f^p_n\cup \bigcup_{\xi<\kappa_\ell} (f^{q_\eta}_n\setminus a^*_n), F^{*0}_n, F^{*1}_n, F^{*2}_n), & \text{if $n=\ell$;}\\
q^*_n, & \text{if $n>\ell$.}
\end{cases}
$$
and $q^*_n$ is a $\leq_{n0}$-lower bound for $\langle (q_\xi)_n\mid \xi<\kappa_\ell\rangle$.

\begin{claim}
$p^*$ is a condition in $\mathbb{P}$ diagonalizing $\langle q_\xi\mid \xi<\kappa_\ell\rangle$.
\end{claim}
\begin{proof}
Clearly, $p^*\in P$ and it is routine to check that $p^*\leq^0 p$.

Let $s\in W_1(p^*)$ and $\nu\in A^*_\ell$ be with $s=p^*\cat {\nu}$. Since $\pi_{\mc(a^*_\ell), \mc(a_\ell)}``A^*_\ell$ is contained in $A_\ell$ there is some $\xi<\kappa_\ell$ such that $\tilde{\nu}=\nu_\xi$. Note that $w(p,s)=p\cat{\nu_\xi}$, hence we need to prove that $p^*\cat{\nu}\leq^0 q_\xi$. Note that for this it is enough to show that
$h^{iq_\xi}_\ell\s F^{*i}_\ell(\pi_{\mc(a^*_\ell), \mc(a^*_\ell\cap \mu)}(\nu))$ for $i<3$. And, of course, this follows from our definition of $F^{*i}_\ell$  and  the fact that $\tilde{\nu}=\nu_\xi$.
\end{proof}
The above shows that $\pI$ has a winning strategy for the game $\Game_\mathbb{P}(p,\vec{r})$.
\end{proof}

\begin{lemma}\label{C8forMotis}
$(\mathbb{P},\ell)$ has the $\rm{CPP}$.
\end{lemma}
\begin{proof}
Fix $p\in P$, $n<\omega$ and $U$  a $0$-open set. Set $\ell:=\ell(p)$.
\begin{claim}\label{Diagonalization}
There is $q\leq^0 p$ such that if $r\in P^q\cap U$ then $w(q,r)\in U$.
\end{claim}
\begin{proof}
For each $n<\omega$ and a good enumeration $\vec{r}:=\langle r^n_\xi\mid \xi<\chi\rangle$ of $W_n(p)$ appeal to Lemma~\ref{propDcollapses} and  find $p_n\leq^0 p$ such that $p_n$ diagonalizes a sequence $\langle q^n_\xi\mid \xi<\chi\rangle$ of moves for $\pII$ which moreover satisfies   $$P^{r^n_\xi}_0\cap U\neq \emptyset\;\implies\;q^n_\xi\in U.$$
Appealing iteratively to Lemma~\ref{propDcollapses} we arrange $\langle p_n\mid n<\omega\rangle$  to be $\leq^0$-decreasing, and by Definition~\ref{SigmaPrikry}\ref{c2} we find $q\leq^0 p$ a lower bound for  it.

Let $r\leq q$ be in $D$ and set $n:=\ell(r)-\ell(q)$. Then, $r\leq^n p_n$  and so $r\leq^0 w(p_n,r)\leq^0 q^n_\xi$ for some $\xi$. This implies that $q^n_\xi\in U$. Finally, since $w(q,r)\leq^0 w(p_n,r)\leq^0 q^n_\xi$ we infer from $0$-openess of $U$ that  $w(q,r)\in U$.
\end{proof}

Let $q\LE^0 p$ be as in the conclusion of Claim~\ref{Diagonalization}. We define by induction a $\LE^0$-decreasing sequence of conditions $\langle q_n\mid n<\omega\rangle$ such that for each $n<\omega$
$$(\star)_n\;\;W_n(q_n)\s U\text{ or }W_n(q_n)\cap U=\emptyset.$$
The cases $n\leq 1$ are easily handled and the cases $n\geq 3$ are similar to the case $n=2$. So, let us simply describe how do we proceed in this latter case.
 Suppose that $q_1$ has been  defined. For each $\nu\in A^{q_1}_\ell$, define $$A_\nu^+:=\{\delta\in A^{q_1}_{\ell+1}\mid q_1{}^\curvearrowright\langle\nu,\delta\rangle\in U\}\;\;\text{and}\;\; A_\nu^-:=A^{q_1}_{\ell+1}\setminus A_\nu^+.$$
 If $A_\nu^+$ is large then set $A_\nu:= A_\nu^+$.\footnote{More explicitly, $E_{\ell+1,\mc(a^{q_1}_{\ell+1})}$-large.}  Otherwise, define $A_\nu:=A_\nu^-$. Put $A^+:=\{\nu\in A^{q_1}_\ell\mid A_\nu= A_\nu^+\}$, and $A^-:=\{\nu\in A^{q_1}_\ell\mid A_\nu= A_\nu^-\}$. If $A^+$ is large we let $A_\ell:=A^+$ and otherwise $A_\ell:=A^-$.  Finally, let $q_2\LE^0q_1$ be such that $A^{q_2}_\ell:=A_\ell$ and $A^{q_2}_{\ell+1}:=\bigcap_{\nu\in A^{q_1}_\ell} A_\nu$ Then, $q_2$ witnesses $(\star)_2$.

Once we have defined $\langle q_n\mid n<\omega\rangle$, let $q_\omega$ be a $\LE^0$-lower bound for it (cf.~Definition~\ref{SigmaPrikry}\eqref{c2}). It is routine to check that, for each $n<\omega$, the condition $q_\omega$ witnesses property $(\star)_n$, hence it is as desired.
\end{proof}

Let us dispose with the verification of Clauses~\eqref{PnprojectstoSn} and \eqref{moreclosedness}:
\begin{lemma}\label{C3forMoti}\label{AdditionalAssumption1forMoti}
For all $n<\omega$, the map $\varpi_n$ is a nice projection from $\mathbb{P}_{\geq n}$ to $\mathbb{S}_n$ such that, for all $k\ge n$, $\varpi_n\restriction \mathbb P_k$ is again a nice projection to $\mathbb{S}_n$.

Moreover, the sequence of nice projections $\vec{\varpi}$ is  coherent.\footnote{See Definition~\ref{CoherentSystem}.}

\end{lemma}
\begin{proof}
 Fix some $n<\omega$. By definition,  $\varpi_n(\one_\mathbb{P})=\one_{\mathbb{S}_n}$  and  it is not hard to check that it  is order-preserving. Let $p\in P_{\geq n}$ and $s\preceq_n \varpi_n(p)$.

 Then  $s=\langle (\rho^p_k, h^{0}_k, h^{1}_k,h^{2}_k)\mid k<n\rangle$, and we define $r:=\langle r_k\mid k<\omega\rangle$ as
$$r_k:=\begin{cases}
(\rho^p_k, f^p_k, h^{0}_k, h^{1}_k, h^{2}_k), & \text{if $k<n$};\\
p_k, & \text{otherwise}.
\end{cases}
$$
It is not hard to check that $r\LE^0 p$ and  $\varpi_n(r)=s$.  
Actually, $r$ is the greatest such condition, hence $r=p+s$.  This yields Clause~\eqref{niceprojection3} of Definition~\ref{niceprojection}.

 For the verification of Clause~\eqref{theprojection} of Definition~\ref{niceprojection},  let $q\leq^0 p+s$ and define a sequence $p':=\langle p'_n\mid n<\omega\rangle$ as follows:

 $$p'_k:=\begin{cases}
(\rho^p_k, f^p_k, h^{0p}_k, h^{1p}_k, h^{2p}_k), & \text{if $k<n$};\\
q_k, & \text{otherwise}.
\end{cases}$$
Note that $p'\in P$ and $p'+\varpi_n(q)=q$. Thus, Clause~\eqref{theprojection} follows.

 Altogether, the above shows that $\varpi_n$ is a nice projection. 
 Similarly, one shows that $\varpi_n\restriction \mathbb P_k$ is a nice projection for each $k\geq n$. 
 Finally, the moreover part of the lemma follows from the definition of $\varpi_n$ and the fact that $W(p)=\{p{}^\curvearrowright\vec{\nu}\mid \vec{\nu}\in \prod_{k=\ell(p)}^{\ell(p)+|\vec{\nu}|-1} A^p_k\}.$
\end{proof}

\begin{lemma}\label{C4forMoti}
For each $n<\omega$, $\mathbb{P}^{\varpi_n}_n$ is $\sigma_n$-directed-closed.\footnote{In particular, taking $\z{\mathbb{P}}_n:=\mathbb{P}_n$ Clause~\eqref{moreclosedness} follows.}
\end{lemma}
\begin{proof}
Since $\mathbb{P}^{\varpi_0}_0=\{\one\}$ the result is clearly true for $n=0$.

Let $n\geq 1$ and $D\s \mathbb{P}^{\varpi_n}_n$ be a directed set of size $<\sigma_n$. By definition, $$\varpi_n[D]=\{\langle (\rho^p_k, h^{0p}_k, h^{1p}_k,h^{2p}_k)\mid k<n\rangle\},$$ for some (all) $p\in D$. By taking intersection of the measure one sets and unions on the other components of the conditions of $D$ one can easily form a condition $q$ which is a $\leq^{\vec{\varpi}}$-lower bound fo $D$.
\end{proof}

Finally, the proof of the next is identical to \cite[Corollary 10.2.53]{Pov}.
\begin{cor}\label{MuInEBPFC}
$\one_{\mathbb{P}}\forces_{\mathbb{P}}\check{\mu}=\kappa^+$.\qed
\end{cor}

Combining all the previous lemmas we  finally arrive at the desired result:
\begin{cor}\label{TheEBPFCisweakly}
$(\mathbb{P},\ell,c,\vec{\varpi})$ is a   $(\Sigma,\vec{\mathbb{S}})$-Prikry forcing that has property $\mathcal{D}$.

Moreover,  the sequence $\vec{\varpi}$ is coherent.\qed
\end{cor}

\subsection{EBPFC is suitable for reflection}\label{sec43}
In this section we show that $(\mathbb{P}_n,\mathbb{S}_n,\varpi_n)$ is suitable for reflection with respect to a relevant sequence of cardinals.
Our setup will be the same as the one from page~\pageref{setupGitik} and we will also rely on the notation established in page~\pageref{NotationMotis}.
The main result of the section is Corollary~\ref{Motisuitableforreflection}, which will be preceded by a series of technical lemmas.
The first one is essentially due to Sharon:

\begin{lemma}[\cite{AS}]\label{EBPFcollapsesmu}
For each $n<\omega$,
$V^{\mathbb{Q}^*_{n0}}\models|\mu|=\cf(\mu)=\kappa_n$.
\end{lemma}
\begin{proof}
By Lemma~\ref{QnsAreclosed},  $\mathbb{Q}^*_{n0}$ preserves cofinalities ${\leq}\kappa_n$, and by the Linked$_0$-property \cite[Lemma 10.2.41]{Pov} it preserves cardinals ${\geq}\mu^+$.

Next we show that $\mathbb{Q}^*_{n0}$ collapses $\mu$ to $\kappa_n$.
For each condition $p\in \mathbb{Q}^*_{n0}$, denote $p:=(a^p, A^p, f^p)$.
Let $G$ be $\mathbb{Q}^*_{n0}$-generic and set $a:=\bigcup_{p\in G}a^p$. By a density argument, $\otp(a\cap\mu)=\kappa_n$, and so $\mu$ is collapsed.
Finally, by a result of Shelah, this implies that $V^{\mathbb{Q}^*_{n0}}\models \cf(|\mu|)=\cf(\mu)$.\footnote{For Shelah's theorem, see e.g. \cite[Fact 4.5]{cfm}.}
\end{proof}

\begin{lemma}\label{MotisCollapsesmu}
For each $n<\omega$, $V^{\mathbb{Q}_{n0}}\models|\mu|=\cf(\mu)=\kappa_n=(\sigma_n)^+$.
\end{lemma}
\begin{proof}
 For each $p\in Q_{n0}$,   
 $F^{0p}$, $F^{1p}$ and $F^{2p}$ can be seen  (respectively) as  representatives of conditions in the collapses $ \col(\sigma_{n},{<}\kappa_n)^{M_n^*}$, $\col(\kappa_n,\kappa^+)^{M_n^*}$ and $\col(\kappa^{+3},{<}j_n(\kappa_n))^{M_n^*}$, where $M^*_n\cong \Ult(V,E_n\restriction\mu)$.\footnote{Despite the fact that $\dom(F^{ip})\in E_{n,\alpha_{p_n}}$ there is a natural factor embedding between $M_{n,\alpha_{p_n}}\simeq \Ult(V,E_{n,\alpha_{p_n}
})$ and $M^*_n$ that enable us to see $j_n(F^{ip})(\alpha_{p_n})$ as a member of the said collapses; namely, the embedding defined by $k(j_{n,\alpha_{p_n}}(F^{ip})(\alpha_{p_n}))=j_n(F^{ip})(\alpha_{p_n}).$}  Also, observe that the first of these forcings is nothing but $\col(\sigma_n,{<}\kappa_n)^V$.\footnote{Here we use that $V_{\kappa_n}\s M^*_n$.} Set $C_n:=\{\langle F^{1p},F^{2p}\rangle\mid p\in Q_{n0}\}$ and define $\sqsubseteq$ as follows:
 $$\langle F^{1p},F^{2p}\rangle \sqsubseteq \langle F^{1q},F^{2q}\rangle \;\; \text{iff}\;\; \forall i\in\{1,2\}\; j_n(F^{ip})(\alpha_p)\supseteq j_n(F^{iq})(\alpha_q).$$ Clearly, $\sqsubseteq$ is transitive, so that $\mathbb{C}_n:=(C_n,\sqsubseteq)$ is a forcing poset.
  For each condition $c$ in $\mathbb{C}_n$ let us  denote by $\alpha_{c}$  the ordinal $\alpha_{p_c}$ relative to a condition $p_c$ in $\mathbb{Q}_{n0}$ witnessing that $c\in C_n$.
 The following is a routine verification:

 \begin{claim}\label{CnIsClosed}
 $\mathbb{C}_n$ is $\kappa_n$-directed closed. Furthermore, if $D\s \mathbb{C}_n$ is a directed set  of size $<\kappa_n$ and $\alpha<\mu$ is $\leq_{E_n}$-above all  $\{\alpha_c\mid c\in D\}$. 
 Then, there is $\sqsubseteq$-lower bound $\langle F^1,F^2\rangle$ for $D$  with $\dom(F^1)=\dom(F^2)\in E_{n,\alpha}$.
 \end{claim}

Let $G$  be a $\mathbb{Q}_{n0}$-generic filter over $V$ and denote by $G^*$ the $\mathbb{Q}^*_{n0}$-generic induced by $G$ and the projection $\varrho_n$ of Lemma~\ref{QnsAreclosed}(1).
By Lemma~\ref{EBPFcollapsesmu}, $V[G^*]\models |\mu|=\cf(\mu)=\kappa_n$, hence it is left to check that $\kappa_n$ is preserved and turned into $(\sigma_n)^+$. This is established in two steps.

\begin{claim}\label{claimdenseembedding}
The map $e\colon \mathbb{Q}_{n0}/G^*\rightarrow \col(\sigma_n,{<}\kappa_n)^V\times\mathbb{C}_n$ defined in $V[G^*]$ via $$p\mapsto \langle j_n(F^{0p})(\alpha_p), \langle F^{1p}, F^{2p} \rangle\rangle$$ is a dense embedding.
\end{claim}
\begin{proof} By a routine verification.
\end{proof}

\begin{claim}
$V[G]\models |\mu|=\cf(\mu)=\kappa_n=(\sigma_n)^+$.
\end{claim}
\begin{proof}
Since $\mathbb{Q}^*_{n0}$ is $\kappa_n$-directed closed,  $\col(\sigma_n,{<}\kappa_n)^V=\col(\sigma_n,{<}\kappa_n)^{V[G^*]}$. Note that  $\mathbb{C}_n$ is still $\kappa_n$-directed closed over $V[G^*]$ and that in any generic extension by 
$\col(\sigma_n,{<}\kappa_n)^{V}$ over $V[G^*]$, ``$|\mu|=\kappa_n=(\sigma_n)^+$'' holds.

 Appealing to Easton's lemma,  $\mathbb{C}_n$ is $\kappa_n$-distributive in any  extension of $V[G^*]$ by $\col(\sigma_n,{<}\kappa_n)^{V}$. Thus, forcing with $\col(\sigma_n,{<}\kappa_n)^V\times \mathbb{C}_n$ (over $V[G^*]$) yields a generic extension  where  ``$|\mu|=\kappa_n=(\sigma_n)^+$''  holds. Since  $(\mu^+)^V$ is preserved, a theorem of Shelah (see \cite[Fact~4.5]{cfm}) yields $``\cf(\mu)=\cf(|\mu|)$'' in the above generic extension. Thus, $\col(\sigma_n,{<}\kappa_n)^V\times \mathbb{C}_n$ forces (over $V[G^*]$)  that $``|\mu|=\cf(\mu)=\kappa_n=(\sigma_n)^+$'' holds.
 The result now follows using Claim~\ref{claimdenseembedding}, as it in particular implies that $\col(\sigma_n,{<}\kappa_n)^V\times \mathbb{C}_n$ and $\mathbb{Q}_{n0}/G^*$ are forcing equivalent over $V[G^*]$.
\end{proof}
This completes the proof.
\end{proof}
\begin{lemma}\label{productdoesnotcollapse}
For all non-zero $n<\omega$, $\prod_{i<n}\mathbb{Q}_{i1}$ is isomorphic to a product of $\mathbb{S}_n$ with some $\mu$-directed-closed forcing.  
\end{lemma}
\begin{proof}
The map $p\mapsto \langle  \langle (\rho^{p_i}, h^{0p_i}, h^{1p_i}, h^{2p_i})\mid i<n\rangle\rangle, \langle f^{p_i}\rangle_{i<n}\rangle$
yields the desired isomorphism.
\end{proof}

\begin{lemma}\label{CardinalConfigurationMotisModel}
For each $n<\omega$, $V^{\mathbb{P}_n}\models|\mu|= \cf(\mu)=\kappa_n=(\sigma_n)^+$.
\end{lemma}
\begin{proof}
 Observe that
 $\mathbb{P}_n$ is  a dense subposet of $\prod_{i<n}\mathbb{Q}_{i1}\times\prod_{i\geq n}\mathbb{Q}_{i0}$, hence both forcing produce the same generic extension.
By virtue of Lemma~\ref{MotisCollapsesmu} we have
$V^{\mathbb{Q}_{n0}}\models|\mu|=\cf(\mu)=\kappa_n=(\sigma_{n})^+.$ Also, $\mathbb{Q}_{n0}$ is $\sigma_n$-directed-closed,  hence Easton's lemma, Lemma~\ref{QnsAreclosed}(3) and Lemma~\ref{productdoesnotcollapse} combined imply that $\mathbb{Q}_{n0}$ forces $\prod_{i<n}\mathbb{Q}_{i1}$ to be a product of a $(\kappa_{n-1})$-cc forcing  times a $\kappa_n$-distributive forcing. Similarly,  $\mathbb{Q}_{n0}$ forces  $\prod_{i>n}\mathbb{Q}_{i0}$ to be $\kappa_n$-distributive. Thereby, forcing with $\prod_{i<n} \mathbb{Q}_{i1}\times \prod_{i>n} \mathbb{Q}_{i0}$ over $V^{\mathbb{Q}_{n0}}$ preserves the above cardinal configuration and thus the result follows.
\end{proof}

As a consequence of the above we get the main result of the section:
\begin{cor}\label{Motisuitableforreflection}
For each $n\geq 2$, $(\mathbb{P}_n,\mathbb{S}_n,\varpi_n)$ is suitable for reflection with respect to the sequence $\langle \sigma_{n-1},\kappa_{n-1},\kappa_n, \mu\rangle$.
\end{cor}
\begin{proof}
We go over the clauses of Definition~\ref{suitableforiteration}.
Clause~\eqref{suitableforiteration0} is obvious.
Clause~\eqref{suitableforiteration1} follows from  Lemma~\ref{C3forMoti} and Lemma~\ref{C4forMoti}.
Clause~\eqref{suitableforiteration3} follows from the comments in Remark~\ref{RemarkonS}.
For Clause~\eqref{suitableforiteration2}, note that $\mathbb P_n$ forces ``$|\mu|=\cf(\mu)=\kappa_n=(\sigma_{n})^+"$
(Lemma~\ref{CardinalConfigurationMotisModel}), hence the last paragraph of Remark~\ref{Pnisomorphic} implies that $\mathbb{S}_n\times \mathbb{P}_n^{\varpi_n}$ forces the same.\footnote{Recall that $\sigma_n:=(\kappa_{n-1})^+$ (Setup~\ref{setupGitik}).} 
\end{proof}

We conclude this section, establishing two more facts that will be needed for the proof of the Main Theorem in Section~\ref{ReflectionAfterIteration}.

\begin{definition}\label{defn438}
For every $n<\omega$, let $\mathbb{T}_n:=\mathbb{S}_n\times \Col(\sigma_n,{<}\kappa_n)$,
and let
$\psi_n:\mathbb{P}_n\rightarrow \mathbb{T}_n$ be the map defined via
$$\psi_n(p):=\begin{cases}
 \langle \varpi_n(p), j_n(F^{0p_n})(\alpha_{p_n})\rangle, & \text{if $\ell(p)>0$};\\
 \langle\one_{\mathbb{S}_n}, \emptyset\rangle, & \text{otherwise.}
 \end{cases}$$
\end{definition}

\begin{lemma}\label{keysubclaim} Let $n<\omega$.
\begin{enumerate}
\item\label{keysubclaim1} $\mathbb T_n$ is a $\kappa_n$-cc poset of size $\kappa_n$;
\item $\psi_n$ defines a nice projection;\label{keysubclaimprojection}
\item $\mathbb{P}^{\psi_n}_n$ is  $\kappa_n$-directed-closed;\label{keysubclaim2}
\item for each $p\in P_n$,   $\mathbb{P}_n\downarrow p$ and
$(\mathbb{T}_n\downarrow \psi_n(p))\times (\mathbb{P}^{\psi_n}_n\downarrow p)$ are isomorphic. In particular, both are forcing equivalent.\label{keysubclaim3}
\end{enumerate}
\end{lemma}
\begin{proof} (1) This is obvious. 

(2) Let us go over the clauses of Definition \ref{niceprojection}.
Clearly,  $\psi_n(\one_{\mathbb{P}})=\langle\one_{\mathbb{S}_n}, \emptyset\rangle$,
so Clause~\eqref{niceprojection1} holds. Likewise,  using that $\varpi_n$ is order-preserving it is routine to check that so is $\psi_n$. Thus, Clause~\eqref{niceprojection2} holds, as well.

Let us now prove Clause~\eqref{niceprojection3}. Let  $p\in P_n$ and ${t}\LE_{\mathbb{S}_n\times \Col(\sigma_n,{<}\kappa_n)}{\psi_n(p)}$.
Putting $t=:\langle s, c\rangle$ we have $s\sle_n \varpi_n(p)$ and $c\supseteq j_n(F^{0p_n})(\alpha_{p_n})$.
On one hand, since $\varpi_n$ is a nice projection, $q:=p+s$ is a condition in $\mathbb{P}_n$. On the other hand,  there is a function $F$ and $\beta<\mu$  with $\dom(F)\in E_{n,\beta}$ and  $j_n(F)(\beta)=c$.\footnote{Here we use that $\Col(\sigma_n,{<}\kappa_n)^V=\Col(\sigma_n,{<}\kappa_n)^{M_{n}^*}$, where $M_n^*\cong \Ult(V,E_n\restriction\mu)$. } By possibly enlarging $a^{q_n}$ we may actually assume that $\beta=\alpha_q$ and also that  $\dom(F)=\pi_{\mc(a^{q_n}),\alpha_{q_n}}``A^{q_n}$. Let $r$ be the condition in $\mathbb{P}_n$ with the same entries as $q$ but with $F^{0r_n}:=F$. Clearly, $r\LE q \LE p$. Also, by the way $r$ is defined,
$\psi_n(r)=\langle \varpi_n(r),c\rangle=\langle \varpi_n(q),c\rangle=\langle s,c\rangle=t.
$

Note that if $u\in P_n$ is such that $u\LE p$ and  $\psi_n(u)=t$, then $u\leq r$. Altogether, $r=p+t$, which yields Clause~\eqref{niceprojection3}.\footnote{The $+$ here is regarded with respect to the map $\psi_n$.} Finally, for Clause~\eqref{theprojection} one argues in the same lines as in Lemma~\ref{C3forMoti}.

(3) Let $D\s\mathbb{P}^{\psi_n}_n$ be a directed set of size $<\kappa_n$. Then, $\psi_n[D]=\{\langle s, c\rangle\}$ for some $\langle s,c\rangle\in\mathbb{S}_n\times\Col(\sigma_n,{<}\kappa_n)$.
Thus,  for each  $p\in D$, $j_n(F^{0p_n})(\alpha_{p_n})=c$. Arguing as usual, let $a\in[\lambda]^{<\kappa_n}$ be such that both $a\cap \mu$ and $a$ have $\leq_{E_n}$-greatest elements $\alpha$ and $\beta$, respectively, and $a\supseteq \bigcup_{p\in D} a^{p_n}$. Then, for each $p,q\in D$, 
$B_{p,q}:=\{\nu<\kappa_n\mid F^{0p_n}(\pi_{\alpha,\alpha_{p_n}}(\nu))=F^{0q_n}(\pi_{\alpha,\alpha_{q_n}}(\nu))\}\in E_{n,\alpha}$ and, by $\kappa_n$-completedness of $E_{n,\alpha}$,
$B:=\bigcap\{B_{p,q}\mid p,q\in D\}\in E_{n,\alpha}$.

Set $A:=\pi^{-1}_{\beta,\alpha}``B$. By shrinking $A$ if necessary, we may further assume  $\pi_{\beta,\mc(a^p_n)}``A\s A^{p_n}$
for each $p\in D$. Since $\psi_n\upharpoonright D$ is constant the map  $\varpi_n\colon p\mapsto \langle (\rho^{p}_k, h^{0p}_k, h^{1p}_k, h^{2p}_k)\mid k<n\rangle$ is  so. Let $\langle(\rho_k, h^{0}_k, h^1_k, h^2_k)\mid k<n\rangle$ be such constant value. For each $k<\omega$, set $f_k:=\bigcup_{p\in D}f^{p}_k$ and $F^0$ be such that $\dom(F^0)=B$ and $F^0(\nu):=F^{0p}(\pi_{\alpha,\alpha_{p_n}}(\nu))$ for some $p\in D$.

Observe that $\{\langle F^{1p}_n, F^{2p}_n\rangle\mid p\in D\}$ forms a directed subset of $\mathbb{C}_n$ of size $<\kappa_n$ (Lemma~\ref{MotisCollapsesmu}). Using the $\kappa_n$-directed-closedness of $\mathbb{C}_n$  
we may let $\langle F^1, F^2\rangle\in C_n$ be a $\sqsubseteq$-lower bound. Actually, by using the moreover clause of  Lemma~\ref{MotisCollapsesmu} we may assume that $\dom(F^1)=\dom(F^2)\in E_{n,\alpha}$. Thus, by shrinking $A$ and $B$ if necessary we may assume  $\dom(F^1)=\dom(F^2)=B$.

Define $p^*:=\langle p^*_k\mid k<\omega\rangle$ as follows:
$$
p^*_k:=
\begin{cases}
(f_k, \rho_k, h^0_k, h^1_k, h^2_k), & \text{if $k<n$};\\
(a, A, f_k, F^0,F^1, F^2), & \text{if $k=n$};\\
(a_k, A_k,f_k, F^{0}_k, F^1_k, F^2_k), & \text{if $k>n$},
\end{cases}
$$
where $(a_k, A_k,f_k, F^{0}_k, F^1_k, F^2_k)$ is constructed as described in Lemma~\ref{C5forMoti}.  Clearly, $p^*\in Q_{n0}$ and it gives a $\LE^{\psi_n}$-lower bound for $D$.

(4) By Item~\eqref{keysubclaimprojection} of this lemma, $(\mathbb{T}_n\downarrow \psi_n(p))\times (\mathbb{P}^{\psi_n}_n\downarrow p)$ projects onto  $\mathbb{P}_n\downarrow p$.\footnote{See Definition~\ref{niceprojection}\eqref{theprojection}.} Actually both posets are easily seen to be  isomorphic.
\end{proof}

\begin{lemma}\label{GCHmoti} Assume $\gch$. Let $n<\omega$.
\begin{enumerate}
\item $\mathbb{P}_n$ is $\mu^+$-Linked;
\item $\mathbb{P}_n$ forces $\ch_\theta$ for any cardinal $\theta\ge\sigma_n$;
\item $\mathbb P_n^{\varpi_n}$ preserves the $\gch$.
\end{enumerate}
\end{lemma}
\begin{proof} (1) By Definition~\ref{cfunctionforMoti}, Lemma~\ref{C5forMoti} and the fact that $|H_\mu|=\mu$.

(2) As $\mathbb P_n$ has size $\le\mu^+$, Clause~(1) together with a counting-of-nice-names argument implies that
 $2^\theta=\theta^+$ for any cardinal $\theta\ge\mu^+$.
By Lemma~\ref{CardinalConfigurationMotisModel}, in any generic extension by $\mathbb P_n$, $|\mu|=\cf(\mu)=\kappa_n=(\sigma_{n})^+$.
It thus left to verify that $\mathbb P_n$ forces $2^\theta=\theta^+$ for $\theta\in\{\sigma_n,\kappa_n\}$.

$\br$ By Clauses \eqref{keysubclaim1}, \eqref{keysubclaim2} and \eqref{keysubclaim3} of Lemma~\ref{keysubclaim},
together with Easton's lemma, $\mathbb{P}_n$ forces $\ch_{\sigma_n}$ if and only if $\mathbb{T}_n$ forces $\ch_{\sigma_n}$.
By Clause~\eqref{keysubclaim1} of Lemma~\ref{keysubclaim}, $\mathbb{T}_n$ is a $\kappa_n$-cc poset of size $\kappa_n$,
so, the number of  $\mathbb{T}_n$-nice names for subsets of $\sigma_n$ is at most $\kappa_n^{<\kappa_n}=\kappa_n=\sigma_n^+$, as wanted.

$\br$ The number of $\mathbb{P}_n$-nice names for subsets of $\kappa_n$ is $((\mu^+)^\mu)^{\kappa_n}=\mu^{+}$, and hence $\ch_{\kappa_n}$ is forced by $\mathbb P_n$.

(3) By Lemma~\ref{C4forMoti}, $\mathbb P_n^{\varpi_n}$ preserves $\gch$ below $\sigma_n$.
Also, since $\mathbb{P}_n\simeq (\mathbb{P}_n^{\varpi_n}\times\mathbb{S}_n)$  (Remark~\ref{Pnisomorphic}) and  $|\mathbb S_n|<\sigma_n$,
we infer from Item~(2) above that $\gch$ holds at cardinals $\ge\sigma_n$, as well.
\end{proof}

\section{Nice forking projections}\label{SectionNiceForking}

In this short section we introduce the notions of \emph{nice} and \emph{super nice forking projection}. 
These provide a natural (and necessary) strengthening of  the corresponding
key concept from Part~I of this series. The reader familiar with \cite[\S2]{PartII} may want to skip some of the details here and, instead,  consult the new concepts (Definitions~\ref{NiceForking}, \ref{SuperNiceForking}, \ref{typeexact}) and results (Lemma~\ref{Liftingsuitability}).

\begin{definition}[{\cite[\S4]{PartI}}]\label{forking}  Suppose that $(\mathbb P,\lh_{\mathbb P},c_{\mathbb P},\vec{\varpi})$ is a $(\Sigma,\vec{\mathbb{S}})$-Prikry forcing,\footnote{In \cite{PartI}, the following notion is formulated in terms of $\Sigma$-Prikry forcings. However the same notion is meaningful in the more general context of $(\Sigma,\vec{\mathbb{S}})$-Prikry forcings.}
$\mathbb A=(A,\unlhd)$ is a notion of forcing,
and $\lh_{\mathbb A}$ and $c_{\mathbb A}$ are functions with $\dom(\lh_{\mathbb A})=\dom(c_{\mathbb A})=A$. A pair of functions $(\pitchfork,\pi)$ is said to be a \emph{forking projection} from $(\mathbb A,\lh_{\mathbb A})$ to $(\mathbb P,\lh_{\mathbb P})$
iff all of the following hold:
\begin{enumerate}
\item\label{frk1}  $\pi$ is a projection from $\mathbb A$ onto $\mathbb P$, and $\lh_{\mathbb A}=\lh_{\mathbb P}\circ\pi$;
\item\label{frk0}  for all $a\in A$, $\fork{a}$ is an order-preserving function from $(\cone{\pi(a)},\le)$ to $(\conea{a},\unlhd)$;
\item\label{frk3} for all $p\in P$, $\{ a \in A\mid \pi(a)=p\}$ admits a greatest element, which we denote by $\myceil{p}{\mathbb A}$;
\item\label{frk4}  for all $n,m<\omega$ and $b\unlhd^{n+m} a$, $m(a,b)$ exists and satisfies:
 $$m(a,b)=\fork{a}(m(\pi(a),\pi(b)));$$
\item\label{frk5} for all $a\in A$ and $r\LE\pi(a)$,   $\pi(\fork{a}(r))=r$;
\item\label{frk6} for all $a\in A$ and $r\LE\pi(a)$, $a=\myceil{\pi(a)}{\mathbb A}$ iff $\fork{a}(r)=\myceil{r}{\mathbb A}$;
\item\label{frk7} for all $a\in A$, $a'\unlhd^0 a$ and $r\LE^0 \pi(a')$, $\fork{a'}(r)\unlhd\fork{a}(r)$.
\setcounter{condition}{\value{enumi}}
\end{enumerate}

The pair $(\pitchfork,\pi)$ is said to be a forking projection from
$(\mathbb A,\lh_{\mathbb A},c_{\mathbb A})$ to $(\mathbb P,\lh_{\mathbb P},c_{\mathbb P})$
iff, in addition to all of the above, the following holds:
\begin{enumerate}
\setcounter{enumi}{\value{condition}}
\item\label{frk2}  for all $a,a'\in A$, if $c_{\mathbb A}(a)=c_{\mathbb A}(a')$, then $c_{\mathbb P}(\pi(a))=c_{\mathbb P}(\pi(a'))$ and, for all $r\in P_0^{\pi(a)}\cap P_0^{\pi(a')}$, $\fork{a}(r)=\fork{a'}(r)$.
\end{enumerate}
\end{definition}

\begin{definition}\label{NiceForking}
A pair of functions $(\pitchfork,\pi)$ is said to be a \emph{nice forking projection} from $(\mathbb A,\lh_{\mathbb A},\vec{\varsigma})$ to $(\mathbb P,\lh_{\mathbb P},\vec{\varpi})$
iff all of the following hold:
\begin{enumerate}[label=(\alph*)]
\setcounter{enumi}{0}
\item\label{frknew0} $(\pitchfork,\pi)$ is a forking projection from $(\mathbb{A},\ell_\mathbb{A})$ to $(\mathbb{P},\ell_\mathbb{P})$;
\item \label{frknew2} $\vec\varsigma=\vec\varpi\bullet\pi$, that is, $\varsigma_n=\varpi_n\circ \pi$ for all $n<\omega$. Also, for each $n$,  $\varsigma_n$ is a nice projection from $\mathbb{A}_{\geq n}$ to $\mathbb{S}_n$, and for each $k\geq n$, $\varsigma_n\restriction\mathbb{A}_k$ is again a nice projection.
\end{enumerate}
The pair $(\pitchfork,\pi)$ is said to be a  \emph{nice forking projection} from $(\mathbb A,\lh_{\mathbb A},c_\mathbb{A},\vec{\varsigma})$ to $(\mathbb P,\lh_{\mathbb P},c_\mathbb{P},\vec{\varpi})$ if, in addition, Clause~\eqref{frk2} of Definition~\ref{forking} is satisfied.
 \end{definition}

  \begin{remark}
If $(\mathbb{P},\ell_\mathbb{P},c_\mathbb{P})$ is a $\Sigma$-Prikry forcing then a pair of maps $(\pitchfork,\pi)$ is a forking projection from $(\mathbb{P},\ell_\mathbb{P})$ to $(\mathbb{A},\ell_\mathbb{A})$ iff it is a nice forking projection from $(\mathbb{P},\ell_\mathbb{P},\vec{\varpi})$ to $(\mathbb{A},\ell_\mathbb{A},\vec{\varsigma})$. In a similar vein, the same applies to forking projections from $(\mathbb A,\lh_{\mathbb A},c_\mathbb{A},\vec{\varsigma})$ to $(\mathbb P,\lh_{\mathbb P},c_\mathbb{P},\vec{\varpi})$.
 \end{remark}

As we will see, most of the theory of iterations of $(\Sigma,\vec{\mathbb{S}})$-Prikry forcings can be developed starting from the concept of nice forking projection. Nonetheless, to be successful at limit stages, one needs nice forking projections yielding canonical  witnesses for niceness of the $\varsigma_n$'s. Roughly speaking, we want that whenever $p'$ is a witness for niceness of some $\varpi_n$ then there is a witness $a'$ for niceness of $\varsigma_n$ which \emph{lifts} $p'$. This leads to the concept of \emph{super nice forking projection} that we now turn to introduce:
\begin{definition}\label{SuperNiceForking}
A nice forking projection $(\pitchfork,\pi)$ from   $(\mathbb A,\lh_{\mathbb A},\vec{\varsigma})$ to $(\mathbb P,\lh_{\mathbb P},\vec{\varpi})$ is called \emph{super nice} if for each $n<\omega$ the following property holds:

 Let $a,a'\in A_{\geq n}$  and $s\in S_n$  such that $a'\unlhd a+s$.  Then, for each $p^*\in P_{\geq n}$ such that $p^*\LE^{\varpi_n} \pi(a)$ and $\pi(a')=p^*+\varsigma_n(a')$, there is $a^*\unlhd^{\varsigma_n} a$ with $$\pi(a^*)=p^*\;\;\text{and}\;\;a'=a^*+\varsigma_n(a').$$
The notion of \emph{super nice forking projection} from $(\mathbb A,\lh_{\mathbb A},c_\mathbb{A},\vec{\varsigma})$ to $(\mathbb P,\lh_{\mathbb P},c_\mathbb{P},\vec{\varpi})$ is defined in a similar fashion.
\end{definition}
\begin{remark}
The notion of super nice forking projection will not be needed in the current section. Its importance will only become clear in Section~\ref{Iteration}, where we present our iteration scheme. The purpose for drawing this distinction is to emphasize that at successor stages of the iteration \emph{nice} forking projection suffice to establishing $(\Sigma,\vec{\mathbb{S}})$-Prikryness of the different iterates. In other words, the iteration machinery introduced in \cite[\S3]{PartII} is essentially successful. 
 Nevertheless, \emph{super nice} forking projection will play a crucial role in the  verifications at limit stages. For an example, see  Claim~\ref{PropertiesForking}, p.~\pageref{supernicenessinaction}. 
\end{remark}

 Next we show that if $(\pitchfork,\pi)$ is a forking projection (not necessarily nice) and $\vec{\varsigma}=\pi\circ \vec{\varpi}$ then $\varsigma_n$ satisfies Definition~\ref{niceprojection}\eqref{niceprojection3} for each $n<\omega$.

 \begin{lemma} \label{pitchforkexact}
 Let $(\pitchfork, \pi)$ be a forking projection from $(\mathbb A,\lh_{\mathbb A})$ to $(\mathbb P,\lh_{\mathbb P})$ and suppose that $\vec{\varsigma}=\pi\circ\vec{\varpi}$. Then, for all $a\in A$, $n\leq \ell_\mathbb{A}(a)$ and $s\sle_n \varsigma_n(a)$, $$a+s=\fork{a}(\pi(a)+s).$$
 \end{lemma}
 \begin{proof}
Combining Clauses~\eqref{frk0} and \eqref{frk5} of Definition~\ref{forking} with  Clause~\ref{frknew2} of Definition~\ref{NiceForking} it follows that $\fork{a}(\pi(a)+s)\unlhd a$  and $\varsigma_n(\fork{a}(\pi(a)+s))=s.$

Let $b\in A$ such that $b\unlhd^0 a$ and $\varsigma_n(b)\sle_n s$. Then, $\pi(b)\LE \pi(a)+s$. By \cite[Lemma~2.17]{PartII}, $b=\fork{b}(\pi(b))$, hence Clauses~\eqref{frk0} and \eqref{frk7} of Definition~\ref{forking} yield
$b\unlhd^0 \fork{a}(\pi(b))\unlhd^0\fork{a}(\pi(a)+s)$.

We are not done yet with establishing that $a+s=\fork{a}(\pi(a)+s)$, as we have just dealt with $b\unlhd^0 a$. However we can further argue as follows. Let $b\unlhd a$ be with $\varsigma_n(b)\sle_n s$. Put, $b':=\fork{w(a,b)}(w(\pi(a),\pi(b))+\varsigma_n(b))$. It is easy to check that $b\unlhd b'\unlhd^0 a$ and that $\varsigma_n(b')=\varsigma_n(b)\sle_n s$. Hence, applying the previous argument we arrive at $b\unlhd b'\unlhd^0 \fork{a}(\pi(a)+s)$. 
 \end{proof}

 In \cite[\S2]{PartII}, we drew a map of connections between $\Sigma$-Prikry forcings and forking projection,
demonstrating that this notion is crucial to define a viable iteration scheme for $\Sigma$-Prikry posets. However, to be successful in iterating $\Sigma$-Prikry forcings, forking projections need to be accompanied with \emph{types}, which are key to establish the CPP and property $\mathcal{D}$ for $(\mathbb{A},\ell_\mathbb{A})$.

 \begin{definition}[{\cite[\S2]{PartII}}]\label{type}
A \emph{type over a forking projection $(\pitchfork,\pi)$}
is a map $\tp\colon A\rightarrow{}^{<\mu}\omega$ having the following properties:
\begin{enumerate}
\item\label{type1} for each $a\in A$, either $\dom(\tp(a))=\alpha+1$ for some $\alpha<\mu$, in which case we define $\mtp(a):=\tp(a)(\alpha)$,
or $\tp(a)$ is empty, in which case we define $\mtp(a):=0$;
\item\label{type2} for all $a,b\in A$ with $b\unlhd a$, $\dom(\tp(a))\leq \dom(\tp(b))$ and for each $i\in\dom(\tp(a))$, $\tp(b)(i)\leq \tp(a)(i)$;
\item\label{type3} for all $a\in A$ and $q\le \pi(a)$, $\dom(\tp(\fork{a}(q)))=\dom(\tp(a))$;
\item\label{type6} for all $a\in A$,  $\tp(a)=\emptyset$ iff $a=\myceil{\pi(a)}{\mathbb{A}}$; 
\item\label{type4} for all $a\in A$ and  $\alpha\in \mu\setminus \dom(\tp(a))$, there exists a \emph{stretch of $a$ to $\alpha$}, denoted $a{}^{\curvearrowright\alpha}$,
and satisfying the following:
\begin{enumerate}[label=(\alph*)]
\item $a{}^{\curvearrowright\alpha}\unlhd^\pi a$;
\item  $\dom(\tp(a{}^{\curvearrowright\alpha}))=\alpha+1$;
\item  $\tp(a{}^{\curvearrowright\alpha})(i)\leq \mtp(a)$ whenever $\dom(\tp(a))\le i\le\alpha$;
\end{enumerate}
\item \label{newstretch} for all $a,b\in A$ with $\dom(\tp(a))=\dom(\tp(b))$, for every $\alpha\in \mu\setminus \dom(\tp(a))$,
if $b\unlhd a$, then
$b{}^{\curvearrowright\alpha}\unlhd a{}^{\curvearrowright\alpha}$;
\item \label{type5} For each $n<\omega$, the poset $\z{\mathbb{A}}_n$ is dense in $\mathbb{A}_n$, where $\z{\mathbb{A}}_n:=(\z{A}_n,\unlhd)$ and $\z{A}_n:=\{a\in A_n\mid \pi(a)\in\z{P}_n\ \&\ \mtp(a)=0\}$.
\end{enumerate}
 \end{definition}
 \begin{remark}\label{RemarkType} Note that Clauses \eqref{type2} and \eqref{type3} imply that for all $m,n<\omega$, $a\in\z A_m$
 and $q \le\pi(a)$, if $q\in\z P_n$ then $\fork{a}(q)\in\z A_n$.
 \end{remark}

 In the more general context of $(\Sigma,\vec{\mathbb{S}})$-Prikry forcings
 -- where the pair $(\pitchfork,\pi)$ needs to be a nice forking projection -- we need to  require a bit more:

 \begin{definition}\label{typeexact}
 A \emph{nice type over a nice forking projection $(\pitchfork,\pi)$}
is a type over $(\pitchfork,\pi)$ which moreover satisfies the following:
\begin{enumerate}
\setcounter{enumi}{7}
\item\label{ringismoredense} For each $n<\omega$, the poset $\z{\mathbb{A}}^{\varsigma_n}_n$ is dense in $\mathbb{A}^{\varsigma_n}_n$.\footnote{ Here $\z{\mathbb{A}}_n$ is the forcing from Definition~\ref{type}\eqref{type5}.}
\end{enumerate}
 \end{definition}

 \begin{remark}
If $(\mathbb{P},\ell_\mathbb{P}, c_\mathbb{P})$ is $\Sigma$-Prikry    then any forking projection $(\pitchfork,\pi)$ is nice and $\z{\mathbb{A}}^{\varsigma_n}_n=\z{\mathbb{A}}_n$ for all $n<\omega$. In particular, any type over $(\pitchfork,\pi)$ is nice.
 \end{remark}

We now turn to collect sufficient conditions --- assuming the existence of a nice forking projection $(\pitchfork,\pi)$ from $(\mathbb{A},\ell_\mathbb{A},c_\mathbb{A},\vec{\varsigma})$ to
$(\mathbb{P},\ell_\mathbb{P},c_\mathbb{P},\vec{\varpi})$ --- for $(\mathbb{A},\ell_\mathbb{A},c_\mathbb{A},\vec{\varsigma})$ to be $(\Sigma,\vec{\mathbb{S}})$-Prikry on its own,
and then address the problem of ensuring that the $\mathbb A_n$'s be suitable for reflection.
This study will be needed in Section~\ref{killingone}, most notably, in the proof of Theorem~\ref{AisweaklySigmaPrikry}.

\begin{setup}\label{SetupExactForking}Throughout the rest of this section, we suppose that:
\begin{itemize}
\item $\mathbb P=(P,\le)$ is a notion of forcing with a greatest element $\one_{\mathbb P}$;
\item $\mathbb A=(A,\unlhd)$ is a notion of forcing with a greatest element $\one_{\mathbb A}$;
\item $\Sigma=\langle \sigma_n\mid n<\omega\rangle$ is a non-decreasing sequence of regular uncountable cardinals,
converging to some cardinal $\kappa$, and  $\mu$ is a cardinal such that $\one_{\mathbb P}\Vdash_{\mathbb P}\check\mu=\check\kappa^+$;
\item $\vec{\mathbb S}=\langle \mathbb S_n\mid n<\omega\rangle$ is a sequence of notions of forcing, $\mathbb S_n=(S_n,\SLE_n)$,
with $|S_n|< \sigma_n$;
\item\label{Czeta} $\lh_{\mathbb P}$, $c_{\mathbb P}$ and $\vec{\varpi}=\langle \varpi_n\mid n<\omega\rangle$ are witnesses for $(\mathbb P,\lh_{\mathbb P},c_{\mathbb P},\vec{\varpi})$ being $(\Sigma,\vec{\mathbb{S}})$-Prikry;
\item $\lh_{\mathbb A}$ and $c_{\mathbb A}$ are functions with $\dom(\lh_{\mathbb A})=\dom(c_{\mathbb A})=A$,
and $\vec{\varsigma}=\langle\varsigma_n\mid n<\omega\rangle$ is a sequence of functions.
\item $(\pitchfork,\pi)$ is a nice  forking projection from $(\mathbb A,\lh_{\mathbb A},c_{\mathbb A},\vec{\varsigma})$ to $(\mathbb P,\lh_{\mathbb P},c_{\mathbb P},\vec{\varpi})$.
\end{itemize}
\end{setup}

\begin{theorem}\label{forkingSigmaPrikryLight}
Under the assumptions of Setup~\ref{SetupExactForking}, $(\mathbb A,\lh_{\mathbb A}, c_{\mathbb A},\vec{\varsigma})$ satisfies all the clauses of Definition~\ref{SigmaPrikry}, with the possible exception of \eqref{c2}, \eqref{c6} and
\eqref{moreclosedness}.
Moreover,  if $\one_{\mathbb P}\forces_{\mathbb P}``\check\kappa\text{ is singular}"$, then $\one_{\mathbb A}\forces_{\mathbb{A}}\check{\mu}=\check\kappa^+$.

\end{theorem}
\begin{proof}
Clauses~\eqref{graded} and \eqref{c1} follow respectively  from \cite[Lemmas  4.3 and 4.7]{PartI}. Clause~\eqref{c5} holds by virtue of Clause~\eqref{frk4} of Definition~\ref{forking}. Clauses~
\eqref{csize} and \eqref{itsaprojection} are respectively proved in \cite[Lemmas~4.7 and 4.10]{PartI}, and Clause~\eqref{PnprojectstoSn} follows from Clause~\ref{frknew2} of Definition~\ref{NiceForking}. Finally, Lemma~\ref{l14}\eqref{l14(2)} yields the moreover part. For more details, see \cite[Corollary~4.13]{PartI}.
\end{proof}

Next, we give sufficient conditions in order for $(\mathbb{A},\ell_\mathbb{A})$ to satisfy the CPP. In Part~II of this series we proved that CPP  follows from property $\mathcal{D}$ of $(\mathbb{A},\ell_\mathbb{A})$:

\begin{lemma}[{\cite[Lemma~2.21]{PartII}}]\label{lemmaforCPP}
Suppose that $(\mathbb{A},\ell_\mathbb{A})$ has property $\mathcal{D}$. Then it has the $\CPP$.
\end{lemma}
Therefore, everything amounts to find sufficient conditions for $(\mathbb{A},\ell_\mathbb{A})$ to satisfy property $\mathcal{D}$. The following  concept will be useful on that respect:
 \begin{definition}[Weak Mixing property]\label{mixingproperty}
The forking projection $(\pitchfork,\pi)$ is said to have the \emph{weak mixing property} iff
it admits a type $\tp$ satisfying that for every $n<\omega$, 
$a\in A$, $\vec r$, and $p'\LE^0 \pi(a)$,
and for every function $g:W_n(\pi(a))\rightarrow \mathbb{A}\downarrow a$, if there exists an ordinal $\iota$ such that all of the following hold:
\begin{enumerate}
\item\label{Mixing1} $\vec r=\langle r_\xi \mid \xi<\chi\rangle$ is a good enumeration of $W_n(\pi(a))$;
\item\label{Mixing2}  $\langle \pi(g(r_\xi))\mid \xi<\chi\rangle$ is diagonalizable with respect to $\vec r$, as witnessed by $p'$;\footnote{Recall Definition~\ref{DiagonalizabilityGame}.}
\item\label{Mixing3} for every $\xi<\chi$:\footnote{The role of the $\iota$ is to keep track of the support when we apply the weak mixing lemma in the iteration (see, e.g. \cite[Lemma~3.10]{PartII}).}
\begin{itemize}
\item if $\xi<\iota$, then $\dom(\tp(g(r_\xi))=0$;
\item if $\xi=\iota$, then $\dom(\tp(g(r_\xi))\geq 1$;
\item if $\xi>\iota$, then $(\sup_{\eta<\xi}\dom(\tp(g(r_\eta)))+1<\dom(\tp(g(r_\xi))$;
\end{itemize}
\item\label{Mixing4} for all $\xi\in(\iota,\chi)$ and $i\in[\dom(\tp(a)),\sup_{\eta<\xi}\dom(\tp(g(r_\eta)))]$,
$$\tp(g(r_\xi))(i)\leq\mtp(a),$$

\item\label{Mixing5} $\sup_{\xi<\chi} \mtp(g(r_\xi))<\omega$,
\end{enumerate}
then there exists $b\unlhd^0 a$ with $\pi(b)=p'$ 
such that, for all  $q'\in W_n(p')$, $$\fork{b}(q')\unlhd^0 g(w(\pi(a),q')).$$
\end{definition}
\begin{remark}
We would like to emphasize that the above notion make sense even when both $(\pitchfork,\pi)$ and $\tp$ are not nice. This is simply because the above clauses do not involve the maps $\varsigma_n$'s nor the forcings $\z{\mathbb{A}}_n^{\varsigma_n}$.
\end{remark}
As shown in \cite[\S2]{PartII}, the weak mixing property is the key to ensure that $(\mathbb{A},\ell_\mathbb{A})$ has property $\mathcal{D}$. In this respect, the following lemma gathers the results proved in Lemma~2.27 and Corollary~2.28 of \cite{PartII}:
\begin{lemma}\label{propDsuccessor}\label{PropDtoCPP}
Suppose that $(\pitchfork,\pi)$ has the weak mixing property and that $(\mathbb{P},\lh_\mathbb{P})$ has property $\mathcal{D}$. Then $(\mathbb{A},\ell_\mathbb{A})$ has property $\mathcal{D}$, as well.

In particular, if $(\mathbb{P},\lh_\mathbb{P})$ has property $\mathcal{D}$  and $(\pitchfork,\pi)$ has the weak mixing property, then $(\mathbb{A},\lh_\mathbb{A})$ has both property $\mathcal{D}$ and the $\CPP$.
\end{lemma}

We still need to verify  Clause~\eqref{c2} and \eqref{moreclosedness} of Definition~\ref{SigmaPrikry}. Arguing similarly to \cite[Lemma~2.29]{PartII} we can prove the following:
\begin{lemma}\label{forkinganddirectedclosure}
Suppose that $({\pitchfork},{\pi})$ is as in Setup~\ref{SetupExactForking} or,
just a pair of maps satisfying Clauses~\eqref{frk1}, \eqref{frk0}, \eqref{frk5} and \eqref{frk7} of Definition~\ref{forking}.

Let $n<\omega$.
If $(\pitchfork,\pi)$ admits a  type, and $\z{\mathbb{A}}_n$ is defined according to the last clause of Definition~\ref{type},
if $\z{\mathbb A}_n^\pi$ is $\aleph_1$-directed-closed, then so is $\z{\mathbb{A}}_n$. Similarly, if $\z{\mathbb A}_n^\pi$ is $\sigma_n$-directed-closed, then so is $\z{\mathbb{A}}_n^{\varsigma_n}$.

If in addition $(\pitchfork,\pi)$ admits a nice  type then $(\mathbb{A},\ell_\mathbb{A},c_\mathbb{A},\vec{\varsigma})$ satisfies Clauses~\eqref{c2} and \eqref{moreclosedness} of Definition~\ref{SigmaPrikry}.
\end{lemma}
Additionally, a routine verification gives the following:
\begin{lemma}\label{liftingcoherency}
Suppose that $(\pitchfork, \pi)$ is as in  Setup~\ref{SetupExactForking}. Then, if $\vec{\varpi}$ is a coherent sequence of nice projections then so is $\vec{\varsigma}$.\qed
\end{lemma}

We conclude this section by providing a sufficient condition for the posets $\mathbb A_n$'s to be suitable for reflection with respect to a sequence of cardinals for which the posets $\mathbb{P}_n$'s were so.

\begin{lemma}\label{Liftingsuitability}
Let $n$ be a positive integer. Assume:
\begin{enumerate}[label=(\roman*)]
\item \label{Liftingsuitability1} $\kappa_{n-1},\kappa_{n}$ are regular uncountable cardinals with $\kappa_{n-1}\leq \sigma_{n}<\kappa_{n}$;
\item \label{Liftingsuitability2} $(\mathbb{A},\ell_\mathbb{A},c_\mathbb{A},\vec{\varsigma})$ is $(\Sigma,\vec{\mathbb{S}})$-Prikry;
\item \label{Liftingsuitability3} $(\pitchfork,\pi)$ is a nice forking projection from  $(\mathbb{A},\ell_\mathbb{A},\vec{\varsigma})$ to $(\mathbb{P},\ell_\mathbb{P},\vec{\varpi})$;
\item\label{Liftingsuitability4}  $(\mathbb{P}_n,\mathbb{S}_n, \varpi_n)$ is suitable for reflection with respect to $\langle \sigma_{n-1},\kappa_{n-1},\kappa_n, \mu\rangle$;
\item \label{Liftingsuitability5} $\mathbb{S}_n\times \mathbb{A}_n^{\varsigma_n}$ forces $``|\mu|=\cf(\mu)=\kappa_n=(\kappa_{n-1})^{++}"$.
\end{enumerate}
Then $(\mathbb{A}_n,\mathbb{S}_n,\varsigma_n)$ is suitable for reflection with respect to $\langle \sigma_{n-1},\kappa_{n-1},\kappa_n, \mu\rangle$.
\end{lemma}
\begin{proof} Clauses \eqref{suitableforiteration0},
\eqref{suitableforiteration1} and  \eqref{suitableforiteration3}
 of Definition~\ref{suitableforiteration}
 hold by virtue of hypotheses, $(iv)$, $(ii)$-$(iii)$-$(iv)$ and $(iv)$, respectively.

Now let us address Clause~\eqref{suitableforiteration2}.
Given hypothesis $(v)$, we are left with verifying that
$\mathbb{A}_n$ forces $``|\mu|=\cf(\mu)=\kappa_n=(\kappa_{n-1})^{++}"$.
By Definition~\ref{niceprojection}\eqref{theprojection},
for every $a\in A_n$, $(\mathbb{S}_n\downarrow \varsigma_n(a))\times(\mathbb{A}_n^{\varsigma_n}\downarrow a)$ projects onto  $\mathbb{A}_n\downarrow a$.
In addition, by hypothesis $(iii)$, $\mathbb A_n$ projects onto $\mathbb P_n$.
Since both ends force $``|\mu|=\cf(\mu)=\kappa_n=(\kappa_{n-1})^{++}"$,
the same is true for $\mathbb{A}_n$.
\end{proof}

\section{Stationary Reflection and Killing a Fragile Stationary Set}\label{killingone}
In this section, we isolate a natural notion of a \emph{fragile set} and study two aspects of it.
In the first subsection, we prove that, given a $(\Sigma, \vec{\mathbb{S}})$-Prikry poset $\mathbb{P}$ and an $r^\star$-fragile stationary set $\dot{T}$, a tweaked version of Sharon's functor $\mathbb{A}(\cdot,\cdot)$ from \cite[\S4.1]{PartII}
yields a $(\Sigma, \vec{\mathbb{S}})$-Prikry poset $\mathbb A(\mathbb P,\dot{T})$ admitting a (super) nice forking projection to $\mathbb{P}$ and killing the stationarity of $\dot{T}$.
In the second subsection, we make the connection between fragile stationary sets,
suitability for reflection and non-reflecting stationary sets.
The two subsections can be read independently of each other.

\begin{setup}\label{setupkillingone}
As a setup for the whole section,
we  assume that $(\mathbb P,\lh,c,\vec{\varpi})$ is a given  $(\Sigma, \vec{\mathbb{S}})$-Prikry  forcing such that $(\mathbb{P},\ell)$ satisfies property $\mathcal{D}$.
Denote $\mathbb{P}=(P,\le)$, $\Sigma=\langle \sigma_n\mid n<\omega\rangle$,  $\vec{\varpi}=\langle\varpi_n\mid n<\omega\rangle$,
$\vec{\mathbb{S}}=\langle \mathbb{S}_n\mid n<\omega\rangle$. 
Also, define $\kappa$ and $\mu$ as in Definition~\ref{SigmaPrikry},
and assume that $\one_{\mathbb P}\forces_{\mathbb P}``\check\kappa\text{ is singular}"$ and that $\mu^{<\mu}=\mu$. For each $n<\omega$, we denote by $\z{\mathbb{P}}_n$ the countably-closed dense suposet of $\mathbb{P}_n$ given by Clause~\eqref{c2} of Definition~\ref{SigmaPrikry}. Recall that by virtue of  Clause~\eqref{moreclosedness}, $\z{\mathbb{P}}_n^{\varpi_n}$ is a $\sigma_n$-directed-closed dense subforcing of $\z{\mathbb{P}}_n$. We often refer to $\z{\mathbb{P}}_n$ as {\it the ring} of $\mathbb{P}_n$.
In addition, we will assume that $\vec{\varpi}$ is a coherence sequence of nice projections (see Definition~\ref{CoherentSystem}).
\end{setup}

The following concept is implicit in the proof of \cite[Theorem~11.1]{cfm}:

\begin{definition}\label{fragilestationary} Suppose $r^\star\in P$ forces that $\dot T$ is a $\mathbb P$-name for a stationary subset $T$ of $\mu$.
We say that $\dot T$ is \emph{$r^\star$-fragile} if, looking for each $n<\omega$ at $\dot{T}_n:=\{(\check\alpha,p)\mid (\alpha,p)\in \mu\times P_n\ \&\ p\forces_{\mathbb P}\check\alpha\in\dot T\}$,
then, for every $q\LE r^\star$, $q\forces_{\mathbb{P}_{\lh(q)}}``\dot{T}_{\lh(q)}\text{ is nonstationary}"$.
\end{definition}

\subsection{Killing one fragile set}\label{killingonesubsection}

Let $r^\star\in P$ and $\dot{T}$ be a $\mathbb{P}$-name for an $r^\star$-fragile stationary subset of $\mu$.
Let $I:=\omega\setminus\lh(r^\star)$.
By Definition~\ref{fragilestationary},
for all $q\LE r^\star$, 
$q\forces_{\mathbb{P}_{\lh(q)}}``\dot{T}_{\lh(q)}\text{ is nonstationary}"$.
Thus, for each $n\in I$, we may pick a $\mathbb P_n$-name $\dot{C}_n$ for a club subset of $\mu$ such that,
 for all $q\LE r^\star$, $$q\forces_{\mathbb P_{\ell(q)}}\dot T_{\ell(q)}\cap\dot C_{\ell(q)}=\emptyset.$$
Consider the following binary relation:
$$R:=\{(\alpha,q)\in\mu\times P\mid q\LE r^\star\ \&\ \forall r\LE q\,(r\forces_{\mathbb{P}_{\lh(r)}}\check\alpha\in \dot{C}_{\lh(r)})\},$$\label{DefinitionofR}
and define $\dot{T}^+:=\{(\check{\alpha},p)\mid (\alpha,p)\in \mu\times P \;\&\; p\forces_{\mathbb{P}_{\ell(p)}} \check{\alpha}\notin \dot{C}_{\ell(p)}\}.$

Arguing as in \cite[Lemma~4.6]{PartII} we have that $r^\star\forces_{\mathbb{P}}\dot{T}\s \dot{T}^+$ 
and $q\forces_\mathbb{P}\check\alpha\notin\dot{T}^+$ for all $(\alpha,q)\in R$. Also, if $(\alpha,q)\in R$ and $q'\LE q$ a routine verification shows that  $(\alpha,q')\in R$, as well

\smallskip

As in \cite[\S4.1]{PartII}, in this section we attempt to kill the stationarity of the bigger set $\dot{T}^+$ in place of $T$. 
The reason for this was explained in \cite{PartII}, but we briefly reproduce it for the reader's benefit:
for each $n<\omega$ let $\tau_n$ be the $\mathbb{P}_n$-name  $\{(\check{\alpha}, p)\in \dot{T}^+\mid \alpha \in \mu\; \&\; p\in P_n\}$. Intuitively speaking, $\tau_n$ is the \emph{trace} of $\dot{T}^+$ to a $\mathbb{P}_n$-name. The two key properties of $\tau_n$ are: 
\begin{itemize}
	\item $\tau_n\s \dot{T}^+_n$
	\item $p\forces_{\mathbb{P}_n} \tau_n=(\check\mu\setminus \dot{C}_n)$ for all $p\in P_n$.
\end{itemize}
These two features of $\tau_n$ were crucially used in \cite[Lemma~4.28]{PartII} when we verified the density of the \emph{ring poset} $(\z{\mathbb{P}}_\delta)_n$ in $(\mathbb{P}_\delta)_n$, for $\delta\in \acc(\mu^++1)$.

\smallskip

The next upcoming lemma  is a generalization of \cite[Claim 5.6.1]{PartI}:

\begin{lemma}\label{Runbdeddirectmore}
For all $p\leq r^\star$ and $\gamma<\mu$, there is an ordinal $\bar{\gamma}\in(\gamma,\mu)$ and $\bar{p}\LE^{\vec{\varpi}} p$, such that $(\bar{\gamma}, \bar{p})\in R$.
\end{lemma}
\begin{proof}
We begin by proving the following auxiliary claim:
\begin{claim}\label{ClaimR}
For all $p\LE r^\star$  and $\gamma<\mu$, there is an ordinal $\bar{\gamma}\in(\gamma,\mu)$ and $\bar{p}\LE^{\vec{\varpi}} p$, such that for all $q\LE \bar{p}$, $q\forces_{\mathbb{P}_{\lh(q)}} \dot{C}_{\ell(q)}\cap (\check{\gamma},\check{\bar{\gamma}})\neq \check{\emptyset}$.
\end{claim}
\begin{proof}
Let $\gamma$ and $p$ be as above. Set $\ell:=\lh(p)$,  $s:=\varpi_{\ell}(p)$ and put $$D_{p,\gamma}:=\{q\in P\mid q\LE p\;\&\;\exists\gamma'>\gamma\;(q\forces_{\mathbb{P}_{\lh(q)}} \check{\gamma}'\in \dot{C}_{\ell(q)})\}.$$
Clearly, $D_{p,\gamma}$ is $0$-open. Hence appealing to Clause~\eqref{Prikry3} of Lemma~\ref{Prikry} we  obtain a condition $\bar{p}\LE^{\vec{\varpi}} p$ with the property that the set $$U_{D_{p,\gamma}}:=\{t\sle_{\ell} s\mid \forall n<\omega~(P^{\bar{p}+t}_n\s D_{p,\gamma}\;\text{or}\; P^{\bar{p}+t}_n\cap D_{p,\gamma}=\emptyset)\}$$ is dense in  $\mathbb{S}_\ell\downarrow s$. 

We note that $U_{D_{p,\gamma}}=\{t\sle_{\ell} s\mid \forall n<\omega~ (P^{\bar{p}+t}_n\s D_{p,\gamma})\}$: Fix $t\in U_{D_{p,\gamma}}$ and $n<\omega.$ Let $q\leq^n\bar{p}+t$ be arbitrary. Since $q\leq \bar{p}\leq p\leq r^\star$ we have that $q\Vdash_{\mathbb{P}_{\ell(q)}}``\dot{C}_{\ell(q)}\text{ is a club in }\mu"$. Thus, there is $q'\leq^0 q$ and $\gamma'\in (\gamma,\mu)$ such that $q'\Vdash_{\mathbb{P}_{\ell(q)}}\gamma'\in \dot{C}_{\ell(q)}$. By definition, $q'\in D_{p,\gamma}\cap P^{\bar{p}+t}_n$, so that this latter intersection is non-empty. Since we picked $t\in U_{D_{p,\gamma}}$ it follows that $P^{\bar{p}+t}_n\subseteq D_{p,\gamma}$. This proves the above-displayed equality. 

For each condition $r\in W(\bar{p}+t)$  pick some ordinal $\gamma_r\in (\gamma,\mu)$ witnessing that $r\in D_{p,\gamma}$,  and put $$\bar{\gamma}:=\sup\{\gamma_r\mid r\in W(\bar{p}+t)\;\&\; t\sle_\ell s\}+1.$$
Combining Clauses~\ref{Cbeta} and \eqref{csize} of Definition~\ref{SigmaPrikry} we infer that $\bar{\gamma}<\mu$.

\smallskip

We claim that $\bar{p}$ is as desired.  Otherwise, let $q\leq \bar{p}$ forcing the negation of the claim. By virtue of Clause~\eqref{PnprojectstoSn}  of Definition~\ref{SigmaPrikry}, $\varpi_\ell$ is a nice projection from $\mathbb{P}_{\geq \ell}$ to $\mathbb{S}_\ell$, hence  Definition~\ref{niceprojection}\eqref{theprojection} applied to this map yields $q=\bar{q}+\varpi_\ell(q)$, for some $\bar{q}\LE^{\varpi_\ell} \bar{p}$. Putting $t:=\varpi_\ell(q)$,  it is clear that $t\sle_\ell s$. By extending $t$ if necessary, we may freely assume that $t\in U_{D_{p,\gamma}}$.

 On the other hand, $q\LE^0 w(\bar{p},q)$, hence Lemma~\ref{AdditionalAssumption3} and Clause~\eqref{AdditionalAssumption5} of Lemma~\ref{LemmaAdditionalAssumptions} yield
 $q\LE^0 w(\bar{p},\bar{q}+t)+t=w(\bar{p},\bar{q})+t \in W(\bar{p}+t).$
  This clearly yields a contradiction with our choice of $q$.
\end{proof}
Now we take advantage of the previous claim to prove the lemma. So, let $p\LE r^\star$ and $\gamma<\mu$. Applying the above claim inductively, we find a $\LE^{\vec{\varpi}}$-decreasing sequence $\langle p_n\mid n<\omega\rangle$ and an increasing sequence of ordinals below $\mu$, $\langle \gamma_n\mid n<\omega\rangle$,   such that $p_0:=p$, $\gamma_0:=\gamma$, and such that for every $n<\omega$, the pair $(p_{n+1},\gamma_{n+1})$ witnesses 
the conclusion of Claim~\ref{ClaimR} when putting $(p,\gamma):=(p_n,\gamma_n)$. Moreover, density of $\z{\mathbb{P}}^{\varpi_\ell}_\ell$ in ${\mathbb{P}}^{\varpi_\ell}_\ell$
(Definition~\ref{SigmaPrikry}\eqref{moreclosedness}) enable us to  assume that the $p_n\in \z{\mathbb{P}}^{\varpi_\ell}_\ell\downarrow p$, In particular, there is  $\bar{p}$ a $\LE^{\vec{\varpi}}$-lower bound for the sequence. Setting $\bar{\gamma}:=\sup_{n<\omega}\gamma_n$ we have that $(\bar{\gamma},\bar{p})\in R$.
\end{proof}

\subsubsection{Definition of the functor and basic properties}

\begin{definition}\label{labeled-p-tree} Let $p$ be a condition in  $\mathbb P$.
A \emph{labeled $\langle p,\vec{\mathbb{S}}\rangle$-tree} is a function $S\colon \dom(S)\rightarrow[\mu]^{<\mu}$, where $$\dom(S)=\{(q,t)\mid q\in W(p)\;\&\; t\in\bigcup_{\ell(p)\leq n\leq\ell(q)}\,\mathbb{S}_{n}\downarrow \varpi_n(q)\},$$ 
and such that for all  $(q,t)\in \dom(S)$ the following hold:
\begin{enumerate}
\item\label{C1ptree} $S(q,t)$ is a closed bounded subset of $\mu$;
\item\label{C2ptree} $S(q',t')\supseteq S(q,t)$ whenever $q'+t'\LE q+t$;
\item\label{C3ptree}  $q+t\forces_{\mathbb P} S(q,t)\cap\dot{T}^+=\emptyset$;
\item\label{d162}
\label{C4ptree} there is  $m<\omega$ such that
for any $q\in W(p)$ and  $(q',t')\in\dom(S)$ with $q'\LE q$, if $S(q',t')\neq\emptyset$ and $\lh(q)\ge\lh(p)+m$, then $(\max(S(q',t')),q)\in R$. The least such $m$ is denoted by $m(S)$.
\end{enumerate}
\end{definition}

\begin{remark}\label{incompatiblewithrstar}
By Clause~\eqref{C4ptree} and the definition of $R$, for any $(q',t'), (q,t)$ in  $\dom(S)$ with $q'+t'\LE q+t$,
if $q$ is incompatible with $r^\star$ then $S(q',t')=\emptyset$.
\end{remark}

\begin{definition}\label{strategy}
For $p\in P$, we say that $\vec S=\langle S_i\mid i\leq\alpha\rangle$ is a \emph{$\langle p,\vec{\mathbb{S}}\rangle$-strategy} iff all of the following hold:
\begin{enumerate}
\item\label{C1pstrategy} $\alpha<\mu$;
\item\label{i3}
\label{C2pstrategy}  for all $i\leq\alpha$, $S_i$ is a labeled $\langle p,\vec{\mathbb{S}}\rangle$-tree;

\item\label{C3pstrategy} for every $i<\alpha$ and $(q,t)\in \dom(S_i), S_i(q,t)\sqsubseteq S_{i+1}(q,t);$
\item\label{C4pstrategy} for every $i<\alpha$ and pairs $(q,t),(q',t')$ in $\dom(S_i)$ with $q'+t'\LE q+t$,  \linebreak $S_{i+1}(q,t)\setminus S_{i}(q,t)\sqsubseteq S_{i+1}(q',t')\setminus S_i(q',t');$
\item\label{C5pstrategy} for every  limit $i\leq\alpha$ and $(q,t)\in \dom(S_i)$,  $S_i(q,t)$ is the ordinal closure of ${\bigcup_{j<i} S_j(q,t)}$.
\end{enumerate}
\end{definition}

\begin{definition}\label{SharonNew}
Let $\mathbb{A}(\mathbb{P},\vec{\mathbb{S}},\dot{T})$ be the notion of forcing $\mathbb{A}:=(A,\unlhd)$, where:
\begin{enumerate}
\item $(p,\vec{S})\in A$ iff $p\in P$ and  either $\vec{S}=\emptyset$ or $\vec{S}$ is a $\langle p,\vec{\mathbb{S}}\rangle$-strategy;
\item $(p',{\vec{S}}')\unlhd (p,{\vec{S}})$ iff:
\begin{enumerate}[label=(\alph*)]
\item $p'\LE p$;\label{SharonNew1}
\item $\dom(\vec{S}')\geq \dom(\vec{S})$;\label{SharonNew2}
\item\label{SharonNew3} for each $i\in \dom(\vec{S})$ and $(q,t)\in \dom(S'_i)$,
 $$S'_i(q,t)=S_i(w(p,q), t_q),$$
 where $t_q:=\varpi_{\ell(q)}(q+t)$.\footnote{Note that $t_q\sle_n\varpi_n(w(p,q))$ for some $n\in[\ell(p'),\ell(q)]$. Thus, $(w(p,q),t_q)\in\dom(S_i)$. And if $t\sle_{\ell(q)}\varpi_{\ell(q)}(q)$, then $t=t_q$.}
\end{enumerate}
For all $p\in P$, denote $\myceil{p}{\mathbb{A}}:=(p,{\emptyset}).$
\end{enumerate}
\end{definition}

\begin{definition}[The maps]\label{d45}
\hfill
\begin{enumerate}
\item\label{newell} Let $\lh_{\mathbb A}:=\lh\circ\pi$ and $\vec{\varsigma}:=\vec{\varpi}\bullet\pi$,
where $\pi:\mathbb{A}\rightarrow \mathbb{P}$ is defined via $\pi(p, {\vec{S}}):=p;$
\item Define $c_{\mathbb A}:A\rightarrow H_\mu$, by letting, for all $(p,{\vec{S}})\in A$,
$$c_{\mathbb A}(p,{\vec S}):=(c(p),\{(i,c(q), S_i(q,\cdot))\mid i\in \dom({\vec{S}}), q\in W(p)\}),$$
where $S_i(q,\cdot)$ denotes the map $ t\mapsto S_i(q,t)$;

\item\label{newpitchfork} Let $a=(p,{\vec S})\in A$. The map $\fork{a}:\cone{p}\rightarrow A$ is defined
by letting $\fork{a}(p'):=(p',{\vec{S'}})$, where ${\vec{S'}}$ is a  sequence such that $\dom({\vec{S}'})=\dom(\vec{S})$, and for all $i\in\dom(\vec{S}')$ the following are true:
\begin{enumerate}[label=(\alph*)]
\item $\dom(S'_i)=\{(r,t)\mid r\in W(p')\;\&\; t\in\bigcup_{\ell(p')\leq n\leq\ell(r)}\,\mathbb{S}_n\downarrow \varpi_n(r)\}$,
\item  for all $(q.t)\in\dom(S'_i)$,
\begin{equation}\label{pitchfork}
\tag{*} S'_i(q,t)=S_i(w(p,q), t_q).
\end{equation}
\end{enumerate}
\end{enumerate}
\end{definition}

\begin{remark}
If $(\mathbb{P},\ell,c)$ is $\Sigma$-Prikry  then  a moment's reflection reveal that we arrive at the corresponding notions from \cite[\S4]{PartII}.
\end{remark}

We next show that $(\pitchfork,\pi)$ defines a super nice  forking projection from $(\mathbb{A},\ell_\mathbb{A},c_\mathbb{A},\vec{\varsigma})$ to $(\mathbb{P},\ell,c,\vec{\varpi})$. The next lemma takes care partially of this task by showing that $(\pitchfork,\pi)$ is a   forking projection from $(\mathbb{A},\ell_\mathbb{A},c_\mathbb{A})$ to $(\mathbb{P},\ell,c)$.

\begin{lemma}[Forking projection]The pair $(\pitchfork,\pi)$ is a forking projection between $(\mathbb{A},\ell_\mathbb{A},c_\mathbb{A})$ and $(\mathbb{P},\ell, c)$.
\label{forkingindeed}
\end{lemma}

\begin{proof}
We just give some details for the verification of Clause~\eqref{frk0}. 
The rest can be proved arguing similarly to \cite[Lemma~6.13]{PartI}. Here it goes:

\smallskip

(2) Let $a=(p,{\vec{S}})$ and $p'\LE \pi(a)$. We just proved that $\fork{a}$ is well-defined. The argument for $\fork{a}$ being order-preserving is very similar to the one from \cite[Lemma~6.13(2)]{PartI}.  If $\vec S=\emptyset$, then Definition \ref{d45}\eqref{pitchfork} yields $\fork{a}(p')=(p',\emptyset)\in A$. So, suppose that $\dom(\vec S)=\alpha+1$. Let $(p',\vec{S'}):=\fork{a}(p')$ and $i\leq\alpha$. We shall first verify that $S'_i$ is a $\langle p',\vec{\mathbb{S}}\rangle$-labeled tree. Let $(q,t)\in \dom(S'_i)$ and let us go over the clauses of Definition~\ref{labeled-p-tree}. Since the verification of Clauses~\eqref{C3ptree} and \eqref{C4ptree} are on the same lines of that of \eqref{C2ptree} we just give details for the latter.

\smallskip

\eqref{C2ptree}: Let $(q',t')\in \dom(S'_i)$ be such that $q'+t'\LE q+t$. By Clause \eqref{itsaprojection} of Definition~\ref{SigmaPrikry}, $w(p,q'+t')\LE w(p,q+t)$. Also, combining \cite[Lemma~2.9]{PartI} and Lemma~\ref{AdditionalAssumption3} we have the following chain of equalities: $$w(p,q'+t')=w(p,w(p', q'+t'))=w(p,w(p',q'))=w(p,q').$$ Similarly one shows $w(p,q+t)=w(p,q)$. Thus,   $w(p,q')\LE w(p,q)$.

By coherency of $\vec{\varpi}$ (Definition~\ref{CoherentSystem}\eqref{NiceCoherent})  we have 
\begin{eqnarray*}
	\varpi_{\ell(q)}(w(p,q')+t'_{q'})=\pi_{\ell(q'),\ell(q)}(\varpi_{\ell(q')}(w(p,q')+t'_{q'}))\\
=	\pi_{\ell(q'),\ell(q)}(t'_{q'})=\pi_{\ell(q'),\ell(q)}(\varpi_{\ell(q')}(q'+t'))=\varpi_{\ell(q)}(q'+t')
\end{eqnarray*}
Since $q'+t'\LE q+t$ we also have  $\varpi_{\ell(q)}(q'+t')\sle_{\ell(q)} \varpi_{\ell(q)}(q+t)=t_q.$\label{projectingtow(p,q)preserves}

Thereby, combining both things we arrive at  $$w(p,q')+t'_{q'}\LE w(p,q)+t_q.$$ Finally,  use Clause~\eqref{C2ptree} for the labeled $\langle p,\vec{\mathbb{S}}\rangle$-tree $S_i$ to get that
$$S_i'(q',t')=S_i(w(p,q'),t'_{q'})\supseteq S_i(w(p,q),t_q)=S_i'(q,t).$$

\smallskip

To prove that $(p',\vec{S}')\in A$ it is left to argue that $\vec{S}'$ fulfils the requirements described in Clauses~\eqref{C3pstrategy}, \eqref{C4pstrategy} and \eqref{C5pstrategy} of Definition~\ref{strategy}. Indeed, each of these clauses follow  from the corresponding ones for $\vec{S}$. There is just one delicate point in Clause~\eqref{C4pstrategy}, where one needs to argue that $w(p,q')+t'_{q'}\LE w(p,q)+t_{q}$. This is again done as in the verification of Clause~\eqref{C2ptree} above.

Finally, it is clear that $\fork{a}(p')=(p',\vec{S}')\unlhd (p,\vec{S})$ (see~Definition~\ref{SharonNew}). This concludes the verification of Clause~\eqref{frk0}.
\end{proof}

\begin{lemma}\label{niceforkingindeed3}
For each $n<\omega$,  $\varsigma_n$ is  a nice projection from $\mathbb{A}_{\geq n}$ to $\mathbb{S}_n$, and for each $k\geq n$, $\varsigma_n\restriction \mathbb{A}_k$ is again a nice projection.
\end{lemma} 
\begin{proof}
We go over the clauses of Definition \ref{niceprojection}. Clauses~\eqref{niceprojection1} and \eqref{niceprojection2} follow from the fact that $\varsigma_n$ is the composition of the projections $\varpi_n$ and $\pi$ and Clause~\eqref{niceprojection3} follows from Lemma~\ref{pitchforkexact} and  Lemma~\ref{forkingindeed}. To complete the argument we prove the following, which, in particular, yields Clause \eqref{theprojection}.

\begin{claim}\label{varsigmaindeednice}
Let $a, a'\in A_{\geq n}$ and $s\sle_n\varsigma_n(a)$ with  $a'\unlhd a+s$.  Then, for each $p^*\in P_{\geq n}$ such that $p^*\LE^{\varpi_n} \pi(a)$ and $\pi(a')=p^*+\varsigma_n(a')$ there is   $a^*\in A_{\geq n}$ such that $a^* \unlhd^{\varsigma_n} a$ with $\pi(a^*)=p^*$ and $a'=a^*+\varsigma_n(a')$.

 In particular, $\varsigma_n$ satisfies Clause \eqref{theprojection}.
\end{claim}
\begin{proof}
Let $a=(p,\vec{S})$, $a'=(p',\vec{S}')$ and $s\sle_n\varsigma_n(a)$ be as above.

By Lemma~\ref{pitchforkexact}, $a'\unlhd a+s=\fork{a}(p+s)$, hence $p'\LE p+s$. Since $\varpi_n$ is a nice projection from $\mathbb{P}_{\geq n}$ to $\mathbb{S}_n$, Definition~\ref{niceprojection}\eqref{theprojection} yields the existence of a condition $p^*\in P_{\geq n}$ such that $p^*\LE^{\varpi_n} p$ and $p'=p^*+\varsigma_n(a')$.   So, let $p^*$ be some such condition and set $t:=\varpi_n(p')$. We have that  $\varsigma_n(a')=t$. Our aim is to find a sequence $\vec{S}^*$ such that $a^*:=(p^*,\vec{S}^*)$ is a condition in $\mathbb{A}_{\geq n}$ with the property that $a^*\unlhd^{\varsigma_n}a$ and $a^*+t=a'$.

As $n\leq \ell(p^*)$, coherency of $\vec\varpi$   yields $\varpi_n``W(p^*)=\{\varpi_n(p)\}$ (see Definition~\ref{CoherentSystem}\eqref{AdditionalAssumption1}), hence $q+t$ is well-defined for all $q\in W(p^*)$. Moreover,  by  virtue of Lemma~\ref{LemmaAdditionalAssumptions}\eqref{AdditionalAssumption5}, $q+t\in W(p^*+t)=W(p')$ for every $q\in W(p^*)$.

Put $\vec{S}:=\langle S_i\mid i\leq \alpha\rangle$ and $\vec{S}':=\langle S'_i\mid i\leq \beta\rangle$. Let $\vec{S}^*:=\langle S_i\mid i\leq\beta\rangle$ be the sequence where for each $i\leq \beta$,  $S^*_i$ is the function with domain $\{(q,u)\mid q\in W(p^*)\,\&\, u\in \bigcup_{\ell(p^*)\leq m\leq\ell(q)}\,\mathbb{S}_m\downarrow \varpi_m(q)\}$ defined  according to the following casuistic:
\begin{enumerate}[label=(\alph*)]
\item\label{case1Niceproj} If $\vec{S}$ is the empty sequence, then  $$S^*_i(q,u):=\begin{cases}
S'_i(q+t, u_q), &\text{if $q+u\LE q+t$};\\
\emptyset, & \text{otherwise.}
\end{cases}
$$
\item\label{case2Niceproj} If $\vec{S}$ is non-empty then there are two more cases to consider:

  \begin{enumerate}[label=(\arabic*)]
  \item\label{case2Niceproj1}  If $\alpha<i\leq \beta$, then $$S^*_i(q,u):=\begin{cases}
S'_i(q+t, u_q), &\text{if $q+u\LE q+t$};\\
S_\alpha(w(p,q),u_q), & \text{otherwise.}
\end{cases}
$$

\item\label{case2Niceproj2}  Otherwise,  $S^*_i(q,u):=S_i(w(p,q),u_q)$.
\end{enumerate}
\end{enumerate}

By Lemma~\ref{LemmaAdditionalAssumptions}\eqref{SumInvariant}, $q+u=q+u_q$ for all $(q,u)\in \dom(S^*_i)$ and $i\leq \beta$.

We next show that $\vec{S}^*$ is a $\langle p^*,\vec{\mathbb{S}}\rangle$-strategy. To this aim we go over the clauses of Definition~\ref{strategy}. Clause \eqref{C1pstrategy} is indeed obvious. As for the rest we just provide details for \eqref{C2pstrategy}. The reason for this choice is the following: first, the verification of \eqref{C2pstrategy} makes a crucial use of the notion of \emph{coherent system of projections} (Definition~\ref{CoherentSystem}) and of the $t_q$ component in the labeled trees (see Definition~\ref{SharonNew}\ref{SharonNew3});\footnote{This is yet another wrinkle compared with the non-collapses scenario.} second, this verification contains all the key ingredients to assist the reader in the verification of the rest of clauses.

\smallskip

So, let us verify that ${S}^*_i$ is a labeled $\langle p^*,\vec{S}\rangle$-tree
 for all $i\leq \beta$.
Fix  some $i\leq \beta$ and let us go over the clauses of Definition \ref{labeled-p-tree}.

\smallskip

\eqref{C1ptree}: This is obvious.

\smallskip

\eqref{C2ptree}:   Let $(q',u'), (q,u)\in \dom(S^*_i)$ such that $q'+u'\LE q+u$.

\underline{Case \ref{case1Niceproj}:} We need to distinguish two subcases:

$\br$ If  $q+u \nleq q+t$, then $S^*_i(q,u)=\emptyset$ and so $S^*_i(q,u)\s S^*_i(q',u')$.

$\br$ Otherwise, $q+u\LE q+t$ and so $S_i^*(q,u)=S'_i(q+t,u_q)$. On the other hand, since $q'+u'\LE q+u\LE q+t$, we have that $\varpi_n(q'+u')\sle_n \varpi_n(q+t)=t$. In particular, $q'+u'\LE q'+t$ and so $S^*_i(q',u')=S'_i(q'+t,u'_{q'})$.
Now, it is routine to check that $(q+t)+u_q=q+u_q=q+u$. Similarly, the same applies to  $q'$ and $u'_{q'}$.
Appealing to Clause~\eqref{C2ptree} for $S'_i$ we get $S^*_i(q,u)\s S^*_i(q',u')$.

\smallskip

\underline{Case \ref{case2Niceproj}\ref{case2Niceproj1}:} There are several cases to consider:

$\br$ Assume $q+u\nleq q+t$. Then $S^*_i(q,u)=S_\alpha(w(p,q),u_q)$.

$\br\br$ Suppose $q'+u'\nleq q+t$. Then $S^*_i(q',u')=S_\alpha(w(p,q'),u'_{q'})$. On one hand,  $w(p,q'+u')\LE w(p,q+u)$ (see Definition~\ref{SigmaPrikry}\eqref{itsaprojection}). Combining \cite[Lemma~2.9]{PartI} with Lemma~\ref{AdditionalAssumption3}, we have  $w(p,q'+u')=w(p,w(p^*,q'+u'))=w(p,w(p^*,q'))=w(p,q')$. Similarly,  one shows that $w(p,q+u)=w(p,q)$.  Thus, $w(p,q')\LE w(p,q)$. Also, arguing  as in page~\pageref{projectingtow(p,q)preserves}, one can prove that $w(p,q')+u'_{q'}\LE w(p,q)+u_q$.  This finally yields $$S^*_i(q,u)=S_\alpha(w(p,q),u_q)\s S_\alpha(w(p,q'),u'_{q'})=S^*_i(q',u').$$

$\br\br$ Otherwise,  $q'+u'\LE q+t$, and so $S^*_i(q',u')=S'_i(q'+t,u'_{q'})$.

Since $\alpha<i$ and $b\unlhd a$, Clauses~\eqref{C3pstrategy} and \eqref{C5pstrategy} of Definition~\ref{strategy} for $\vec{S}'$ yield
$$S_\alpha(w(p,q'+t),u^*)= S'_\alpha(q'+t,u'_{q'})\s S'_i(q'+t,u'_{q'}),$$
where $u^*:=\varpi_{\ell(q')}((q'+t)+u'_{q'})$. One can check that $(q'+t)+u'_{q'}=q'+u'_{q'}$, hence $u^*=u'_{q'}$, and thus $S_\alpha(w(p,q'+t),u'_{q'})\s S'_i(q'+t,u'_{q'}).$

Arguing as in the previous case,  $w(p, q'+t)=w(p,q')$.
Therefore,
$$S_\alpha(w(p,q'),u'_{q'})\s S'_i(q'+t,u'_{q'})=S^*_i(q',u').$$
Once again, $w(p,q')+u'_{q'}\LE w(p,q)+u_q$. Hence,
Clause~\eqref{C2ptree} for $S_\alpha$ yields
$$S^*_i(q,u)=S_\alpha(w(p,q),u_q)\s S_\alpha(w(p,q'),u'_{q'})\s S^*_i(q',u').$$

$\br$ Assume $q+u\LE q+t$.  Then $q'+u'\LE q+t$, as well. In particular, $$S^*_i(q,u)=S'_i(q+t,u_q)\s S'_i(q'+t,u'_{q'})=S^*_i(q',u'),$$ where the above follows from Clause~\eqref{C2ptree} for $S'_i$.\footnote{See also the argument of Case~\ref{case1Niceproj} above.}

\underline{Case~\ref{case2Niceproj}\ref{case2Niceproj2}:} This is clear using Clause~\eqref{C2ptree} for $S_i$.

\medskip

\eqref{C3ptree}:  Let $(q,u)\in \dom(S^*_i)$. There are two cases to discuss:

\underline{Case~\ref{case1Niceproj}:} As before there are two cases depending on whether $q+u\nleq q+t$ or not. The first case is obvious, as $S^*_i(q,u)=\emptyset$.  
Otherwise, $S^*_i(q,u)=S'_i(q+t,u_q)$, and so Clause~\eqref{C3ptree} for $S'_i$ yields $q+u\forces_\mathbb{P} S^*_i(q,u)\cap \dot{T}^+=\emptyset$.\footnote{Once again, we have used that  $(q+t)+u_q=q+u_q=q+u$.}

\underline{Case~\ref{case2Niceproj}\ref{case2Niceproj1}:} There are two cases to consider:

$\br$ Assume $q+u\nleq q+t$. Then, $S^*_i(q,u)=S_\alpha(w(p,q),u_q)$. Combining $q+u_q\LE w(p,q)+u_q$ with Clause~\eqref{C3ptree} for $S_\alpha$, $q+u\forces_\mathbb{P} S^*_i(q,u)\cap \dot{T}^+=\emptyset$.

$\br$ Otherwise $q+u\LE q+t$, and so $S^*_i(q,u)=S'_i(q+t, u_q)$.
As in previous cases we have $q+u\forces_\mathbb{P} S^*_i(q,u)\cap \dot{T}^+=\emptyset$.

 \underline{Case~\ref{case2Niceproj}\ref{case2Niceproj2}:}  This follows using Clause~\eqref{C3ptree} for $S_i$.

\medskip

\eqref{C4ptree}: Let $q\in W(p^*)$ and a pair $(q',u')\in \dom(S^*_i)$ with $q'\LE q$ and $\ell(q)\geq \ell(p^*)+m_i$, where
$m_i= \max\{m(S_\alpha), m(S_i),m(S'_i)\}+1.\footnote{If $\vec{S}$ is empty then we convey that $m(S_\alpha):=0$.}$

To avoid trivialities, suppose $S^*_i(q',u')\neq \emptyset$ and put $\delta:=\max(S^*_i(q',u'))$.

\underline{Case~\ref{case1Niceproj}:} Since $S^*_i(q',u')\neq\emptyset$ we have $q'+u'\LE q+t$. In this case $S^*_i(q',u')=S'_i(q'+t,u'_{q'})$.  Since $q'+t\LE q$, Clause~\eqref{C4ptree} for $S'_i$ yields $(\delta, q)\in R$.

\underline{Case~\ref{case2Niceproj}:} The verification of the clause in Case~\ref{case2Niceproj}\ref{case2Niceproj1} and Case~\ref{case2Niceproj}\ref{case2Niceproj2} is identical to the previous one. Simply note that we can still invoke Clause~\eqref{C4ptree} for  $S_i$, $S_\alpha$ or $S'_i$, as $m_i$ is sufficiently large.

\medskip

After the above verification  we infer that $a^*=(p^*,\vec{S}^*)\in \mathbb{A}_{\geq n}$. One can easily check that $a^*\unlhd^{\varsigma_n} a$, so we are just left with checking that $a^*+t=a'$.\footnote{Recall that $t=\varpi_n(p')=\varsigma_n(a')$.} 

\smallskip

By Lemma~\ref{pitchforkexact}, $a^*+t=\fork{a^*}(p^*+t)$. Thus, since $p^*+t=p'$, $a^*+t=\fork{a^*}(p')$. We will be done by  showing that $\fork{a^*}(p')=a'$.

Put $\fork{a^*}(p')=(p',\vec{Q})$.
Let $i\leq\beta$ and $(q,u)\in \dom(Q_i)$. By virtue of Definition \ref{d45}\eqref{newpitchfork} we have that $q\LE p'$ and $u\sle_m\varpi_m(q)$. Using coherency of $\vec{\varpi}$ (particularly, Definition~\ref{CoherentSystem}\eqref{NiceCoherent}) one can check that $q+u\LE q+t$. 

\underline{Case~\ref{case1Niceproj}:} In this case we have the following chain of equalities: $$Q_i(q,u)=S^*_i(w(p^*,q),u_q)=S'_i(w(p^*,q)+t,u_q)=S'_i(q,u_q)=S'_i(q,u).$$
 The first equality follows from  Definition \ref{d45}\eqref{newpitchfork}\eqref{pitchfork}; the  third from Lemma \ref{LemmaAdditionalAssumptions}\eqref{AdditionalAssumption2}; and the right-most combining Definition~\ref{labeled-p-tree}\eqref{C2ptree} with $q+u_{q}=q+u$.

\underline{Case~\ref{case2Niceproj}:}
If  $\alpha<i\leq\beta$ then arguing as before $Q_i(q,u)=S'_i(q,u)$.

 Otherwise, $i\leq \alpha$ and we have the following chain of equalities
$$Q_i(q,u)=S^*_i(w(p^*,q),u_q)=S_i(w(p,q),u_q)=S'_i(q,u_q)=S'_i(q,u),$$
Note that for the third equality we used  $a'\unlhd a$ and $u_q=\varpi_{\ell(q)}(q+u_q)$.
\end{proof}
The above claim finishes the proof of the lemma.
\end{proof}

\begin{cor}\label{PitchforkSuperNice}
The pair $(\pitchfork,\pi)$ is a super nice forking projection from $(\mathbb{A},\ell_\mathbb{A},c_\mathbb{A},\vec{\varsigma})$ to $(\mathbb{P},\ell,c,\vec{\varpi})$.
\end{cor}
\begin{proof}
First, $(\pitchfork,\pi)$ is a forking projection from $(\mathbb{A},\ell_\mathbb{A},c_\mathbb{A},\vec{\varsigma})$ to $(\mathbb{P},\ell,c,\vec{\varpi})$ by virtue of Lemma~\ref{forkingindeed}. Second, $\vec{\varsigma}=\vec{\varpi}\bullet \pi$ by our choice in Definition~\ref{d45}. Besides, Lemma~\ref{niceforkingindeed3} shows that for each $n<\omega$,  $\varsigma_n$ is a nice projection from $\mathbb{A}_{\geq n}$ and $\mathbb{S}_n$. Moreover, Claim~\ref{varsigmaindeednice} actually shows that $(\pitchfork,\pi)$ is super nice (see Definition~\ref{SuperNiceForking}). This completes the proof.
\end{proof}

Next, we introduce a map $\tp$ and we will later prove that it indeed defines a nice type over $(\pitchfork,\pi)$. Afterwards, we will also show that $\tp$ witness  the pair $(\pitchfork,\pi)$ to satisfy the weak mixing property.

	\begin{definition}\label{defntypes} Define a map $\tp:A\rightarrow{}^{<\mu}\omega$, as follows.

	Given $a=(p,\vec S)$ in  $A$, write $\vec S$ as $\langle S_i\mid i<\beta\rangle$, and then let
	$$\tp(a):=\langle m(S_i)\mid i<\beta\rangle.$$
		\end{definition}
		
		We shall soon verify that $\tp$ is a nice type, but will use the $\mtp$ notation of Definition~\ref{type} from the outset. In particular, for each $n<\omega$,  we will have $\z{\mathbb{A}}_n:=(\z{A}_n,\unlhd)$, with $\z{A}_n:=\{a\in A\mid \pi(a)\in \z{P}_n \,\&\,\mtp(a)=0\}$. We will often refer to the $m(S_i)$'s 	as the {\bf delays} of the strategy $\vec{S}$.
		
		\smallskip
		
		Arguing along the lines of \cite[Lemma~6.15]{PartI} one can prove the following:

		\begin{fact}\label{directedclosedring}
For each $n<\omega$, $\z{\mathbb{A}}^\pi_n$ is a $\mu$-directed closed forcing.
\end{fact}

		\begin{lemma}\label{typeindeed}
	The map $\tp$ is a nice type over $(\pitchfork,\pi)$.
	\end{lemma}
	\begin{proof}
The verification of Clauses~\eqref{type1}--\eqref{newstretch} of Definition~\ref{typeexact} is essentially the same as in \cite[Lemma~4.19]{PartII}. A moment's reflection makes it clear that it suffices to prove Clause~\eqref{ringismoredense} to complete the lemma.

Let $a=(p, \vec{S})\in A$ and to avoid trivialities, let us assume that $\vec{S}\neq \emptyset$.

$\br$ Suppose that $p$ is incompatible with $r^\star$. Then, by Remark~\ref{incompatiblewithrstar}, for all $i<\dom(\tp(a))$ and all $(q,t)\in \dom(S_i)$, $S_i(q,t)=\emptyset$.  Therefore, $\mtp(a)=0$.  Using Definition~\ref{SigmaPrikry}\eqref{moreclosedness}, let $p'\LE^{\vec{\varpi}} p$ be in $\z{P}_{\ell(p)}$ and set $b:=\fork{a}(p')$. Combining Clauses~\eqref{type2}
and \eqref{type3} of Definition~\ref{type} with  $\mtp(a)=0$ it is immediate that $\mtp(b)=0$. Also, $\pi(b)=p'\in \z{P}_{\ell(p)}$. Thus, $b\in \z{A}_{\ell_\mathbb{A}(a)}$ and $b\unlhd^{\vec{\varsigma}} a$.

$\br$ Suppose $p\LE r^\star$. Appealing to Clause~\eqref{csize} of Definition~\ref{SigmaPrikry} let  $\gamma<\mu$ be above $\sup_{i<\dom(\vec{S})}\{S_i(q,s)\mid (q,s)\in \dom(S_i)\}$ and $\dom(\vec{S})$. By Lemma~\ref{Runbdeddirectmore}, let $\bar{\gamma}\in (\gamma,\mu)$ and $\bar{p}\LE^{\vec{\varpi}} p$ such that $(\bar{\gamma},\bar{p})\in R$. Using Definition~\ref{SigmaPrikry}\eqref{moreclosedness}  we may further assume that $\bar{p}$ belongs to  $\z{P}_{\ell(p)}$.

 Next, define a sequence $\vec{T}=\langle T_i\mid i\leq \bar{\gamma}\rangle$ with $$\dom(T_i):=\{(q,u)\mid q\in W(\bar{p})\,\&\, u\in \bigcup_{\ell(p^*)\leq m\leq\ell(q)}\,\mathbb{S}_m\downarrow \varpi_m(q)\},$$
 as follows
 $$T_i(q,u):=\begin{cases}
 S_i(w(p,q), u_q), & \text{if $i<\dom(\vec{S})$,}\\
 S_{\max(\dom(\vec{S}))}(w(p,q),u_q)\cup\{\bar{\gamma}\}, & \text{otherwise.}
 \end{cases}
 $$

Arguing as in Claim~\ref{varsigmaindeednice} one shows that $(\bar{p}, \vec{T})$ is a condition in $\z{\mathbb{A}}_{\ell(p)}$. Clearly $b\unlhd^{\vec{\varsigma}}a$, so $\z{\mathbb{A}}^{\varsigma_n}_n$ is dense in $\mathbb{A}^{\varsigma_n}_n$.
	\end{proof}

We now check that the pair $(\pitchfork,\pi)$ has the weak mixing property,  as witnessed by the type $\tp$ given in Definition~\ref{defntypes} (see Definition~\ref{mixingproperty}).

\begin{lemma}\label{mixing}
The pair $(\pitchfork,\pi)$ has the weak mixing property as witnessed by the type $\tp$ from Definition~\ref{defntypes}.
\end{lemma}
\begin{proof} Let $a$, $\vec{r}$, $p'\LE^0 \pi(a)$, $g\colon W_n(\pi(a))\rightarrow \conea{a}$ and $\iota$ be as in the statement of the Weak Mixing Property (see Definition~\ref{mixingproperty}). More precisely, $\vec{r}=\langle r_\xi\mid \xi<\chi\rangle$ is a good enumeration of $W_n(\pi(a))$, $\langle \pi(g(r_\xi))\mid \xi<\chi\rangle$ is diagonalizable with respect to $\vec{r}$ (as witnessed by $p'$) and $g$ is a function witnessing Clauses~\eqref{Mixing3}--\eqref{Mixing5} of Definition~\ref{mixingproperty} with respect to the type $\tp$.

Put $a:=(p,\vec{S})$ and for each $\xi<\chi$, set $(p_{\xi}, \vec{S}^{\xi}):=g(r_\xi)$.

\begin{claim}\label{iotachiOK}
If $\iota\ge\chi$ then there is a condition $b$ in $\mathbb A$ as in the conclusion Definition~\ref{mixingproperty}.
\end{claim}
\begin{proof}
If $\iota\ge\chi$ then Clause~\eqref{Mixing3} yields $\dom(\tp(g(r_\xi))=0$ for all $\xi<\chi$. Hence, Clause~\eqref{type6} of Definition~\ref{type} yields $g(r_\xi)=\myceil{p_\xi}{\mathbb{A}}$ for all $\xi<\chi$. Since $g(r_\xi)\unlhd a$ this in particular implies that $a=\myceil{p}{\mathbb{A}}$. 

Set $b:=\myceil{p'}{\mathbb{A}}$. Clearly, $\pi(b)=p'$ and $b\unlhd^0 a$. 
Let $q'\in W_n(p')$. By Clause~\eqref{Mixing2} of Definition~\ref{mixingproperty}, $q'\LE^0 p_\xi$, where $\xi$ is such that $r_\xi=w(p, q')$.  Finally, Definition~\ref{forking}\eqref{frk6} yields
$\fork{b}(q')=\myceil{q'}{\mathbb{A}}\unlhd^0\myceil{p_\xi}{\mathbb{A}}=g(r_\xi)$.
\end{proof}
So, hereafter let us assume that $\iota<\chi$. For each $\xi\in[\iota,\chi)$,  Clause~\eqref{Mixing3} of Definition~\ref{mixingproperty} and Definition~\ref{defntypes} together imply that $\dom(\vec{S}^\xi)=\alpha_\xi+1$ for some $\alpha_\xi<\mu$.
Moreover, Clause~\eqref{Mixing3} yields $\sup_{\iota\leq \eta<\xi} \alpha_\eta<\alpha_\xi$ for all $\xi\in (\iota,\chi)$. Also,  the same clause  implies that $g(r_\xi)=\myceil{p_\xi}{\mathbb{A}}$, hence $\vec{S}^\xi=\emptyset$, for all $\xi<\iota$.

\smallskip

Let $\langle s_\tau\mid \tau<\theta\rangle$ be the  good enumeration  of $W_n(p')$.
By Definition~\ref{SigmaPrikry}\eqref{csize}, $\theta<\mu$.
 For each $\tau<\theta$, set $r_{\xi_\tau}:=w(p,s_\tau)$.

By Definition~\ref{mixingproperty}\eqref{Mixing1},  $$s_\tau\LE^0 \pi(g({w(p,s_\tau)}))=\pi(g(r_{\xi_\tau}))=p_{\xi_\tau},$$
for each $\tau<\theta$. Set $\alpha':=\sup_{\iota\leq \xi<\chi}\alpha_{\xi}$ and $\alpha:=\sup(\dom(\vec{S}))$.\footnote{Note that $a$ might be $\myceil{p}{\mathbb{A}}$, so we are allowing $\alpha=0$.} 
By regularity of $\mu$ and Definition~\ref{mixingproperty}\eqref{Mixing3} it follows that $\alpha< \alpha'<\mu$. 
Our goal is to define a sequence $\vec{T}=\langle T_i\mid i\leq\alpha'\rangle$, with $\dom(T_i):=\{(q,u)\mid q\in W(p')\,\&\, u\in \bigcup_{\lh(p')\leq m\leq \lh(q)}\,\mathbb{S}_m\downarrow\varpi_m(q)\}$ for $i\leq \alpha'$, such that $b:=(p',\vec{T})$ is a condition in $\mathbb{A}$ satisfying the conclusion of the weak mixing property.

As  $\langle s_\tau\mid \tau<\theta\rangle$ is a good enumeration of the $n^{th}$-level of the $p'$-tree $W(p')$,
Lemma~\ref{W(p)maxAntichain}\eqref{W(p)maxAntichain2} entails that, for each $q\in W({p'})$, there is a unique ordinal $\tau_q<\theta$, such that $q$ is comparable with $s_{\tau_q}$.
It thus follows from Lemma~\ref{W(p)maxAntichain}\eqref{W(p)maxAntichain4} that, for all $q\in W(p')$, $\ell(q)-\ell(p')\geq n$ iff $q\in W(s_{\tau_q})$.
Moreover, for each $q\in W_{\geq n}(p')$, $q\LE s_{\tau_q}\LE^0 p_{{\xi_{\tau_q}}}$, hence $w(p_{_{\xi_{\tau_q}}},q)$ is well-defined. Now, for all $i\leq\alpha'$ and $q\in W({p'})$, let:
$$T_i(q,u):=
\begin{cases}
S^{\xi_{\tau_q}}_{\min\{i,\alpha_{\xi_{\tau_q}}\}}(w(p_{\xi_{\tau_q}},q),u_q),& \text{if }q\in {W(s_{\tau_q})}\;\& \;\iota\leq \xi_{\tau_q};\\
S_{\min\{i,\alpha\}}(w(p,q),u_q),& \text{if }q\notin {W(s_{\tau_q})}\;\&\; \alpha>0;\\
\emptyset, & \text{otherwise.}
\end{cases}$$
Next, we show that $\vec{T}$ is a $\langle p',\vec{\mathbb{S}}\rangle$-strategy. The only non-routine part is to show that  $T_i$ is a labeled $\langle p',\vec{\mathbb{S}}\rangle$-tree for all $i\leq \alpha'$ (i.e., Clause~\eqref{C2pstrategy} in  Definition~\ref{strategy}).

\begin{claim}\label{claim5211}
Let $i\leq\alpha'$. Then $T_i$ is a labeled $\langle p',\vec{\mathbb{S}}\rangle$-tree.
\end{claim}
\begin{proof}
Fix $(q,u)\in \dom(T_i)$ and let us go over the Clauses of Definition~\ref{labeled-p-tree}. The verification of \eqref{C1ptree}, \eqref{C2ptree} and \eqref{C3ptree} are similar to that of \cite[Claim~6.16.1]{PartI} and, actually, also  to that of Claim~\ref{varsigmaindeednice} above. The reader is thus referred there for more details. We just elaborate on  Clause~\eqref{C4ptree}.

\smallskip

For each $i< \alpha'$, set
$ \xi(i):=\min\{\xi\in[\iota,\chi)\mid i\leq \alpha_\xi\}.$

\begin{subclaim}
If $i< \alpha'$, then
$$m(T_i)\leq n+\max\{\mtp(a),\sup\nolimits_{\iota\leq \eta<\xi(i)} \mtp(g(r_\eta)),\tp(g(r_{\xi(i)})(i)\}.$$
\end{subclaim}

\begin{proof}
Let $q\in W_k(p')$ and $(q',u')$ be a pair in $\dom(T_i)$ with $q'\LE q$, where
$$k\geq n+\max\{\mtp(a),\sup\nolimits_{\iota\leq \eta<\xi(i)} \mtp(g(r_\eta)),\tp(g(r_{\xi(i)})(i)\}.$$
Suppose that  $T_i(q',u')\neq \emptyset$. Denote $\tau:=\tau_{q'}$ and $\delta:=\max(T_i(q',u'))$.
Since $\ell(q)\geq \ell(p')+n$, note that  $q,q'\in W(s_\tau)$. Also, $\iota\leq \xi_\tau$, as otherwise $T_i(q',u')=\emptyset$. Thus, we fall into the first option of the casuistic getting
 $$T_i(q',u')=S^{{\xi_\tau}}_{\min\{i,\alpha_{\xi_\tau}\}}(w(p_{\xi_\tau},q'),u'_{q'}).$$

$\br$  Assume that $\xi_\tau<\xi(i)$. Then,  $\alpha_{\xi_\tau}<i$ and so  $$T_i(q',u')=S^{\xi_\tau}_{\alpha_{\xi_\tau}}(w(p_{\xi_\tau},q'),u'_{q'}).$$
We have that $w(p_{\xi_\tau},q')\LE w(p_{\xi_\tau},q)$ is a pair in $W_{k-n}(p_{\xi_\tau})$ and that the set $S^{\xi_\tau}_{\alpha_{\xi_\tau}}(w(p_{\xi_\tau},q'),u'_{q'})$ is non-empty. Also, $k-n\geq \mtp(g(r_{\xi_\tau}))=m(S^{\xi_\tau}_{\alpha_{\xi_\tau}})$. 
So, by Clause~\eqref{C4ptree} for $S^{\xi_\tau}_{\alpha_{\xi_\tau}}$, we have that $(\delta, w(p_{\xi_\tau},q))\in R$. Finally, since $q\LE w(p_{\xi_\tau},q)$, we have  $(\delta,q)\in R$, as desired. 

\smallskip

$\br$  Assume that $\xi(i)\leq  \xi_\tau$. Then
 $i\leq \alpha_{\xi(i)}\leq \alpha_{\xi_\tau}$, and thus $$T_i(q',u')=S^{\xi_\tau}_i(w(p_{\xi_\tau},q'),u'_{q'}).$$
 If $\dom(\tp(a))\leq i\leq \sup_{\iota\leq \eta<\xi(i)}\alpha_\eta$,  by Clause~\eqref{Mixing4} of Definition~\ref{mixingproperty},  $$\tp(g(r_{\xi_\tau}))(i)\leq \mtp(a).$$
 Otherwise, if $\sup_{\iota\leq \eta<\xi(i)}\alpha_\eta<i\leq \alpha_{\xi(i)}$, again by Definition~\ref{mixingproperty}\eqref{Mixing4}
 $$\tp(g(r_{\xi_\tau}))(i)\leq \max\{\mtp(a),\tp(g(r_{\xi(i)})(i)\}.$$
 In either case,  $w(p_{\xi_\tau},q)\in W_{k-n}(p_{\xi_\tau})$ and  $k-n\geq \tp(g(r_{\xi_\tau}))(i)= m(S^{\xi_\tau}_i)$. By Clause~\eqref{C4ptree} of $S^{\xi_\tau}_i$  we get that $(\delta, w(p_{\xi_\tau},q))\in R$, hence $(\delta,q)\in R$.
\end{proof}

\begin{subclaim}
$m(T_{\alpha'})\leq n+\sup_{\iota\leq \xi<\chi} \mtp(g(r_\xi))$.
\end{subclaim}
\begin{proof}
Let $q\in W_k(p')$ and $(q',u')\in \dom(T_i)$ with $q'\LE q$, where
$k\geq  n+\sup_{\iota\leq \xi<\chi} \mtp(g(r_\xi))$. Suppose that  $T_{\alpha'}(q',u')\neq \emptyset$ and
 denote $\tau:=\tau_{q'}$ and $\delta:=\max(T_{\alpha'}(q',u'))$. Since $k\geq n$, $q,q'\in W(s_{\tau})$. Also, $\iota\leq \xi_\tau$, as otherwise $T_{\alpha'}(q',u')=\emptyset$.   So, $T_{\alpha'}(q',u')=S^{\xi_{\tau}}_{\alpha_{\xi_\tau}}(w(p_{\xi_{\tau}},q'),u'_{q'}).$
Then $w(p_{\xi_{\tau}},q')\leq w(p_{\xi_{\tau}},q)$ is a pair in $W_{k-n}(p_{\xi_\tau})$ with $k-n\geq \mtp(g(r_{\xi_\tau}))= m(S^{\xi_\tau}_{\alpha_{\xi_\tau}})$. So,  Definition~\ref{labeled-p-tree}\eqref{C4ptree}  regarded with respect to $S^{\xi_\tau}_{\alpha_{\xi_\tau}}$ yields  $(\delta,w(p_{\xi_{\tau}},q))\in R$.  Once again, it follows that $(\delta,q)\in R$, as wanted.
\end{proof}

The combination of the above subclaims yield Clause~\eqref{C4ptree} for $T_i$. \qedhere
\end{proof}

Thereby we establish that $b:=(p',\vec{T})$ is a legitimate condition in $\mathbb{A}$. Next, we show that $b$ satisfies the requirements for $(\pitchfork,\pi)$ to have the weak mixing property. By definition, $\pi(b)=p'$, and it is easy to show that $b\unlhd^0 a$.

\begin{claim}\label{Claimw} Let $\tau<\theta$. For each $q\in W_n(s_\tau)$, $w(p',q)=w(s_\tau,q)=q$.\qedhere
\end{claim}

\begin{claim}
For each $\tau<\theta$, $\fork{b}(s_\tau)\unlhd^0 g(r_{\xi_\tau})$.\footnote{Recall that $\langle s_\tau\mid \tau<\theta\rangle$ was a good enumeration of $W_n(p')$. }
\end{claim}
\begin{proof} Let $\tau<\theta$ and denote $\fork{b}(s_\tau)=(s_\tau,\vec{T}_\tau)$. By Lemma~\ref{forkingindeed}\eqref{frk5},
 we have that  $\pi(\fork{b}(s_\tau))=s_\tau\LE^0 p_{\xi_\tau}$, so Clause~\ref{SharonNew1} of Definition~\ref{SharonNew} holds.

If $\xi_\tau<\iota$, then
 $\fork{b}(s_\tau)\unlhd^0 \myceil{p_{\xi_\tau}}{\mathbb{A}}=g(r_{\xi_\tau})$, and we are done. So, let us assume that $\iota\leq \xi_\tau$. Let $i\leq \alpha_{\xi_\tau}$ and $q\in W(s_\tau)$. By Definition~\ref{d45}\eqref{pitchfork}, $T^\tau_i(q,u)=T_i(w(p',q),u_q)$ and by Claim~\ref{Claimw}, $w(p',q)=w(s_\tau,q)=q$, hence $T^\tau_i(q,u)=T_i(q,u_q)=T_i(q,u)$.\footnote{For the second equality we use Definition~\ref{labeled-p-tree}\eqref{C2ptree} for $T_i$ and, again, that  $q+u=q+u_q$.} Also $r_{\xi_{\tau_q}}=w(p,s_{\tau_q})=w(p,s_\tau)=r_{\xi_\tau}$, where the second equality follows from  $q\in W(s_\tau)$. Therefore, $$T^\tau_i(q,u)=S^{\xi_\tau}_{\min\{i,\alpha_{\xi_\tau}\}}(w(p_{\xi_\tau},q),u_q)=S^{\xi_\tau}_{i}(w(p_{\xi_\tau},q),u_q).$$ Altogether, $\fork{b}(s_\tau)\unlhd^0 g(r_{\xi_\tau})$, as wanted.
\end{proof}
The combination of the above claims yield the proof of the lemma.
\end{proof}

Let us sum up what we have shown so far:
\begin{cor}\label{exactforkingAtoP}\label{corA} \label{AisweaklySigmaPrikry}
$(\pitchfork, \pi)$ is a super nice forking projection from $(\mathbb{A},\ell_\mathbb{A},c_\mathbb{A},\vec{\varsigma})$ to
$(\mathbb{P},\ell,c,\vec{\varpi})$ having the weak mixing property. 

In particular, $(\mathbb{A}, \lh_{\mathbb{A}}, c_{\mathbb{A}},\vec{\varsigma})$ is a $(\Sigma, \vec{\mathbb{S}})$-Prikry, $(\mathbb{A},\ell_\mathbb{A})$ has property $\mathcal{D}$, $\one_\mathbb{A}\forces_\mathbb{A}\mu=\check{\kappa}^+$ and $\vec{\varsigma}$ is a coherent sequence of nice projections.
\end{cor}
\begin{proof}
The first part follows from Corollary~\ref{PitchforkSuperNice} and Lemma~\ref{mixing}. Likewise, $(\mathbb{A},\ell_\mathbb{A})$ has property $\mathcal{D}$ by virtue of Lemma~\ref{propDsuccessor}, and $\vec{\varsigma}$ is coherent by virtue of Lemma~\ref{liftingcoherency} (see also Setup~\ref{setupkillingone}). Thus, we are left with arguing that $(\mathbb{A}, \lh_{\mathbb{A}}, c_{\mathbb{A}},\vec{\varsigma})$ is $(\Sigma, \vec{\mathbb{S}})$-Prikry and that $\one_\mathbb{A}\forces_\mathbb{A}\mu=\check{\kappa}^+$. All the Clauses of Definition~\ref{SigmaPrikry}  with the possible exception of \eqref{c2}, \eqref{c6} and
\eqref{moreclosedness} follow from Theorem~\ref{forkingSigmaPrikryLight}. First, from this latter result and the assumptions in Setup~\ref{setupkillingone},  $\one_\mathbb{A}\forces_\mathbb{A}\mu=\check{\kappa}^+$; Second, Clauses~\eqref{c2} and \eqref{moreclosedness} follow from Lemma~\ref{forkinganddirectedclosure}, Fact~\ref{directedclosedring} and Lemma~\ref{typeindeed}. Finally, Clause~\eqref{c6} is an immediate consequence of Lemma~\ref{propDsuccessor}.
\end{proof}

Our next task is to show that, after forcing with $\mathbb{A}$ the $\mathbb{P}$-name, $\dot{T}^+$ ceases to be stationary (Recall our blanket assumptions from page~\pageref{killingonesubsection}).

\begin{lemma}\label{kilingfragiles}
$\myceil{r^\star}{\mathbb{A}}\forces_\mathbb{A}``\dot{T}^+$ is nonstationary''.
\end{lemma}
\begin{proof}
Let $G$ be $\mathbb{A}$-generic over $V$, with $\myceil{r^\star}{\mathbb{A}}\in G$. Work in $V[G]$. Let $\bar{G}$ (resp. $H_n$)  denote the generic filter for $\mathbb{P}$ (resp. $\mathbb{S}_n$) induced by $\pi$ (resp. $\varsigma_n$) and $G$.   For all $a=(p, \vec{S})\in G$ and  $i\in\dom(\vec{S})$ write 
$$d^i_{a}:=\bigcup\{S_i(q, t)\mid q\in \bar{G}\cap W(p)\,\&\,\exists n\in[\ell(p),\ell(q)]\, (t\sle_n\varpi_n(q_n)\,\wedge\, t\in H_n)\},$$
where $\langle q_n\mid n\geq \ell(p)\rangle$ is the increasing enumeration of $\bar{G}\cap W(p)$ (see~Lemma~\ref{W(p)maxAntichain}).  

Then, let
$$d_{a}:=\begin{cases}
d^{\max(\dom(\vec{S}))}_{a}, & \text{if $\vec{S}\neq \emptyset$};\\
\emptyset, & \text{otherwise.}
\end{cases}
$$

\begin{claim}\label{claimdisjoint}
Suppose that  $a=(p, \vec{S})\in G$. 
In $V[\bar{G}]$,  for all $i\in\dom(\vec{S})$, the ordinal closure $\cl(d^i_{a})$ of  $d^i_{a}$ is disjoint from $(\dot{T}^+)_G$.
\end{claim}
\begin{proof}
To avoid trivialities we shall assume that $\vec{S}\neq \emptyset$. We prove the claim by induction on $i\in\dom(\vec{S})$. The base case $i=0$ is trivial, as $S_0\colon W(p)\rightarrow\{\emptyset\}$ (see Definition~\ref{strategy}\eqref{C5pstrategy}). So, let us assume by induction that $\cl(d^j_a)$ is disjoint from $(\dot{T}^+)_G$ for every $0\leq j<i$.

Let $\gamma\in \cl(d^i_a)\setminus  \bigcup_{j<i}\cl(d^j_a)$. By virtue of Clause~\eqref{C2ptree} of Definition~\ref{labeled-p-tree} applied to $S_i$, we may further assume that $\gamma\notin d^i_a$.

\smallskip

\underline{Succesor case:}  Suppose that $i=j+1$.  There are two cases to discuss:

\smallskip

$\br$ Assume $\cf^{V[\bar{G}]}(\gamma)=\omega$. Working in $V[\bar{G}]$, we have $\gamma=\sup_{n<\omega}\gamma_n$, where for each $n<\omega$, there is $(q_n,t_n)$, such that
$q_n\in\bar{G}\cap W(p)$, $t_n\sle_k\varpi_k(q_n)$ with $t_n\in H_k$ for some $k\in[\ell(p),\ell(q_n)]$, and $\gamma_n\in S_{j+1}(q_n,t_n)\setminus S_j(q_n,t_n)$. 
Strengthening if necessary, we may  assume  that $q_n+t_n\LE q_m+t_m$ for  $m\leq n$.\footnote{For this  we use Definition~\ref{labeled-p-tree}\eqref{C2ptree}. }

For each $n<\omega$ set $\delta_n:=\max(S_{j+1}(q_n,t_n))$. Clearly, $\gamma\leq \sup_{n<\omega}\delta_n$.

We claim that $\gamma=\sup_{n<\omega}\delta_n$: Assume to the contrary that this is not the case. Then, there is $n_0<\omega$ such that $\gamma_m<\delta_{n_0}$ for all $m<\omega$. Let $m\geq n_0$. Then, $\gamma_m\in S_{j+1}(q_m,t_m)\setminus S_j(q_m,t_m)$. Also, since $\gamma_{n_0}\notin S_{j}(q_{n_0},t_{n_0})$, 
Definition~\ref{strategy}\eqref{C3pstrategy} for $\vec{S}$  yields $\delta_0\in  S_{j+1}(q_{n_0},t_{n_0})\setminus S_{j}(q_{n_0},t_{n_0})$. By virtue of Clause~\eqref{C4pstrategy} of Definition~\ref{strategy} we also have
$$S_{j+1}(q_{n_0},t_{n_0})\setminus S_{j}(q_{n_0},t_{n_0})\sq S_{j+1}(q_{m},t_{m})\setminus S_{j}(q_{m},t_{m})\ni \gamma_m.$$
Thus, as $\gamma_m<\delta_0$, we have that $\gamma_m$ belongs to the left-hand-side set.

Since $m$ above was arbitrary we get $\gamma\in S_{j+1}(q_{n_0},t_{n_0})\s d^{i}_a$. This yields a contradiction with our original assumption that $\gamma\notin d^i_a$.

So, $\gamma=\sup_{n<\omega}\delta_n$. Now, let $n^\star<\omega$ such that $\ell(q_{n^\star})\geq \ell(p)+m(S_{j+1})$. Then, for all $n\geq n^\star$, Clause~\eqref{C4ptree} of Definition~\ref{labeled-p-tree} yields $(\delta_n, q_{n^\star})\in R$. In particular,  $(\gamma, q_{n^*})\in R$ and thus $q_{n^*}\forces_\mathbb{P}\check{\gamma}\notin \dot{T}^+$ (see page~\pageref{DefinitionofR}). Finally,  since $q_{n^*}\in \bar{G}$, we conclude that $\gamma\notin (\dot{T}^+)_G$, as desired.

\smallskip

$\br$ Assume $\cf^{V[\bar{G}]}(\gamma)\geq \omega_1$. Working in $V[\bar{G}]$, we have $\gamma=\sup_{\alpha<\cf(\gamma)}\gamma_\alpha$, where for each $\alpha<\cf(\gamma)$, there is $t_\alpha\in H_n$ with
$t_\alpha\sle_n\varpi_n(q)$ such that  $\gamma_\alpha\in S_{j+1}(q,t_\alpha)\setminus S_{j}(q,t_\alpha)$. Here, $q\in \bar{G}\cap W(p)$ and $n\in [\ell(p),\ell(q)]$.\footnote{Note that this is the case in that $W(p)$ is a tree with height $\omega$ and $\cf(\gamma)\geq \omega_1$.} 
By strengthening $q$ if necessary, we may also assume that $q\in W_{\geq m(S_{j+1})}(p)$.\footnote{Note that increasing $q$ would only increase $S_{j+1}(q,t_\alpha)$.}

For each $\alpha<\cf(\gamma)$, set $\delta_\alpha:=\max(S_{j+1}(q,t_\alpha))$. Clearly, $\gamma\leq  \sup_{\alpha<\cf(\gamma)}\delta_\alpha$.

We claim that  $\gamma= \sup_{\alpha<\cf(\gamma)}\delta_\alpha$: Otherwise, suppose $\alpha^*<\cf(\gamma)$ is such that $\gamma_\beta<\delta_{\alpha^*}$, for all $\beta< \cf(\gamma)$. Fix $\beta\geq \alpha^*$ and let $t\in H_n$ be such that $t\sle_n t_\beta, t_{\alpha^*}$.  Then, $q+t\LE q+t_\beta$, so  Definition~\ref{labeled-p-tree}\eqref{C2ptree} for $S_{j+1}$ yields $\gamma_\beta\in S_{j+1}(q,t_\beta)\s S_{j+1}(q,t)$.  
Hence, we have that $\gamma_\beta\in S_{j+1}(q,t) \setminus  S_{j}(q,t)$. Also, arguing as in the previous case we have $\delta_{\alpha^*}\in S_{j+1}(q,t_{\alpha^*})\setminus S_j(q,t_{\alpha^*})$.  Finally, combining Clause~\eqref{C4pstrategy} of Definition~\ref{strategy} with $\gamma_\beta<\delta_{\alpha^*}$ we conclude that
$\gamma_\beta\in S_{j+1}(q,t_{\alpha^*})$.
Since $S_{j+1}(q,t_{\alpha^*})$ is a closed set, we get  $\gamma\in S_{j+1}(q,t_{\alpha^*})$, which contradicts our assumption that $\gamma\notin d^i_a$.

So, $\gamma=\sup_{\alpha<\cf(\gamma)}\delta_\alpha$. Mimicking the argument of the former case 
we infer that $q\forces_\mathbb{P}\check{\gamma}\notin \dot{T}^+$, which yields $\gamma\notin (\dot{T}^+)_G$.

\smallskip

\underline{Limit case:} Suppose that $i$ is limit. If $\cf(i)\neq \cf(\gamma)$, then $\gamma\in \cl(d^j_a)$ for some $j<i$, and we are done.  Thus, suppose $\cf(i)=\cf(\gamma)$. For simplicity assume {$i=\cf(i)$}, as the general argument is analogous.
We have two cases.

$\br$ Assume $\cf^{V[\bar{G}]}(\gamma)=\omega$. Working in $V[\bar{G}]$, we have $\gamma=\sup_{n<\omega}\gamma_n$, where for each $n<\omega$, there is $(q_n,t_n)$, such that
$q_n\in\bar{G}\cap W(p)$, $t_n\sle_k\varpi_k(q_n)$ with $t_n\in H_k$ for some $k\in[\ell(p),\ell(q_n)]$, and $\gamma_n\in S_{\omega}(q_n,t_n)$. 
Strengthening if necessary, we may further  assume   $q_n+t_n\LE q_m+t_m$ for  $m\leq n$.

For each $n<\omega$ set $\delta_n:=\max(S_{\omega}(q_n,t_n))$. Clearly, $\gamma\leq \sup_{n<\omega}\delta_n$.

As in the previous cases, we claim that $\gamma=\sup_{n<\omega}\delta_n$: Suppose otherwise and let $n_0<\omega$ such that $\gamma_m<\delta_{n_0}$ for all $m<\omega$. Actually $\gamma<\delta_{n_0}$, as otherwise $\gamma\in S_\omega(q_{n_0},t_{n_0})\s d^\omega_a$, which would yield a contradiction.

By Clause~\eqref{C5pstrategy} of Definition~\ref{strategy}, $\delta_{n_0}=\sup_{k<\omega}\max(S_k(q_{n_0},t_{n_0}))$, hence there is some $k_0<\omega$ such that $\gamma<\max(S_{k_0}(q_{n_0},t_{n_0}))$.

Fix $m\geq n_0$. Since $q_{m}+t_m\LE q_{n_0}+t_{n_0}$, Clause~\eqref{C2ptree} of Definition~\ref{labeled-p-tree} yields
$$\gamma<\max(S_{k_0}(q_{n_0},t_{n_0}))\leq \max(S_{k_0}(q_{m},t_{m})).$$
Also, Clause~\eqref{C3pstrategy} of Definition~\ref{strategy}  implies that $$ S_{k_0}(q_{m},t_{m})\sq S_\omega(q_{m},t_{m})\ni \gamma_m,$$
so that, $\gamma_m\in S_{k_0}(q_m,t_m)$. Since $m$ was arbitrary, we infer that $\gamma\in \cl(d^{k_0}_a)$, which yields a contradiction with our original assumption.

So, $\gamma=\sup_{n<\omega}\delta_n$. Arguing as in previous cases conclude that $\gamma\notin (\dot{T}^+)_G$.

\smallskip

$\br$ Assume $\cf^{V[\bar{G}]}(\gamma)\geq \omega_1$. Working in $V[\bar{G}]$, we have $\gamma=\sup_{\alpha<i}\gamma_\alpha$, where for each $\alpha<i$, there is $t_\alpha\in H_n$ with
$t_\alpha\sle_n\varpi_n(q)$ such that  $\gamma_\alpha\in S_{i}(q,t_\alpha)$. As before, here both $q$ and $n$ are fixed and $q\in W_{\geq m(S_{i})}(p)$.

For each $\alpha<i$, set $\delta_\alpha:=\max(S_i(q,t_\alpha))$. Once again, we aim to show that $\gamma=\sup_{\alpha<i}\delta_\alpha$:  Suppose that this is not the case, and let $\alpha^*<i$ such that $\gamma_\alpha<\delta_{\alpha^*}$ for all $\alpha<i$.  As before, $\gamma\neq \delta_{\alpha^*}$, so there is some $\bar{\alpha}<i$ such that $\gamma<\max(S_{\bar{\alpha}}(q,t_{\alpha^*}))$ Now, let $\alpha<i$ be arbitrary and find $s_\alpha\sle_n t_{\alpha^*}, t_\alpha$ in $H_n$.  Then, $\gamma_\alpha\in S_i(q,t_\alpha)\s S_i(q,s_\alpha)$. Also, $S_{\bar{\alpha}}(q,t_{\alpha^*})\s S_{\bar{\alpha}}(q,s_\alpha)$ and so $\max(S_{\bar{\alpha}}(q,s_\alpha))>\gamma$. By Clause~\eqref{C3pstrategy} of Definition~\ref{strategy} we have that $S_{\bar{\alpha}}(q,s_\alpha)\sq S_i(q,s_\alpha)$, hence $\gamma_\alpha\in S_{\bar{\alpha}}(q,s_\alpha)$.

The above shows that $\gamma\in \cl(d^{\bar{\alpha}}_a)$, which is a contradiction.

 So, $\gamma=\sup_{\alpha<i}\delta_\alpha$. Now proceed as in previous cases, invoking Clause~\eqref{C4ptree} of Definition~\ref{labeled-p-tree}, and infer that $\gamma\notin (\dot{T}^+)_G$.
\end{proof}

\begin{claim}\label{unbounded}
Suppose $a=(p, \vec{S})\in A$, where  $p\LE r^\star$.   For every  $\gamma<\mu$,  there exists $\bar{\gamma}\in(\gamma,\mu)$ and $(\bar{p}, \vec{T})\unlhd (p,\vec{S})$, such that $\max(\dom(\vec{T}))=\alpha$ and for all $(q,t)\in\dom(T_\alpha)$, $\max(T_\alpha(q,t))=\bar{\gamma}$.
\end{claim}
\begin{proof}
This is indeed what the argument of Lemma~\ref{typeindeed} shows.
\end{proof}
Working in $V[G]$, the above claim yields an unbounded set $I\s\mu$ such that for each $\gamma\in I$ there is $a_\gamma=(p_\gamma, \vec{S}^\gamma)\in G$ with $\max(\dom(\vec{S}^\gamma))=\gamma$ and  $\max(S_\gamma^\gamma(q,t))=\gamma$ for all $(q,t)\in\dom (S_\gamma^\gamma)$. For each $\gamma\in I$, 
set $D_\gamma:=\cl(d_{a_\gamma}).$ 

\begin{claim}\label{claimcoherent}
For each $\gamma<\gamma'$ both in $I$, $D_\gamma\sqsubseteq D_{\gamma'}$.
\end{claim}
\begin{proof}
Let $\gamma<\gamma'$ be in $I$. It is enough to prove that $d_{a_\gamma}\sq d_{a_{\gamma'}}$; namely, we  show that
$d_{a_\gamma}= d_{a_{\gamma'}}\cap (\gamma+1)$.
Fix  $b=(r,\vec{R})\in G$ be such that $b\unlhd a_\gamma,a_{\gamma'}$.

For the first direction, suppose that $\delta\in d_{a_\gamma}$ and let $(q,t)\in\dom(S_\gamma^\gamma)$ be a pair witnessing this. By strengthening $q$ and $t$ if necessary, we may further assume that $\ell(q)\geq \ell(r)$ and $t\in H_n$, $t\sle_n \varpi_n(q)$, where $n:=\ell(q)$.\footnote{Suppose that $(q,t)$ is the pair we are originally given and that $q'\in W_n(p)\cap \bar{G}$, where $n\geq \ell(r)$. Setting $t':=\varpi_n(q'+t)$ it is immediate that $q'+t'\LE q+t$, hence $\delta\in S^\gamma_\gamma(q',t')$. Also, it is not hard to check that $q'+t\in G$, hence $t'\sle_n \varpi_n(q')$ and $t'\in H_n$. } 

Let $r'\in W(r)\cap \bar{G}$ be the unique condition with $\ell(r')=n$. Also, let $t'\in H_n$ be such that $t'\sle_n \varpi_n(r'), t$. Then $w(p_\gamma,r')=q$, $t'=\varpi_n(r'+t')$ and $q+t'\LE q+t$. So, by  $b\unlhd a_\gamma$ and $b\unlhd a_{\gamma'}$, we get:
$$\delta\in S^\gamma_\gamma(q,t)\s S^\gamma_\gamma(q,t')=R_\gamma(r',t') = S^{\gamma'}_\gamma(w(p_\gamma, r'), t')\s d_{a_{\gamma'}}.$$

For the other direction, suppose that $\delta\in d_{a_{\gamma'}}\cap (\gamma+1)$ and let $(q,t)$ be a pair in $\dom(S_{\gamma'}^{\gamma'})$ witnessing this. Again, by strengthening $q$, we may assume that  $\ell(q)\geq \ell(r)$ and $t\in H_n$, $t\sle_n \varpi_n(q_n)$, where $n:=\ell(q)$. Similarly as above,
let $r'\in W(r)\cap \bar{G}$ be with $\ell(r')=n$, and $t'\in H_n$ be such that $t'\sle_n \varpi_n(r'), t$. Then $w(p_{\gamma'},r')=q$, $t'=\varpi_n(r'+t')$, $q+t'\LE q+t$ and:
\begin{enumerate}
\item
$R_{\gamma'}(r',t')=S^{\gamma'}_{\gamma'}(q,t')$, since $b\unlhd a_{\gamma'}$;
\item
$ R_{\gamma}(r',t')=S^\gamma_\gamma(w(p_\gamma, r'), t')$, and so $\gamma=\max( R_{\gamma}(r',t'))$;
\item
$R_{\gamma}(r',t')\sqsubseteq R_{\gamma'}(r',t')$, by Clause~\eqref{C3pstrategy} of Definition~\ref{strategy} for $\vec{R}$.
\end{enumerate}

Combining all three, we get that
\begin{eqnarray*}
\delta\in S^{\gamma'}_{\gamma'}(q,t)\cap (\gamma+1)\s S^{\gamma'}_{\gamma'}(q,t')\cap (\gamma+1) = \\
R_{\gamma'}(r', t')\cap (\gamma+1) = R_\gamma(r',t')=S^\gamma_\gamma(w(p_\gamma, r'), t')\s d_{a_\gamma},
\end{eqnarray*}
as desired.
\end{proof}
Let $D:=\bigcup_{\gamma\in I} D_\gamma$. By Claims~\ref{claimdisjoint} and \ref{claimcoherent}, $D$ is disjoint from $(T^+)_G$. Additionally, Claim~\ref{claimcoherent} implies that $D$ is closed and, since $I\s D$,  it is also unbounded. So, $(\dot{T}^+)_G$ is nonstationary in $V[G]$.
\end{proof}

\begin{remark}\label{RemarkKillingT}
 Note that Lemma~\ref{kilingfragiles} together with $r^\star\forces_\mathbb{P}\dot{T}\s \dot{T}^+$ (see~page~\pageref{killingonesubsection})  imply that $\myceil{r^\star}{\mathbb{A}}\forces_\mathbb{A}``\dot{T}$ is nonstationary''.
\end{remark}

The next corollary sums up the content of Subsection~\ref{killingonesubsection}:
\begin{cor}\label{onestep}
Suppose that $(\Sigma,\vec{\mathbb{S}})$-Prikry quadruple $(\mathbb P,\lh,c,\vec{\varpi})$ such that,  $\mathbb P=\left(P,\le\right)$ is a subset of $H_{\mu^+}$, $(\mathbb{P},\ell)$ has property $\mathcal{D}$, $\vec{\varpi}$ is a coherent sequence of nice projections,
$\one_{\mathbb P}\forces_{\mathbb P}\check\mu=\check\kappa^+$
and $\one_{\mathbb P}\forces_{\mathbb{P}}``\kappa\text{ is singular}"$.

For every $r^\star\in P$ and a $\mathbb P$-name $z$ for an $r^\star$-fragile stationary subset of $\mu$,
there are a $(\Sigma,\vec{\mathbb{S}})$-Prikry quadruple $(\mathbb A,\lh_{\mathbb A},c_{\mathbb A},\vec{\varsigma})$ having property $\mathcal{D}$, and a pair of maps  $(\pitchfork,\pi)$   such that all the following hold:
\begin{enumerate}[label=(\alph*)]
\item\label{thusonestep1} $(\pitchfork,\pi)$ is  a super nice  forking projection from  $(\mathbb A,\lh_{\mathbb A},c_{\mathbb A},\vec{\varsigma})$ to $(\mathbb P,\lh,c,\vec{\varpi})$ that has the weak mixing property;
\item\label{thusonestep4} $\vec{\varsigma}$ is a coherent sequence of nice projections;
\item\label{thusonestep2} $\one_{\mathbb A}\forces_{\mathbb A}\check\mu=\check\kappa^+$;
\item\label{thusonestep3} $\mathbb A=(A,{\unlhd})$ is a subset of $H_{\mu^+}$;
\item For every $n<\omega$, $\z{\mathbb{A}}^\pi_n$ is a $\mu$-directed-closed;\label{thusonestep6}
\item\label{thusonestep4} $\myceil{r^\star}{\mathbb A}$ forces that $z$ is nonstationary.
\end{enumerate}
\end{cor}
\begin{proof}  Since all the assumptions of Setup~\ref{setupkillingone} are valid we obtain from Definitions~\ref{SharonNew} and \ref{d45},
a notion of forcing $\mathbb A=(A,{\unlhd})$ together with maps $\lh_{\mathbb A}$ and $c_{\mathbb A}$,
and a sequence $\vec{\varsigma}$ such that, by Corollary~\ref{AisweaklySigmaPrikry}, $(\mathbb A,\lh_{\mathbb A},c_{\mathbb A},\vec{\varsigma})$
is a $(\Sigma,\vec{\mathbb{S}})$-Prikry quadruple having property $\mathcal{D}$ and Clauses~\ref{thusonestep1}--\ref{thusonestep2} above hold.
Clause~\ref{thusonestep3} easily follows  from the definition of  $\mathbb A=(A,{\unlhd})$ 
(see, e.g. \cite[Lemma~6.6]{PartI}), Clause~\ref{thusonestep6} is Fact~\ref{directedclosedring} and Clause~\ref{thusonestep4} is Lemma~\ref{kilingfragiles} together with Remark~\ref{RemarkKillingT}.
\end{proof}

\subsection{Fragile sets vs non-reflecting stationary sets}\label{subsection62}
For every $n<\omega$, denote
$\Gamma_n:=\{\alpha<\mu\mid \cf^V(\alpha)<\sigma_{n-2}\}$,
where, by convention, we define $\sigma_{-2}$ and $\sigma_{-1}$ to be $\aleph_0$. \label{NotationGamma}

The next lemma is an analogue of \cite[Lemma 6.1]{PartI} and will be crucial for the proof of reflection in the model of the Main Theorem.

\begin{lemma}\label{key} Suppose  that:
\begin{enumerate}[label=(\roman*)]
\item for every $n<\omega$, $V^{\mathbb P_n}\models \refl(E^\mu_{<\sigma_{n-2}},E^\mu_{<\sigma_{n}})$;
\item $r^\star$ is a condition in $\mathbb P$;
\item $\dot T$ is a nice $\mathbb P$-name for a subset of $\Gamma_{\lh(r^\star)}$;
\item  $r^\star$ $\mathbb P$-forces that $\dot T$ is a non-reflecting stationary set.
\end{enumerate}
Then $\dot T$ is $r^\star$-fragile.
\end{lemma}
\begin{proof}
Suppose that $\dot{T}$ is not $r^\star$-fragile (see Definition~\ref{fragilestationary}), and let $q$ be an extension of $r^\star$ witnessing that.
Set $n:=\lh(q)$, so that $$q\Vdash_{\mathbb{P}_n}``\dot{T}_n\text{ is stationary}".$$
Since $\dot T$ is a nice  $\mathbb P$-name for a subset of $\Gamma_{\lh(r^\star)}$,
it altogether follows that $q$ $\mathbb P_n$-forces that $\dot{T}_n$ is a stationary subset of $E^\mu_{<\sigma_{n-2}}$.

Let $G_n$ be $\mathbb{P}_n$-generic containing $q$.
By Clause~$(i)$, we have that $T_n:=(\dot{T}_n)_{G_n}$ reflects at some ordinal $\gamma$ of ($V[G_n]$-)cofinality $<\sigma_{n}$.
Since $\varpi_n$ is a nice projection, we have that $\mathbb{P}^{\varpi_n}_n\times \mathbb{S}_n$ projects to $\mathbb{P}_n$.\footnote{More precisely, $(\mathbb{P}^{\varpi_n}_n\downarrow q)\times (\mathbb{S}_n\downarrow \varpi_n(q))$ projects onto $\mathbb{P}_n\downarrow q$.}
Then by $|S_n|<\sigma_n$ and the fact that $\mathbb{P}^{\varpi_n}_n$ contains a $\sigma_n$-directed-closed dense subset,  it follows that $\theta:=\cf^V(\gamma)$
is $<\sigma_n$. In $V$, fix a club $C\s \gamma$ of order-type $\theta$.

Work in $V[G_n]$. Set $A:=T_n\cap C$, and note that $A$ is a stationary subset of $\gamma$ of size $\le\theta$.
Let $H_n$ be the $\mathbb{S}_n$-generic filter induced from $G_n$ by $\varpi_n$.

Again, since $\mathbb{P}^{\varpi_n}_n$ contains a $\sigma_n$-directed-closed  dense subset, it cannot have added $A$. So, $A\in V[H_n]$.
Let $\langle \alpha_i\mid i<\theta\rangle$ be some enumeration (possibly with repetitions) of $A$,
and let $\langle\dot{\alpha}_i\mid i<\theta\rangle$ be a sequence of  $\mathbb S_n$-name for it.
Pick a condition $r$ in $\mathbb{P}_n/H_n$ such that $r\Vdash_{\mathbb{P}_n}\dot{A}\s\dot{T}_n\cap\gamma$ and such that $\varpi_n(r)\Vdash_{\mathbb{S}_n}\dot{A}=\{\dot{\alpha}_i\mid i<\theta\}$.
Denote $s:=\varpi_n(r)$ and note that $s\in H_n$.
We now go back and work in $V$.
\begin{claim}
Let $i<\theta$ and $\alpha<\gamma$. For all $r'\leq^{\varpi_n} r$ and $s'\sle_n s$, if $s'\Vdash_{\mathbb{S}_n}\dot{\alpha}_i=\check\alpha$,
then there are $r''\leq^{\varpi_n} r'$ and $s''\sle_n s'$ such that $r''+s''\Vdash_{\mathbb{P}}\check\alpha\in \dot{T}$.
\end{claim}
\begin{proof}
Suppose $r',s'$ are as above. As $r'$ extends $r$ and $s'$ extends $s$,
it follows that $r'+s'\Vdash_{\mathbb{P}_n}\check\alpha\in\dot{T}_n$ and $s'\Vdash_{\mathbb{S}_n}\check\alpha\in\dot{A}$. Let $p\leq^0 r'+s'$ be such that $(\check{\alpha}, p)\in\dot{T}_n$. By the definition of the name $\dot{T}_n$, 
we have that  $p\Vdash_{\mathbb{P}}\check\alpha\in\dot{T}$. Now, since $\varpi_n$ is a nice projection 
Definition~\ref{niceprojection}\eqref{theprojection} gives $s''\sle_n s'$ and $r''\leq^{\varpi_n} r'$  such that $r''+s''=p$.
So $r''+s''\Vdash_{\mathbb{P}}\check \alpha\in \dot{T}$, as desired.
\end{proof}
Fix an injective enumeration $\langle (i_\xi,s_\xi)\mid \xi<\chi\rangle$ of $\theta\times(\mathbb{S}_n\downarrow s)$.
Note that $\chi<\sigma_n$. Using that $\mathbb{P}^{\varpi_n}_n$ is $\sigma_n$-strategically-closed (in $V[H_n]$), 
 build  a $\leq^{\varpi_n}$-decreasing sequence of conditions $\langle r_\xi\mid  \xi\le\chi\rangle$,
such that, for every $\xi<\chi$, $r_\xi\LE^{\vec\varpi} r$, and,
for any $\alpha<\gamma$, if $s_\xi\Vdash_{\mathbb{S}_n}\dot{\alpha}_{i_\xi}=\check\alpha$ (and $s_\xi\in H_n$),
then there is $s^\xi\sle_n s_\xi$ (with $s^\xi\in H_n$) such that $r_\xi+s^\xi\Vdash_{\mathbb{P}}\check\alpha\in \dot{T}$.
Finally, let $r^*:=r_\chi$.
Note that $\varpi_n(r^*)=\varpi_n(r)=s\in H_n$, so that $r^*\in P/H_n$.
\begin{claim}
$r^*\Vdash_{\mathbb{P}/H_n}A\s\dot{T}\cap\check\gamma$.
\end{claim}
\begin{proof} 
Let $r\leq_{\mathbb{P}/H_n} r^*$ and $i<\theta$. By extending $\varpi_n(r)$ if necessary we may assume that $\varpi_n(r)\in H$ and that it decices  (in $\mathbb{S}_n$)
 $\dot{\alpha}_i$ to be some ordinal $\alpha<\gamma$.
Fix $\xi<\chi$ such that $(i_\xi,s_\xi)=(i,s')$.
By the construction, there is $s^\xi\sle_n s_\xi$, $s^\xi\in H_n$  such that  $r_\xi+s^\xi\Vdash_{\mathbb{P}}\check\alpha\in \dot{T}$. Hence, $r^*+s^\xi\Vdash_{\mathbb{P}}\check\alpha\in \dot{T}\cap\check\gamma$, and thus $r+s^\xi$ forces the same. Since $r+s^\xi\in P/H_n$  we are done.
\end{proof}

Finally, since $(\mathbb P,\lh,c,\vec{\varpi})$ is $(\Sigma, \vec{\mathbb{S}})$-Prikry,  Lemma~\ref{l14}\eqref{l14(1)}
implies that $\mathbb{P}/H_n$ does not add any new subsets of $\theta$ and, incidentally, no new subsets of $C$. Hence $\mathbb{P}/H_n$ preserves the stationarity of $A$,
and thus the stationarity of $T\cap\gamma$.
This contradicts hypothesis~$(iv)$ of the lemma.
\end{proof}

\section{Iteration scheme}\label{Iteration}
In this section, we define an iteration scheme for  $(\Sigma,\vec{\mathbb{S}})$-Prikry forcings,
following closely and expanding the work from \cite[\S3]{PartII}. The reader familiar with our iteration machinery may opt for reading just Lemma~\ref{CvIteration}. There we use in a crucial way a new feature of the forking projections; namely, super niceness.
\begin{setup}\label{setupiteration} The blanket assumptions for this section are as follows:
\begin{itemize}
\item $\mu$ is some cardinal satisfying $\mu^{<\mu}=\mu$, so that $|H_\mu|=\mu$;
\item $\langle (\sigma_n,\sigma_n^*)\mid n<\omega\rangle$ is a sequence of pairs of regular uncountable cardinals,
such that, for every $n<\omega$, $\sigma_n\le\sigma_n^*\le\mu$ and $\sigma_n\le\sigma_{n+1}$;
\item $\vec{\mathbb S}=\langle \mathbb S_n\mid n<\omega\rangle$ is a sequence of notions of forcing, $\mathbb S_n=(S_n,\SLE_n)$,
with $|S_n|< \sigma_n$;
\item $\Sigma:=\langle \sigma_n\mid n<\omega\rangle$ and $\kappa:=\sup_{n<\omega}\sigma_n$.
\end{itemize}
\end{setup}

The following convention will be applied hereafter:
\begin{conv}\label{conv71}  For a pair of ordinals $\gamma\le\alpha\le\mu^+$:
\begin{enumerate}
\item  $\emptyset_\alpha:=\alpha\times\{\emptyset\}$ denotes the $\alpha$-sequence with constant value $\emptyset$;
\item  For a $\gamma$-sequence $p$ and an $\alpha$-sequence $q$, $p*q$ denotes the unique $\alpha$-sequence satisfying that for all $\beta<\alpha$:
$$(p*q)(\beta)=\begin{cases}
q(\beta),&\text{if }\gamma\le\beta<\alpha;\\
p(\beta),&\text{otherwise}.
\end{cases}$$
\item Let $\mathbb{P}_\alpha:=(P_\alpha,\LE_\alpha)$ and $\mathbb{P}_\gamma:=(P_\gamma,\LE_\gamma)$ be forcing posets such that $P_\alpha\s{}^\alpha{H_{\mu^+}}$ and $P_\gamma\s{}^\gamma{H_{\mu^+}}$. Also, assume $p\mapsto p\upharpoonright\gamma$ defines a projection between $\mathbb{P}_\alpha$ and $\mathbb{P}_\gamma$. We denote by $i^\alpha_\gamma: V^{\mathbb{P}_\gamma}\rightarrow V^{\mathbb{P}_\alpha}$ the map defined by recursion over the rank of each $\mathbb{P}_\gamma$-name $\sigma$ as follows:
$$i^\alpha_\gamma(\sigma):=\{(i^\alpha_\gamma(\tau),p*\emptyset_\alpha)\mid (\tau,p)\in \sigma\}.$$
\end{enumerate}
\end{conv}

Our iteration scheme requires three building blocks:

\blk{I} We are given a $(\Sigma,\vec{\mathbb{S}})$-Prikry forcing $(\mathbb{Q},\lh,c,\vec{\varpi})$ such that $(\mathbb{Q},\lh)$ satisfies property $\mathcal{D}$. We moreover assume that $\mathbb Q=(Q,\le_Q)$ is a subset of $H_{\mu^+}$, $\one_{\mathbb Q}\forces_{\mathbb Q}``\check\mu=\check\kappa^+\,\&\, \kappa\text{ is singular}"$
and  $\vec{\varpi}$ is a coherent sequence. 
To streamline the matter, we also require that $\one_{\mathbb Q}$ be equal to $\emptyset$.

\blk{II}
Suppose that $(\mathbb P,\lh_{\mathbb P},c_{\mathbb P},\vec{\varpi})$ is a $(\Sigma,\vec{\mathbb{S}})$-Prikry quadruple  having property $\mathcal{D}$ 
such that
$\mathbb P=\left(P,\le\right)$ is a subset of $H_{\mu^+}$, $\vec{\varpi}$ is a coherent sequence of nice projections,
$\one_{\mathbb P}\forces_{\mathbb P}``\check\mu=\check\kappa^+$'' and $ \one_{\mathbb P}\forces_{\mathbb P} ``\check{\kappa}\text{ is singular}"$.

For every $r^\star\in P$, and a $\mathbb P$-name $z\in H_{\mu^+}$,
we are given a   corresponding $(\Sigma,\vec{\mathbb{S}})$-Prikry quadruple $(\mathbb A,\lh_{\mathbb A},c_{\mathbb A},\vec{\varsigma})$ having property $\mathcal{D}$ such that: 
\begin{enumerate}[label=(\alph*)]
\item\label{C1blk2}
there is a  super nice forking projection $(\pitchfork,\pi)$ from
$(\mathbb A,\lh_{\mathbb A},c_{\mathbb A},\vec{\varsigma})$   to $(\mathbb P,\lh_{\mathbb P},c_{\mathbb P},\vec{\varpi})$ that has the weak mixing property;
\item\label{C5blk2} $\vec{\varsigma}$ is a coherent sequence of nice projections; 
\item\label{C4blk2} for every $n<\omega$, $\z{\mathbb{A}}^\pi_n$ is $\sigma^*_n$-directed-closed;\footnote{$\z{A}_n$ is the poset given in Definition~\ref{type}\eqref{type5} defined with respect to the type map witnessing Clause~\ref{C1blk2} above.}
\item\label{C2blk2} $\one_{\mathbb A}\forces_{\mathbb A}\check\mu=\check\kappa^+$;
\item\label{C3blk2} $\mathbb A=(A,\unlhd)$ is a subset of $H_{\mu^+}$;
\setcounter{condition}{\value{enumi}}
\end{enumerate}
By \cite[Lemma 2.18]{PartII}, we may also require that:
\begin{enumerate}[label=(\alph*)]
\setcounter{enumi}{\value{condition}}
\item\label{C5blk2} each element of $A$ is a pair $(x,y)$ with $\pi(x,y)=x$;
\item\label{C6blk2} for every $a\in A$, $\myceil{\pi(a)}{\mathbb A}=(\pi(a),\emptyset)$;
\item\label{C7blk2} for every $p,q\in P$, if $c_{\mathbb P}(p)=c_{\mathbb P}(q)$, then $c_{\mathbb A}(\myceil{p}{\mathbb A})=c_{\mathbb A}(\myceil{q}{\mathbb A})$.
\end{enumerate}

\blk{III} We are given a function $\psi:\mu^+\rightarrow H_{\mu^+}$.

\begin{goal}\label{goals} Our goal is to define a system $\langle  (\mathbb{P}_\alpha,\lh_\alpha,c_\alpha,\vec{\varpi}_\alpha,\langle\pitchfork_{\alpha,\gamma}\mid \gamma\le\alpha\rangle)\mid \alpha\le\mu^+\rangle$ in such a way that for all $\gamma\le\alpha\le\mu^+$:
\begin{enumerate}
\item[(i)] $\mathbb P_\alpha$ is a poset $(P_\alpha,\le_\alpha)$, $P_\alpha\s{}^{\alpha}H_{\mu^+}$, and, for all $p\in P_\alpha$, $|B_p|<\mu$, where $B_p:=\{ \beta+1\mid \beta\in\dom(p)\ \&\ p(\beta)\neq\emptyset\}$;
\item[(ii)] The map $\pi_{\alpha,\gamma}:P_\alpha\rightarrow P_\gamma$ defined by $\pi_{\alpha,\gamma}(p):=p\restriction\gamma$ forms an projection from $\mathbb P_\alpha$ to $\mathbb P_\gamma$ and $\lh_\alpha= \lh_\gamma\circ \pi_{\alpha,\gamma}$; 
\item[(iii)] $\mathbb P_0$ is a trivial forcing, $\mathbb P_1$ is isomorphic to $\mathbb Q$ given by \blkref{I},
and $\mathbb P_{\alpha+1}$ is isomorphic to $\mathbb A$ given by \blkref{II} when invoked with respect to $(\mathbb P_\alpha,\lh_\alpha,c_\alpha,\vec{\varpi}_\alpha)$ and a pair $(r^\star,z)$ which is decoded from $\psi(\alpha)$;
\item[(iv)] If $\alpha>0$, then  $(\mathbb P_\alpha,\lh_\alpha,c_\alpha,\vec{\varpi}_\alpha)$ is a $(\Sigma,\vec{\mathbb{S}})$-Prikry  notion  of forcing satisfying property $\mathcal{D}$, whose greatest  element is $\emptyset_\alpha$, $\lh_\alpha=\lh_1\circ \pi_{\alpha,1}$ and $\emptyset_\alpha\forces_{\mathbb P_\alpha}\check\mu=\check\kappa^+$. Moreover, $\vec{\varpi}_\alpha$ is a coherent sequence of nice projections such that $\vec{\varpi}_\alpha=\vec{\varpi}_\gamma\bullet \pi_{\alpha,\gamma}$ for every $\gamma\leq \alpha$;
\item[(v)] If $0<\gamma< \alpha\le\mu^+$, then $(\pitchfork_{\alpha,\gamma},\pi_{\alpha,\gamma})$ is a nice forking projection from $(\mathbb P_\alpha,\lh_\alpha,\vec{\varpi}_\alpha)$  to $(\mathbb P_\gamma,\lh_\gamma,\vec{\varpi}_\gamma)$;
in case $\alpha<\mu^+$, $(\pitchfork_{\alpha,\gamma},\pi_{\alpha,\gamma})$ is furthermore a nice forking projection from $(\mathbb P_\alpha,\lh_\alpha,c_\alpha,\vec{\varpi}_\alpha)$ to $(\mathbb P_\gamma,\lh_\gamma,c_\gamma,\vec{\varpi}_\gamma)$, and in case $\alpha=\gamma+1$,  $(\pitchfork_{\alpha,\gamma},\pi_{\alpha,\gamma})$ is super nice and has the weak mixing property;
\item[(vi)] If $0<\gamma\le\beta\le\alpha$, then, for all $p\in P_{\alpha}$ and $r\le_\gamma p\restriction\gamma$, $\fork[\beta,\gamma]{p\restriction\beta}(r)=(\fork[\alpha,\gamma]{p}(r))\restriction \beta$.
\end{enumerate}
\end{goal}

\subsection{Defining the iteration}\label{DefiningTheIteration}
For every $\alpha<\mu^+$, fix an injection $\phi_\alpha:\alpha\rightarrow\mu$.
As $|H_\mu|=\mu$, by the Engelking-Kar\l owicz theorem, we may also fix a sequence $\langle e^i\mid i<\mu\rangle$ of functions from $\mu^+$ to $H_\mu$ such that
for every function $e:C\rightarrow H_\mu$ with $C\in[\mu^+]^{<\mu}$, there is $i<\mu$ such that $e\s e^i$.

The upcoming definition is by recursion on $\alpha\le\mu^+$, and we continue as long as we are successful. 

$\br$ Let $\mathbb P_0:=(\{\emptyset\},\le_0)$ be the trivial forcing.
Let $\lh_0$ and $c_0$ be the constant function $\{(\emptyset,\emptyset)\}$ and $\vec{\varpi}_0=\langle \{(\emptyset,\one_{\mathbb{S}_n})\}\mid n<\omega\rangle$. Finally, let $\pitchfork_{0,0}$ be the constant function $\{ (\emptyset,\{(\emptyset,\emptyset)\})\}$,
so that $\fork[0,0]{\emptyset}$ is the identity map.

$\br$ Let $\mathbb P_1:=(P_1,\le_1)$, where $P_1:={}^1Q$ and  $p\le_1 p'$ iff $p(0)\le_Q p'(0)$.
Evidently, $p\overset{\iota}{\mapsto}p(0)$ form an isomorphism between $\mathbb P_1$ and $\mathbb Q$,
so we naturally define  $\lh_{1}:=\lh\circ\iota$, $c_1:=c\circ\iota$ and $\vec{\varpi}_1:=\vec{\varpi}\bullet\iota$.
Hereafter, the sequence $\vec{\varpi}_1$ is denoted by $\langle \varpi^1_n\mid n<\omega\rangle$. For all $p\in P_1$, let $\fork[1,0]{p}:\{\emptyset\}\rightarrow\{p\}$ be the constant function,
and let $\fork[1,1]{p}$ be the identity map.

$\br$ Suppose $\alpha<\mu^+$ and that $\langle (\mathbb P_\beta,\lh_\beta,c_\beta,\vec{\varpi}_\beta,\langle \pitchfork_{\beta,\gamma}\mid\gamma\le\beta\rangle)\mid\beta\le\alpha\rangle$ has already been defined.
We now define $(\mathbb{P}_{\alpha+1},\lh_{\alpha+1}, c_{\alpha+1},\vec{\varpi}_{\alpha+1})$ and $\langle \pitchfork_{\alpha+1,\gamma}\mid\gamma\leq \alpha+1\rangle$.

$\br\br$ If $\psi(\alpha)$ happens to be a triple $(\beta,r,\sigma)$, where $\beta<\alpha$, $r\in P_\beta$ and $\sigma$ is a $\mathbb P_\beta$-name,
then we appeal to \blkref{II} with  $(\mathbb P_\alpha,\lh_\alpha,c_\alpha,\vec{\varpi}_\alpha)$,
$r^\star:=r*\emptyset_\alpha$ and $z:=i^\alpha_\beta(\sigma)$
to get a corresponding $(\Sigma,\vec{\mathbb{S}})$-Prikry quadruple $(\mathbb A,\lh_{\mathbb A},c_{\mathbb A},\vec{\varsigma})$.

$\br\br$ Otherwise, we obtain $(\mathbb A,\lh_{\mathbb A},c_{\mathbb A},\vec{\varsigma})$ by appealing to \blkref{II} with  $(\mathbb P_\alpha,\lh_\alpha,c_\alpha,\vec{\varpi}_\alpha)$, $r^\star:=\emptyset_\alpha$ and $z:=\emptyset$.

In both cases, we  obtain a super nice forking projection $(\pitchfork,\pi)$ from $(\mathbb A,\lh_{\mathbb A},c_{\mathbb A},\vec{\varsigma})$ to $(\mathbb P_\alpha,\lh_\alpha,c_\alpha,\vec{\varpi}_\alpha)$.  Furthermore, each condition in $\mathbb A=(A,\unlhd)$ is a pair $(x,y)$ with $\pi(x,y)=x$, and, for every $p\in P_\alpha$, $\myceil{p}{\mathbb A}=(p,\emptyset)$.
Now, define $\mathbb P_{\alpha+1}:=(P_{\alpha+1},\le_{\alpha+1})$ by letting $P_{\alpha+1}:=\{ x {}^\smallfrown\langle y\rangle \mid (x,y)\in A\}$,
and then letting $p\le_{\alpha+1} p'$ iff $(p\restriction\alpha,p(\alpha))\unlhd (p'\restriction\alpha,p'(\alpha))$.
Put $\lh_{\alpha+1}:=\lh_1\circ \pi_{\alpha+1,1}$ and define $c_{\alpha+1}:P_{\alpha+1}\rightarrow H_\mu$ via $c_{\alpha+1}(p):=c_{\mathbb A}(p\restriction\alpha,p(\alpha))$.

Let $\vec{\varpi}_\alpha=\langle\varpi^\alpha_n\mid n<\omega\rangle$ be defined in the natural way, i.e.,  for each $n<\omega$ and $x{}^\smallfrown \langle y\rangle \in (P_\alpha)_{\geq n}$, we set $\varpi^\alpha_n(x{}^\smallfrown \langle y\rangle):=\varsigma_n(x,y)$.

Next, let $p\in P_{\alpha+1}$, $\gamma\le\alpha+1$ and $r\LE_\gamma p\restriction\gamma$ be arbitrary; we need to define $\fork[\alpha+1,\gamma]{p}(r)$.
For $\gamma=\alpha+1$, let $\fork[\alpha+1,\gamma]{p}(r):=r$, and for $\gamma\le\alpha$, let\label{pitchforkatsuccessors}
$$\fork[\alpha+1,\gamma]{p}(r):=x{}^\smallfrown\langle y\rangle\text{ iff }\fork{p\restriction\alpha,p(\alpha)}(\fork[\alpha,\gamma]{p\restriction\alpha}(r))=(x,y).$$

$\br$ Suppose $\alpha\le\mu^{+}$ is a nonzero limit ordinal,
and that the sequence $\langle (\mathbb P_\beta,\lh_\beta,c_\beta,\vec{\varpi}_\beta,  \langle \pitchfork_{\beta,\gamma}\mid\gamma\le\beta\rangle)\mid\beta<\alpha\rangle$ has already been defined according to Goal \ref{goals}.

Define $\mathbb P_\alpha:=(P_\alpha,\le_\alpha)$ by letting $P_\alpha$ be all $\alpha$-sequences $p$
such that $|B_p|<\mu$ and $\forall\beta<\alpha(p\restriction\beta\in P_\beta)$.
Let $p\le_{\alpha} q$ iff  $\forall{\beta<\alpha}(p\restriction \beta\le_{\beta} q\restriction \beta)$. Let $\lh_\alpha:=\lh_1\circ\pi_{\alpha,1}$.
Next, we define $c_\alpha:P_\alpha\rightarrow H_\mu$, as follows.

$\br\br$ If $\alpha<\mu^+$, then, for every $p\in P_\alpha$, let
$$c_\alpha(p):=\{ (\phi_\alpha(\gamma),c_{\gamma}(p\restriction\gamma))\mid \gamma\in B_p\}.$$

$\br\br$ If $\alpha=\mu^+$, then, given $p\in P_\alpha$,
first let $C:=\cl(B_p)$, then define a function $e:C\rightarrow H_\mu$ by stipulating:
$$e(\gamma):=(\phi_\gamma[C\cap\gamma],c_{\gamma}(p\restriction\gamma)).$$
Then, let $c_\alpha(p):=i$ for the least $i<\mu$ such that $e\s e^i$. Set $\vec{\varpi}_\alpha:=\vec{\varpi}_1\bullet \pi_{\alpha,1}$.

Finally, let $p\in P_{\alpha}$, $\gamma\le\alpha$ and $r\LE_\gamma p\restriction\gamma$ be arbitrary; we need to define $\fork[\alpha,\gamma]{p}(r)$.\label{varpilimit}
For $\gamma=\alpha$, let $\fork[\alpha,\gamma]{p}(r):=r$,
and for $\gamma<\alpha$, let
$\fork[\alpha,\gamma]{p}(r):=\bigcup\{\fork[\beta,\gamma]{p\restriction\beta}(r)\mid \gamma\le\beta<\alpha\}$.

\subsection{Verification} Our next task is to verify that for  all $\alpha\le\mu^+$,
the tuple $(\mathbb{P}_\alpha,\lh_\alpha,c_\alpha,\allowbreak\vec{\varpi}_\alpha,\langle\pitchfork_{\alpha,\gamma}\mid \gamma\le\alpha\rangle)$ fulfills requirements (i)--(vi) of Goal~\ref{goals}.
It is obvious that Clauses (i) and (iii) hold, so we
focus on verifying the rest.

 The next fact deals with an expanded version of Clause~(vi). For the proof we refer the reader to \cite[Lemma 3.5]{PartII}:

\begin{fact}\label{CviIteration}
For all $\gamma\le\alpha\leq  \mu^+$, $p\in P_\alpha$ and $r\in P_\gamma$ with $r\LE_\gamma p\restriction\gamma$, if we let $q:=\fork[\alpha,\gamma]{p}(r)$, then:
\begin{enumerate}
\item\label{vi-like}  $q\restriction\beta=\fork[\beta,\gamma]{p\restriction\beta}(r)$ for all $\beta\in[\gamma,\alpha]$;
\item\label{SupportForking} $B_{q}=B_p\cup B_r$;
\item\label{C5LemmaOfForking} $q\restriction\gamma=r$;
\item If $\gamma=0$, then $q=p$;
\item\label{CViforking} $p=(p\restriction\gamma)*\emptyset_\alpha$ iff $q=r*\emptyset_\alpha$;
\item\label{Cviiforking} for all $p'\LE^0_\alpha p$, if $r\LE^0_\gamma p'\restriction\gamma$, then $\fork[\alpha,\gamma]{p'}(r)\LE_\alpha\fork[\alpha,\gamma]{p}(r)$.
\end{enumerate}
\end{fact}

We move on to Clause~(ii) and Clause~(v):

\begin{lemma}\label{CvIteration}
Suppose that $\alpha\le\mu^+$ is such that for all nonzero $\gamma<\alpha$, $(\mathbb{P}_\gamma, c_\gamma, \lh_\gamma,\vec{\varpi}_\gamma)$ is $(\Sigma,\vec{\mathbb{S}})$-Prikry.
Then:
\begin{itemize}
\item for all nonzero $\gamma\le\alpha$, $({\pitchfork_{\alpha,\gamma}},{\pi_{\alpha,\gamma}})$ is a nice forking projection from $(\mathbb{P}_\alpha, \lh_\alpha,\vec{\varpi}_\alpha)$ to $(\mathbb{P}_\gamma,\lh_\gamma,\vec{\varpi}_\gamma)$,
where $\pi_{\alpha,\gamma}$ is defined as in Goal~\ref{goals}(ii);
\item  if $\alpha<\mu^+$, then $({\pitchfork_{\alpha,\gamma}},{\pi_{\alpha,\gamma}})$ is furthermore a nice forking projection from $(\mathbb{P}_\alpha, \lh_\alpha,c_\alpha,\vec{\varpi}_\alpha)$ to $(\mathbb{P}_\gamma, \lh_\gamma,c_\gamma,\vec{\varpi}_\gamma)$
\item if $\alpha=\gamma+1$,  then $(\pitchfork_{\alpha,\gamma},\pi_{\alpha,\gamma})$ is super nice and has the weak mixing property.
\end{itemize}
\end{lemma}
\begin{proof}
The above items with the exception of the niceness requirement can be proved as in \cite[Lemma~3.6]{PartII}. 
It thus suffices to prove the following: 

\begin{claim}\label{PropertiesForking} For all nonzero $\gamma\le\alpha$,
$\vec{\varpi}_\alpha={\vec{\varpi}_\gamma}\mathrel{\bullet}{\pi_{\alpha,\gamma}}$. Also, for each $n$, $\varpi^\alpha_n$ is a nice projection from $(\mathbb{P}_\alpha)_{\geq n}$ to $\mathbb{S}_n$ and for each  $k\geq n$, $\varpi^\alpha_n\restriction(\mathbb{P}_\alpha)_k$ is again a nice projection.
\end{claim}
\begin{proof}  By induction on $\alpha\le\mu^+$:

$\br$ The case $\alpha=1$ is trivial, since then, $\gamma=\alpha$ and $\vec{\varpi}_1=\vec{\varpi}\mathrel{\bullet}\iota$.

$\br$ Suppose $\alpha=\alpha'+1$ and the claim holds for $\alpha'$.
Recall that $\mathbb P_{\alpha}=\mathbb P_{\alpha'+1}$ was defined
by feeding $(\mathbb{P}_{\alpha'},\lh_{\alpha'},c_{\alpha'},\vec{\varpi}_{\alpha'})$  into \blkref{II},
thus obtaining a $(\Sigma,\vec{\mathbb{S}})$-Prikry forcing $(\mathbb A,\lh_{\mathbb A},c_{\mathbb A},\vec{\varsigma})$
along with the pair $(\pitchfork,\pi)$.
Also, we have that $(x,y)\in  A$ iff $x{}^\smallfrown \langle y\rangle\in P_\alpha$.

By niceness of $(\pitchfork,\pi)$ and our recursive definition, $$\varpi^\alpha_n(x{}^\smallfrown \langle y\rangle)=\varsigma_n(x,y)=\varpi^{\alpha'}_n(\pi(x,y))=\varpi^{\alpha'}_n(\pi_{\alpha,\alpha'}(x{}^\smallfrown \langle y\rangle)), $$
for all $n<\omega$ and $x{}^\smallfrown \langle y\rangle\in (P_\alpha)_{\geq n}$. Hence, $\vec{\varpi}_\alpha=\vec{\varpi}_{\alpha'}\bullet\pi_{\alpha,\alpha'}$. Using the induction hypothesis for $\vec{\varpi}_{\alpha'}$ we arrive at $\vec{\varpi}_\alpha=\vec{\varpi}_\gamma\bullet\pi_{\alpha,\gamma}$.

Let us  now address  the second part of the claim. We just show that for every $n<\omega$, the map $\varpi^\alpha_n$ is a nice projection from $(\mathbb{P}_\alpha)_{\geq n}$ to $\mathbb{S}_n$. The statement that $\varpi^\alpha_n\restriction(\mathbb{P}_\alpha)_k$ is a nice projection can be proved similarly.

So, let us go over the clauses of Definition~\ref{niceprojection}. Clauses~\eqref{niceprojection1} and \eqref{niceprojection2} are evident and Clause~\eqref{niceprojection3} follows from Lemma~\ref{pitchforkexact} applied to $(\pitchfork_{\alpha,\alpha'},\pi_{\alpha,\alpha'})$.

\smallskip

\eqref{theprojection}: Let $p,q\in (P_\alpha)_{\geq n}$ and $s\sle_n \varpi^\alpha_n(p)$ be such that $q\LE_\alpha p+s$. Then, $(q\restriction \alpha',q(\alpha'))\unlhd (p\restriction \alpha',p(\alpha'))+s$. By Clause~\ref{C1blk2} of \blkref{II} we have that $\varsigma_n$ is a nice projection from $\mathbb{A}_{\geq n}$ to $\mathbb{S}_n$, hence there is $(x,y)\in A$ such that $(x,y)\unlhd^{\varsigma_n} (p\restriction \alpha',p(\alpha'))$ and $(q\restriction \alpha',q(\alpha'))=(x,y)+\varsigma_n((q\restriction \alpha',q(\alpha')))$. Setting $p':=x^\smallfrown\langle y\rangle$ it is immediate that $p'\LE_\alpha^{\varpi^\alpha_n}p$ and $$q=p'+\varpi^{\alpha'}_n(q\restriction\alpha')=p'+\varpi^\alpha_n(q).$$

$\br$
For $\alpha\in\acc(\mu^++1)$, the first part follows from $\vec{\varpi}_\alpha:=\vec{\varpi}_1\circ \pi_{\alpha,1}$ and the induction hypothesis. About the verification of the Clauses of Definition~\ref{niceprojection}, Clauses~\eqref{niceprojection1} and \eqref{niceprojection2} are automatic and  Clause~\eqref{niceprojection3} follows from Lemma~\ref{pitchforkexact} applied to $(\pitchfork_{\alpha,1},\pi_{\alpha,1})$. About Clause~\eqref{theprojection} we argue as follows.

Fix $p,q\in (\mathbb{P}_\alpha)_{\geq n}$ and $s\sle_n \varpi^\alpha_n(p)$ be such that $q\LE_\alpha p+s$. The goal is to find a condition $p'\in (P_\alpha)_{\geq n}$ such that  $p'\LE^{\varpi^\alpha_n} p$  and $q=p'+\varpi^\alpha_n(q)$.

Let $\langle\gamma_\tau\mid \tau\leq\theta\rangle$ be the increasing enumeration of the closure of $B_{q}$.\footnote{Recall that $B_q:=\{ \beta+1\mid \beta\in\dom(q)\ \&\ q(\beta)\neq\emptyset\}$.} For every $\tau\in\nacc(\theta+1)$, $\gamma_\tau$ is a successor ordinal, so we let $\beta_\tau$ denote its predecessor. By recursion on $\tau\leq \theta$, we shall define a sequence of conditions $\langle p'_\tau\mid \tau\leq \theta\rangle\in\prod_{\tau\leq \theta} (P_{\gamma_\tau})$ such that $p'_\tau\LE^{\varpi^{\gamma_\tau}_n}_{\gamma_\tau} p\restriction\gamma_\tau$ and $q\restriction\gamma_\tau=p'_\tau+ \varpi^{\gamma_\tau}_n(q\restriction\gamma_\tau)$.

In order to be able to continue with the construction at limits stages we shall moreover secure that $\langle p'_\tau\mid \tau\leq \theta\rangle$ is coherent: i.e., $p'_\tau\restriction\gamma_{\tau'}=p_{\tau'}$ for all $\tau'\leq \tau$. Also, note that $\langle \varpi^{\gamma_\tau}_n(q\restriction\gamma_\tau)\mid \tau\leq \theta\rangle$ is a constant sequence, so hereafter we denote by $t$ its constant value.

To form the first member of the sequence we argue as follows. First note that $q\restriction 1\LE_1 p\restriction 1+s$, so that appealing to Definition~\ref{niceprojection}\eqref{theprojection} for $\varpi^1_n$ we get a condition $p'_{-1}\in P_1$ such that $p'_{-1}\LE^{\varpi^1_n}_1 p\restriction 1$ and $q\restriction 1=p'_{-1}+t$.

Now, let $r_0:=\fork[\gamma_0,1]{p\restriction\gamma_0}(p'_{-1})$. A moment's reflection makes it clear that $r_0+s$ is well-defined and also  $q\restriction\gamma_0\LE_{\gamma_0} r_0+s$. So, appealing to Definition~\ref{niceprojection}\eqref{theprojection} for $\varpi^{\gamma_0}_n$ we get a condition $p'_0\in P_{\gamma_0}$ such that $p'_0\LE^{\varpi^{\gamma_0}_n}_{\gamma_0} r_0$ and $q\restriction \gamma_0=p'_0+t$. Since $p'_{-1}\LE^{\varpi^1_n}_1 p\restriction 1$ and $\varpi^{\gamma_0}_n=\varpi^1_n\circ\pi_{\gamma_0,1}$ we have $\varpi^{\gamma_0}_n (r_0)=\varpi^{\gamma_0}_n(p\restriction\gamma_0)$. This completes the first step of the induction.

\smallskip

Let us suppose that we have already constructed $\langle p_{\tau'}\mid \tau'<\tau\rangle$.

\underline{$\tau$ is successor:} Suppose that $\tau=\tau'+1$.  Then, set $r_\tau:=\fork[\gamma_\tau,\gamma_{\tau'}]{p\restriction\gamma_\tau}(p'_{\tau'})$. Using the induction hypothesis it is easy to see that $q\restriction\gamma_\tau\LE_{\gamma_\tau} r_\tau+s$. Instead of outright invoking the niceness of $\varpi^{\gamma_\tau}_n$ we would like to use that $(\pitchfork_{\gamma_\tau,\beta_\tau}, \pi_{\gamma_\tau,\beta_\tau})$ is a super nice forking projection (see Definition~\ref{SuperNiceForking}). This will secure that the future condition $p'_\tau$ will be coherent with $p'_{\tau'}$, and therefore with all the conditions constructed so far.

Applying the definition of  $\pitchfork_{\gamma_\tau,\gamma_{\tau'}}$ given at page~\pageref{pitchforkatsuccessors}  we have $$r_\tau=\fork[\gamma_\tau,\beta_\tau]{p\restriction\gamma_\tau}(\fork[\beta_\tau,\gamma_{\tau'}]{p\restriction\beta_\tau}(p'_{\tau'})).$$
Since $p\restriction\beta_\tau=p\restriction\gamma_{\tau'}\ast \emptyset_{\beta_\tau}$,  Clause~\eqref{Cviiforking} of Fact~\ref{CviIteration} yields $$\fork[\beta_\tau,\gamma_{\tau'}]{p\restriction\beta_\tau}(p'_{\tau'})=p'_{\tau'}\ast \emptyset_{\beta_\tau}.$$
So, $r_\tau=\fork[\gamma_\tau,\beta_\tau]{p\restriction\gamma_\tau}(p'_{\tau'}\ast \emptyset_{\beta_\tau})$.
\begin{subclaim}
 $p'_{\tau'}\ast \emptyset_{\beta_\tau}\LE^{\varpi^{\beta_\tau}_n}_{\beta_\tau} p\restriction\beta_\tau$ and $q\restriction\beta_\tau=(p'_{\tau'}\ast \emptyset_{\beta_\tau})+t$.
\end{subclaim}
\begin{proof}
The first part follows immediately from $p'_{\tau'}\LE^{\varpi^{\gamma_\tau'}_n}_{\gamma_{\tau'}} p\restriction\gamma_{\tau'}$. For the second part note that $q\restriction\beta_\tau=q\restriction\gamma_{\tau'}\ast \emptyset_{\beta_\tau}$, hence Fact~\ref{CvIteration}\eqref{CViforking} combined with the induction hypothesis yield
$$q\restriction\beta_\tau=\fork[\beta_\tau,\gamma_{\tau'}]{q\restriction\beta_\tau}(q\restriction\gamma_{\tau'})=\fork[\beta_\tau,\gamma_{\tau'}]{q\restriction\beta_\tau}(p'_{\tau'}+t)=(p'_{\tau'}+t)\ast \emptyset_{\beta_\tau}.$$
On the other hand, using Lemma~\ref{pitchforkexact} with respect to $(\pitchfork_{\beta_\tau,\gamma_{\tau'}},\pi_{\beta_{\tau},\gamma_{\tau'}})$,
$$(p'_{\tau'}\ast \emptyset_{\beta_\tau})+t=\fork[\beta_\tau,\gamma_{\tau'}]{p'_{\tau'}\ast \emptyset_{\beta_\tau}}(p'_{\tau'}+t)=(p'_{\tau'}+t)\ast \emptyset_{\beta_\tau},$$
where the last equality follows again from Fact~\ref{CvIteration}\eqref{CViforking}.

Combining the above expressions we arrive at $q\restriction\beta_\tau=(p'_{\tau'}\ast \emptyset_{\beta_\tau})+t$.
\end{proof}
By Clause~\ref{C5blk2} of \blkref{II} and the definition of the iteration at successor stage (see page~\pageref{pitchforkatsuccessors}), the pair $(\pitchfork_{\gamma_\tau,\beta_\tau},\pi_{\gamma_\tau,\beta_\tau})$ is a super nice forking projection from $(\mathbb{P}_{\gamma_\tau},\ell_{\gamma_\tau},\vec{\varpi}_{\gamma_\tau})$ to $(\mathbb{P}_{\beta_\tau},\ell_{\beta_\tau},\vec{\varpi}_{\beta_\tau})$. Combining this with the above  subclaim we find a condition $p'_\tau\LE^{\varpi^{\gamma_\tau}_n}_{\gamma_\tau} r_\tau$ such that $p'_\tau\restriction\beta_\tau=p'_{\tau'}\ast \emptyset_{\beta_\tau}$ and  $q\restriction\gamma_\tau=p'_\tau+t$. Clearly, $p'_\tau$ witnesses the desired property. \label{supernicenessinaction}

\underline{$\tau$ is limit:} Put $p'_\tau:=\bigcup_{\tau'<\tau} p'_{\tau'}$. Thanks to the induction hypothesis it is evident that $p'_\tau\LE^{\varpi^{\gamma_\tau}_n}_{\gamma_\tau} p\restriction\gamma_\tau$. Also,  combining the induction hypothesis with Lemma~\ref{pitchforkexact} for $(\pitchfork_{\gamma_\tau,1},\pi_{\gamma_\tau,1})$ we obtain the following chain of equalities:
$$q\restriction\gamma_\tau=\bigcup_{\tau'<\tau}(p'_{\tau'}+t)=\bigcup_{\tau'<\tau}\fork[\gamma_\tau,1]{p'_{\tau'}}(p'_{\tau'}\restriction 1+t)=\bigcup_{\tau'<\tau}\fork[\gamma_{\tau'},1]{p'_{\tau'}}(p'_{\tau}\restriction 1+t).\footnote{Note that for the right-most equality we have used that $p'_\tau\restriction 1=p'_{\tau'}\restriction 1$, for all $\tau'<\tau.$}$$
Using the definition of the pitchfork at limit stages (see page~\pageref{varpilimit}) we get
$$\bigcup_{\tau'<\tau}\fork[\gamma_{\tau'},1]{p'_{\tau'}}(p'_{\tau}\restriction 1+t)=p'_\tau+t=\fork[\gamma_\tau,1]{p'_\tau}(p'_\tau\restriction 1+t),$$
where the last equality follows from Lemma~\ref{pitchforkexact} for $(\pitchfork_{\gamma_\tau,1},\pi_{\gamma_\tau,1})$.

Altogether, we have shown that $p'_\tau\LE^{\varpi^{\gamma_\tau}_n}_{\gamma_\tau} p\restriction\gamma_\tau$ and $q\restriction\gamma_\tau=p'_\tau+t$.  Additionally, $p'_\tau\restriction\gamma_{\tau'}=p'_{\tau'}$ for all $\tau'<\tau$.

\smallskip

Finally, putting  $p':=p'_\theta$  we obtain a condition in  $(\mathbb{P}_{\alpha})_{\geq n}$ such that $$p'\LE^{\varpi^\alpha_n}_{\alpha} p\;\;\text{and}\;\; q=p'+\varpi^\alpha_n(q).$$ This completes the argument.\qedhere
\end{proof}
This completes the proof of the lemma.
\end{proof}

Recalling Definition~\ref{SigmaPrikry}\eqref{c2},
for all nonzero $\alpha\le\mu^+$ and $n<\omega$, we need to identify a candidate for a dense subposet $(\z{\mathbb P}_{\alpha})_n=((\z{P}_{\alpha})_n,\le_\alpha)$ of $(\mathbb P_{\alpha})_n$.

\begin{definition}\label{particularorderings}  For each nonzero $\gamma<\mu^+$,
we let $\tp_{\gamma+1}$ be a  type witnessing that $(\pitchfork_{\gamma+1,\gamma},\pi_{\gamma+1,\gamma})$ has the weak mixing property.
\end{definition}

\begin{definition}\label{DefRingForLimits}
Let $n<\omega$. Set $\z{P}_{1n}:={}^{1}{{(\z{Q}_n})}$.\footnote{Here, $\z{Q}_n$ is obtained from Clause~\eqref{c2} of Definition~\ref{SigmaPrikry} with respect to the triple $(\mathbb{Q},\ell,c)$ given by \blkref{I}.}
Then, for each   $\alpha\in[2,\mu^+]$, define $\z P_{\alpha n}$ by recursion:
$$\z{P}_{\alpha n}:=\begin{cases}
\{p\in P_\alpha\mid \pi_{\alpha,\beta}(p)\in \z{P}_{\beta n}\ \&\ \mtp_{\beta+1}(p)=0\},&\text{if }\alpha=\beta+1;\\
\{p\in P_\alpha\mid \pi_{\alpha,1}(p)\in\z{P}_{1n}\ \&\ \forall\gamma\in B_p\,\mtp_\gamma(\pi_{\alpha,\gamma}(p))=0\},&\text{otherwise}.\\
\end{cases}
$$
\end{definition}

\begin{lemma}\label{ringcoheres}\label{LiftingAndRings} Let $n<\omega$ and $1\leq \beta<\alpha\leq \mu^+$. Then:
\begin{enumerate}
\item $\pi_{\alpha,\beta}``\z{P}_{\alpha n}\s \z{P}_{\beta n}$;
\item For every $p\in \z{P}_{\beta n}$, $p* \emptyset_{\delta}\in \z{P}_{\alpha n}$.
\end{enumerate}
\end{lemma}
\begin{proof} By induction, relying on Clause~\eqref{type6} of Definition~\ref{type}.
\end{proof}

We now move to establish Clause~(iv) of Goal~\ref{goals}. 

\begin{lemma}\label{MoreClosednessOfIterates} Let $\alpha\in[2,\mu^+]$. Then, for all $n<\omega$ and every directed set of conditions $D$ in $(\z{\mathbb{P}}_\alpha)_n$ (resp. $(\z{\mathbb{P}}_\alpha)^{\varpi^\alpha_n}_n$) of size ${<}\aleph_1$ (resp. ${<}\sigma^*_n$) there is $q\in(\z{P}_\alpha)_n$ such that $q$ is a $\LE_\alpha$ (resp. $\LE^{\varpi^\alpha_n}_n$) lower bound for $D$.

Moreover, $B_q=\bigcup_{p\in D} B_p$.
 \end{lemma}
 \begin{proof}
 The argument is similar to that of \cite[Lemma~3.13]{PartII}.
 \end{proof}

\begin{remark}\label{piclosure}
A straightforward modification of the lemma shows that for all $\alpha\in[2,\mu^+]$ and $n<\omega$, $(\mathbb{P}_\alpha)^{\pi_{\alpha,1}}_n$ is $\sigma^*_n$-directed-closed.
\end{remark}

\begin{theorem}\label{CivIteration}
For all nonzero $\alpha\le\mu^+$, $(\mathbb P_\alpha,\lh_\alpha,c_\alpha,\vec{\varpi}_\alpha)$ satisfies all the requirements to be a $(\Sigma,\vec{\mathbb{S}})$-Prikry quadruple, with the possible exceptions of Clause~\eqref{c6} and the density requirement in Clauses~\eqref{c2} and \eqref{moreclosedness}.

Additionally, $\emptyset_\alpha$ is the greatest condition in $\mathbb{P}_\alpha$, $\lh_\alpha=\lh_1\circ\pi_{\alpha,1}$, $\emptyset_\alpha\forces_{\mathbb{P}_\alpha} \check{\mu}=\kappa^+$ and $\vec{\varpi}_\alpha$ is a coherent sequence of nice projections such that $$\vec{\varpi}_\alpha=\vec{\varpi}_\gamma\bullet\pi_{\alpha,\gamma}\;\;\text{for every}\;\gamma\leq \alpha.$$

Under the extra hypothesis that for each $\alpha\in\acc(\mu^++1)$ and every $n<\omega$, $(\z{\mathbb{P}}^{\varpi^\alpha_n}_\alpha)_n$ is a dense subposet of $(\mathbb{P}^{\varpi^\alpha_n}_\alpha)_n$, we have that for all nonzero $\alpha\leq \mu^+$, $(\mathbb P_\alpha,\lh_\alpha,c_\alpha,\vec{\varpi}_\alpha)$ is a $(\Sigma,\vec{\mathbb{S}})$-Prikry quadruple having property $\mathcal{D}$.
\end{theorem}
\begin{proof}
We argue by induction on $\alpha\le\mu^+$. The base case $\alpha=1$ follows from the fact that $\mathbb P_1$ is isomorphic to $\mathbb Q$ given by \blkref{I}.
The successor step $\alpha=\beta+1$ follows from the fact that $\mathbb P_{\beta+1}$ was obtained by invoking \blkref{II}.

Next, suppose that $\alpha\in\acc(\mu^++1)$ is such that the conclusion of the lemma holds below $\alpha$.
In particular, the hypothesis of Lemma~\ref{CvIteration} are satisfied,
so that, for all nonzero $\beta\le\gamma\le\alpha$,
$(\pitchfork_{\gamma,\beta},\pi_{\gamma,\beta})$ is a nice forking projection from $(\mathbb{P}_\gamma, \lh_\gamma,\vec{\varpi}_\gamma)$ to $(\mathbb{P}_\beta,\lh_\beta,\vec{\varpi}_\beta)$. 
By the very same proof of \cite[Lemma~3.14]{PartII}, we have that Clauses \eqref{graded} and \eqref{c1}--\eqref{itsaprojection} of Definition~\ref{SigmaPrikry} hold for
$(\mathbb P_\alpha,\lh_\alpha,c_\alpha,\vec{\varpi}_\alpha)$, and that  $\lh_\alpha=\lh_1\circ\pi_{\alpha,1}$.   Also, Clauses \eqref{c2} and \eqref{moreclosedness} --without the density requirement-- follow from Lemma~\ref{MoreClosednessOfIterates}.

On the other hand, the equality  $\vec{\varpi}_\alpha=\vec{\varpi}_\gamma\bullet\pi_{\alpha,\gamma}$ follows from Lemma~\ref{CvIteration}. Arguing as in \cite[Claim~3.14.2]{PartII}, we also have that
$\one_{\mathbb P_{\alpha}}\forces_{\mathbb{P}_{\alpha}}\check{\mu}=\check\kappa^+$. Finally, since $\vec{\varpi}_1$ is coherent (see \blkref{I}) and $(\pitchfork_{\alpha,1},\pi_{\alpha,1})$ is a nice forking projection, Lemma~\ref{liftingcoherency} implies that $\vec{\varpi}_\alpha$ is coherent.

To complete the proof let us additionally assume that for every $n$, $(\z{\mathbb{P}}^{\varpi^\alpha_n}_\alpha)_n$ is a dense subposet of $(\mathbb{P}^{\varpi^\alpha_n}_\alpha)_n$. Then, in particular,  $(\z{\mathbb{P}}_\alpha)_n$ is a dense subposet of $(\mathbb{P}_\alpha)_n$. In effect,  the density requirement in Clauses~\eqref{c2} and \eqref{moreclosedness} is automatically fulfilled. About Clause~\eqref{c6}, we take advantage of this extra assumption to invoke
\cite[Corollary~3.12]{PartII} and conclude that $(\mathbb{P}_\alpha,\ell_\alpha)$ has property $\mathcal{D}$. Consequently, Clause~\eqref{c6} for $(\mathbb{P}_\alpha,\ell_\alpha,c_\alpha,\vec{\varpi}_\alpha)$ follows by combining this latter fact with Lemma~\ref{lemmaforCPP}.
\end{proof}

\section{A proof of the Main Theorem}\label{ReflectionAfterIteration}
In this section, we arrive at the primary application of the framework developed thus far.
We will be constructing a model where $\gch$ holds below $\aleph_{\omega}$, $2^{\aleph_\omega}=\aleph_{\omega+2}$ and every stationary subset of $\aleph_{\omega+1}$ reflects.

\subsection{Setting up the ground model}
We want to obtain a ground model with $\gch$ and $\omega$-many supercompact cardinals, which are Laver indestructible under $\gch$-preserving forcing.
The first lemma must be well-known, but we could not find it in the literature,
so we give an outline of the proof.
\begin{lemma}\label{addingGCH}
Suppose $\vec\kappa=\langle\kappa_n\mid n<\omega\rangle$ is an increasing sequence of supercompact cardinals.
Then there is a generic extension where $\gch$ holds and $\vec\kappa$ remains an increasing sequence of supercompact cardinals.
\end{lemma}
\begin{proof}
Let $\mathbb{J}$ be Jensen's iteration to force the \gch. Namely, $\mathbb{J}$ is the inverse limit of the Easton-support iteration $\langle \mathbb{J}_\alpha;\dot{\mathbb{Q}}_\beta\mid \beta\leq \alpha\in\ord\rangle$ such that, if $\one_{\mathbb{J}_\alpha}\forces_{\mathbb{J}_\alpha}``\alpha$ is a cardinal'', then $\one_{\mathbb{J}_\alpha}\forces_{\mathbb{J}_\alpha}``\dot{\mathbb{Q}}_\alpha=\dot{\mathrm{Add}}(\alpha^+,1)$'' and $\one_{\mathbb{J}_\alpha}\forces_{\mathbb{J}_\alpha}``\dot{\mathbb{Q}}_\alpha$ trivial'', otherwise. Let $G$ be a $\mathbb{J}$-generic filter over $V$.

Let $n<\omega$. We claim that $\mathbb{J}$ preserves the supercompactness of $\kappa_n$.  To this end, let $\theta$ be an arbitrary cardinal.
By Solovay's theorem concerning cardinal arithmetic above strongly compact cardinals, we may enlarge $\theta$ to ensure that $\theta^{<\theta}=\theta$.
Let $j:V\rightarrow M$ be an elementary embedding induced by a $\theta$-supercompact measure over $\mathcal{P}_{\kappa_n}(\theta)$. In particular, we are taking $j$ such that $\crit(j)=\kappa_n$, $j(\kappa_n)>\theta$,  $\left({}^\theta M\right)\cap V \s M$ and
$$M=\{j(f)(j``\theta)\mid f\colon \mathcal{P}_{\kappa_n}(\theta)\rightarrow V\}.$$
Observe that $\mathbb{J}$ can be factored into three forcings: the iteration up to $\kappa_n$, the iteration in the interval $[\kappa_n,\theta)$ and, finally, the iteration in the interval $[\theta,\ord)$.
For an interval of ordinals $\mathcal{I}$, let $G_{\mathcal{I}}$ denote the $\mathbb{J}_{\mathcal{I}}$-generic filter induced by $G$. Similarly, we define $G^*_{\mathcal{I}}:=G_{\mathcal{I}}\cap j(\mathbb{J})_\mathcal{I}$.

\begin{claim}
In $V[G]$, there is a lifting  $j_1:V[G_{\kappa_n}]\rightarrow M[G^*_{j(\kappa_n)}]$ of $j$ such that $$\left({}^\theta M[G^*_{j(\kappa_n)}]\right)\cap V[G^*_{j(\kappa_n)}]\s M[G^*_{j(\kappa_n)}].$$
Moreover, $H_{\theta^+}^{M[G^*_{j(\kappa_n)}]}\s V[G_\theta]$.\qed
\end{claim}

\begin{claim} In $V[G]$, there is a lifting $j_2:V[G_{\theta}]\rightarrow M[G^*_{j(\theta)}]$ of $j_1$. \qed
\end{claim}
\begin{claim}
In $V[G]$, there is a lifting $j_3:V[G]\rightarrow M[G^*_{j(\theta)}\ast \dot{K}]$ of $j_2$.\qed
\end{claim}
Finally, define
$$\mathcal{U}:=\{X\in\mathcal{P}^{V[G]}_{\kappa_n}(\theta)\mid j``\theta\in j_3(X)\}.$$
As $j``\theta\in M\s M[G^*_{j(\theta)}\ast \dot{K}]$,
standard arguments show that $\mathcal{U}$ is a $\theta$-supercompact measure over  $\mathcal{P}^{V[G]}_{\kappa_n}(\theta)$.
Hence, $\kappa_n$ is $\theta$-supercompact in $V[G]$.
\end{proof}
Note that in the model of the conclusion of the above lemma, the $\kappa_n$'s are highly-destructible by $\kappa_n$-directed-closed forcing. Our next task is to remedy that, while maintaining $\gch$.
For this, we need the following slight variation of the usual Laver preparation \cite{Lav}.

\begin{lemma}\label{NewPreparation} Suppose that $\gch$ holds, $\chi<\kappa$ are infinite regular cardinals, and $\kappa$ is supercompact.
Then there exists a $\chi$-directed-closed notion of forcing $\mathbb L_{\chi}^\kappa$
that preserves $\gch$ and makes the supercompactness of $\kappa$ indestructible under $\kappa$-directed-closed forcings that preserve $\gch$.
\end{lemma}
\begin{proof}
Let $f$ be a Laver function on $\kappa$, as in \cite[Theorem~24.1]{MR2768691}.
Let  $\mathbb{L}_\chi^\kappa$ be the direct limit of the forcing iteration $\langle\mathbb{R}_\alpha;\dot{\mathbb{Q}}_\beta\mid \chi\leq \beta< \alpha<\kappa\rangle$  where, if $\alpha$ is inaccessible, $\one_{\mathbb{R}_\alpha}\forces_{\mathbb{R}_\alpha} \gch$, and $f(\alpha)$ encodes an $\mathbb{R}_\alpha$-name $\tau\in H_{\alpha^+}$ for some $\alpha$-directed-closed forcing that preserves the $\gch$ of $V^{\mathbb{R}_\alpha}$, 
then  $\dot{\mathbb{Q}}_\alpha$ is chosen to be such $\mathbb{R}_\alpha$-name. Otherwise, $\dot{\mathbb{Q}}_\alpha$ is chosen to be the trivial forcing.

As in the proof of \cite[Theorem 24.12]{MR2768691}, we have that
after forcing with $\mathbb{L}_\chi^\kappa$,
the supercompactness of $\kappa$ becomes indestructible under $\kappa$-directed-closed forcings that preserve $\gch$.

We claim that $\gch$ holds in $V^{\mathbb{L}_\chi^\kappa}$. This is clear for cardinals $\geq\kappa$, since the iteration has size $\kappa$. Now, let $\lambda<\kappa$ and inductively assume $\gch_{<\lambda}$. Observe that $\mathbb{L}_\chi^\kappa\cong \mathbb{R}_{\lambda+1}\ast \dot{\mathbb{Q}}$,
where $\dot{\mathbb{Q}}$ is an $\mathbb{R}_{\lambda+1}$-name for a $\lambda^+$-directed-closed forcing.
In particular, $\mathcal{P}(\lambda)^{V^{\mathbb{L}_\chi^\kappa}}= \mathcal{P}(\lambda)^{V^{\mathbb{R}_{\lambda+1}}}$, and so it is enough to show that $V^{\mathbb{R}_{\lambda+1}}\models \ch_{\lambda}$.
There are two cases.

If $\lambda$ is singular, then $|\mathbb{R}_\lambda|=\lambda^+$, and $\mathbb{\dot{Q}}_\lambda$ is trivial, so $V^{\mathbb{R}_{\lambda+1}}\models \ch_{\lambda}$.

Otherwise, let $\alpha$ be the largest inaccessible, such that $\alpha\leq\lambda$. Then $\mathbb{R}_{\lambda+1}$ is just $\mathbb{R}_\alpha\ast \dot{\mathbb{Q}}_\alpha$ followed by trivial forcing. Since $|\mathbb{R}_\alpha|=\alpha$ and by construction $\dot{\mathbb{Q}}_\alpha$ preserves $\gch$, the result follows.
\end{proof}

\begin{cor}\label{cor84} Suppose that $\vec\kappa=\langle\kappa_n\mid n<\omega\rangle$ is an increasing sequence of supercompact cardinals.
Then, in some forcing extension, all of the following hold:
\begin{enumerate}
\item $\gch$;
\item $\vec\kappa$ is an increasing sequence of supercompact cardinals;
\item For every $n<\omega$, the supercompactness of $\kappa_n$ is indestructible under notions of forcing that
are $\kappa_n$-directed-closed and preserves the $\gch$.
\end{enumerate}
\end{cor}
\begin{proof} By Lemma~\ref{addingGCH}, we may assume that we are working in a model in which Clauses (1) and (2) already hold.
Next, let $\mathbb{L}$ be the direct limit of the iteration $\langle \mathbb{L}_n*\dot{\mathbb{Q}}_n\mid n<\omega\rangle$, where $\mathbb{L}_0$ is the trivial forcing
and, for each $n$, if $\one\forces_{\mathbb{L}_{n}}``\kappa_{n}\text{ is supercompact}"$,
then $\one\forces_{\mathbb{L}_n}\mathbb{\dot{Q}}_n$ is the $(\kappa_{n-1})$-directed-closed,
 $\gch$-preserving forcing making the supercompactness of $\kappa_{n}$ indestructible under $\gch$-preserving $\kappa_{n}$-directed-closed notions of forcing.
(More precisely, in the notation of the previous lemma, $\dot{\mathbb{Q}}_n$ is  $\dot{\mathbb{L}}_{\kappa_{n-1}}^{\kappa_{n}}$, where, by convention, $\kappa_{-1}$ is $\aleph_0$).

Note that, by induction on $n<\omega$, and Lemma~\ref{NewPreparation}, we maintain that $$\one\Vdash_{\mathbb{L}_n}``\kappa_n\text{ is supercompact and \gch\ holds}".$$
And then when we force with $\dot{\mathbb{Q}}_n$ over that model, we make this supercompacness indestructible under $\gch$ -preserving forcing.

Then, after forcing with $\mathbb{L}$,  $\gch$ holds, and each $\kappa_n$ remains supercompact,  indestructible under  $\kappa_n$-directed-closed
forcings that preserve $\gch$.
\end{proof}

\subsection{Connecting the dots}

\begin{setup}\label{setupapplication}
For the rest of this section, we make the following assumptions:
\begin{itemize}
\item $\vec{\kappa}=\langle \kappa_n\mid n<\omega\rangle$ is an increasing sequence of supercompact cardinals.
By convention, we set $\kappa_{-1}:=\aleph_0$;
\item For every $n<\omega$, the supercompactness of $\kappa_n$ is indestructible under notions of forcing that
are $\kappa_n$-directed-closed and preserve the $\gch$;
\item $\kappa:=\sup_{n<\omega}\kappa_n$, $\mu:=\kappa^+$ and $\lambda:=\kappa^{++}$;
\item $\gch$ holds below $\lambda$. In particular, $2^\kappa=\kappa^+$ and $2^\mu=\mu^+$;
\item $\Sigma:=\langle\sigma_n\mid n<\omega\rangle$, where $\sigma_0:=\aleph_1$ and $\sigma_{n+1}:=(\kappa_n)^+$ for all $n<\omega$;\footnote{By convention, we set $\sigma_{-2}$ and $\sigma_{-1}$ to be $\aleph_0$.}
\item $\vec{\mathbb{S}}$ is as in Definition~\ref{DefinitionofSnMotis}.
\end{itemize}
\end{setup}

We now want to appeal to the iteration scheme of the previous section. First, observe that $\mu$, $\langle (\sigma_n, \mu)\mid n<\omega\rangle$, $\vec{\mathbb{S}}$ and $\Sigma$
respectively fulfill all the blanket assumptions of Setup~\ref{setupiteration}.

\smallskip

We now introduce our three building blocks of choice:

\blk{I} We let $(\mathbb{Q},\ell,c,\vec{\varpi})$ be EBPFC as defined in Section~\ref{SectionGitikForcing}.
By Corollary~\ref{TheEBPFCisweakly}, this  is a $(\Sigma,\vec{\mathbb{S}})$-Prikry that has property $\mathcal{D}$, and $\vec{\varpi}$ is a  coherent sequence of nice projection.
Also, $\mathbb Q$ is a subset of $H_{\mu^+}$ and, by Lemma~\ref{MuInEBPFC},  $\one_{\mathbb Q}\forces_{\mathbb Q}\check\mu=\check\kappa^+$.
In addition, $\kappa$ is singular, so that we have $\one_{\mathbb Q}\forces_{\mathbb{Q}}``\kappa\text{ is singular}"$. Finally, for all $n<\omega$, $\z{\mathbb{Q}}_n={\mathbb{Q}}_n$ (see~Lemmas~\ref{C2forMoti} and \ref{C4forMoti}).

\blk{II} Suppose that $(\mathbb P,\lh,c,\vec{\varpi})$ is a $(\Sigma,\vec{\mathbb{S}})$-Prikry quadruple having property $\mathcal{D}$  such that  $\mathbb P=\left(P,\le\right)$ is a subset of $H_{\mu^+}$,  $\vec{\varpi}$ is a coherent sequence of nice projections,
$\one_{\mathbb P}\forces_{\mathbb P}\check\mu=\check\kappa^+$
and $\one_{\mathbb P}\forces_{\mathbb{P}}``\kappa\text{ is singular}"$.

For every $r^\star\in P$ and a $\mathbb P$-name $z$ for an $r^\star$-fragile stationary subset of $\mu$,
there are a $(\Sigma,\vec{\mathbb{S}})$-Prikry quadruple $(\mathbb A,\lh_{\mathbb A},c_{\mathbb A},\vec{\varsigma})$ having property $\mathcal{D}$, and a pair of maps  $(\pitchfork,\pi)$   such that all the following hold:
\begin{enumerate}[label=(\alph*)]
\item$(\pitchfork,\pi)$ is  a super nice  forking projection from  $(\mathbb A,\lh_{\mathbb A},c_{\mathbb A},\vec{\varsigma})$ to $(\mathbb P,\lh,c,\vec{\varpi})$ that has the weak mixing property;
\item $\vec{\varsigma}$ is a coherent sequence of nice projections;
\item $\one_{\mathbb A}\forces_{\mathbb A}\check\mu=\check\kappa^+$;
\item $\mathbb A=(A,{\unlhd})$ is a subset of $H_{\mu^+}$;
\item For every $n<\omega$, $\z{\mathbb{A}}^\pi_n$ is $\mu$-directed-closed;
\item \label{C8onestep} if $r^\star\in P$ and  $z$ is a $\mathbb P$-name for an $r^\star$-fragile stationary subset of $\mu$
then $$\myceil{r^\star}{\mathbb A}\forces_{\mathbb A}``z\text{ is nonstationary}".$$
\end{enumerate}

\begin{remark}
\hfill
\begin{itemize}
\item[$\br$] If $r^\star\in P$ forces that $z$ is a $\mathbb P$-name for an $r^\star$-fragile subset of $\mu$,  we first find some $\mathbb{P}$-name $\tilde{z}$ such that $\one_\mathbb{P}$ forces that $\tilde{z}$ is a stationary subset of $\mu$, $r^\star\forces_\mathbb{P}z=\tilde{z}$ and $\tilde{z}$ is $\one_\mathbb{P}$-fragile.
Subsequently, we obtain $(\mathbb A,\lh_{\mathbb A},c_{\mathbb A},\vec{\varsigma})$ and $(\pitchfork,\pi)$ by appealing to Corollary~\ref{onestep} with
the $(\Sigma,\vec{\mathbb{S}})$-Prikry triple $(\mathbb P,\lh_{\mathbb P},c_{\mathbb P},\vec{\varpi})$,
the condition $\one_{\mathbb P}$
and the $\mathbb{P}$-name $\tilde{z}$.
In effect, $\myceil{\one_{\mathbb P}}{\mathbb A}$ forces that $\tilde{z}$ is nonstationary,
so that $\myceil{r^\star}{\mathbb A}$ forces that $z$ is nonstationary.

\item[$\br$] Otherwise,
we invoke Corollary~\ref{onestep}
with the $(\Sigma,\vec{\mathbb{S}})$-Prikry forcing $(\mathbb P,\lh_{\mathbb P},c_{\mathbb P},\vec{\varpi})$,
the condition $\one_{\mathbb P}$ and the name $z:=\emptyset$.
\end{itemize}
\end{remark}

\blk{III} As $2^\mu=\mu^+$, we fix a surjection $\psi:\mu^+\rightarrow H_{\mu^+}$ such that the preimage of any singleton is cofinal in $\mu^+$.

\medskip

The next lemma deals with the extra assumption in Theorem~\ref{CivIteration}:

\begin{lemma}[Density of the rings]
 For each $\alpha\in\acc(\mu^++1)$ and every integer $n<\omega$, $(\z{\mathbb{P}}^{\varpi^\alpha_n}_\alpha)_n$ is a dense subposet of $(\mathbb{P}^{\varpi^\alpha_n}_\alpha)_n$.\footnote{$(\z{\mathbb{P}}_\alpha)_n$ is as in Definition~\ref{DefRingForLimits}.}
\end{lemma}
\begin{proof}
This follows along the same lines of \cite[Lemma~4.28]{PartII}, with the only difference that now we use the following:
\begin{enumerate}
\item
At successor stages we can get into the ring $(\z{\mathbb{P}}_\alpha^{\varpi^\alpha_n})_n$ by $\LE^{\varpi^\alpha_n}_\alpha$-extending. This is granted by
Lemma~\ref{Runbdeddirectmore}.
\item
For all $\gamma<\alpha$, $(\z{\mathbb{P}}^{\varpi^\gamma_n}_\gamma)_n$ is $\sigma_n$-directed-closed. With this property we take care of the limit stages. (In \cite{PartII}, the full ring $(\z{\mathbb{P}}_\gamma)_n$ was $\sigma_n$-closed). We can make this replacement, because of the first item above.\qedhere
\end{enumerate}
\end{proof}

Now, we can appeal to the iteration scheme of Section~\ref{Iteration} with these building blocks,
and obtain, in return, a sequence $\langle (\mathbb P_{\alpha},\lh_{\alpha},c_{\alpha},\vec{\varpi}_{\alpha})\mid 1\le \alpha\le\mu^+\rangle$ of $(\Sigma,\vec{\mathbb{S}})$-Prikry quadruples.
By Lemma~\ref{MoreClosednessOfIterates} and Theorem~\ref{CivIteration} (see also Remark~\ref{piclosure}),
for all nonzero $\alpha\le\mu^+$,
$(\z{\mathbb{P}}_\alpha)^{\pi_{\alpha,1}}_n$ is $\mu$-directed-closed
and
$\one_{\mathbb P_\alpha}\forces_{\mathbb{P}_\alpha}\check{\mu}=\check\kappa^+$.
Note that by the first clause of Goal~\ref{goals}, $|P_\alpha|\le\mu^+$ for every $\alpha\le\mu^+$.

\label{KeyObservationMoreClosedness}

\begin{lemma}\label{suitableindeed} Let $n\in\omega\setminus2$ and $\alpha\in[2,\mu^+)$.
Then $((\mathbb{P}_\alpha)_n,\mathbb{S}_n,\varpi^\alpha_n)$ is suitable for reflection with respect to $\langle \sigma_{n-2},\kappa_{n-1},\kappa_n,\mu \rangle$.
\end{lemma}
\begin{proof} We go over the clauses of Lemma~\ref{Liftingsuitability}
with $\mathbb{P}_\alpha$ playing the role of $\mathbb A$,
$\varpi_n^\alpha$ playing the role of $\varsigma_n$,
and $\mathbb{P}_1$  playing the role of $\mathbb P$.

As $\mathbb P_1$ is given by \blkref{I},
which is given by Section~\ref{SectionGitikForcing},
we simplify the notation here, and --- for the scope of this proof --- we let $\mathbb P$ denote the forcing $\mathbb P$ from Section~\ref{SectionGitikForcing}.

Clause~\ref{Liftingsuitability1} is part of the assumptions of Setup~\ref{ReflectionAfterIteration}.
Clauses~\ref{Liftingsuitability2} and \ref{Liftingsuitability3} are given by our iteration theorem.
Clause~\ref{Liftingsuitability4} is due to Corollary~\ref{Motisuitableforreflection},\footnote{Since $(\mathbb{P}_n,\mathbb{S}_n,\varpi_n)$ is suitable for reflection with respect to $\langle\sigma_{n-1},\kappa_{n-1},\kappa_n,\mu\rangle$ then so is with respect to $\langle\sigma_{n-2},\kappa_{n-1},\kappa_n,\mu\rangle$. }
and the fact that $\mathbb P_1$ is the Gitik's EBPFC.
Now, we turn to address Clause~\ref{Liftingsuitability5}.

That is, we need to prove that in any generic extension by
$\mathbb{S}_n\times(\mathbb{P}_\alpha)_n^{\varpi^\alpha_n}$,
$$|\mu|=\cf(\mu)=\kappa_n=(\kappa_{n-1})^{++}.$$

The upcoming discussion assumes the notation of Section~\ref{SectionGitikForcing}.
By Lemma~\ref{keysubclaim}, we have:
\begin{enumerate}
\item $\mathbb T_n$ has the $\kappa_n$-cc and has size $\kappa_n$;
\item $\psi_n$ defines a nice projection;
\item $\mathbb{P}^{\psi_n}_n$ is  $\kappa_n$-directed-closed;
\item for each $p\in P_n$,   $\mathbb{P}_n\downarrow p$ and
$(\mathbb{T}_n\downarrow \psi_n(p))\times ((\mathbb{P}^{\psi_n})_n\downarrow p)$ are forcing equivalent.
\end{enumerate}

By Lemma~\ref{CardinalConfigurationMotisModel}, $\mathbb{P}_n$ forces $|\mu|=\cf(\mu)=\kappa_n=(\sigma_n)^+=(\kappa_{n-1})^{++}$, and by our remark before the statement of this lemma,
$\mathbb{(\mathbb{P}_\alpha)}_n^{\pi_{\alpha,1}}$ is $\mu$-directed-closed, hence $\kappa_n$-directed-closed.
Combining Clauses \eqref{keysubclaim1}, \eqref{keysubclaim2} and \eqref{keysubclaim3} above with Easton's Lemma, $(\mathbb{P}^{\psi_n})_n\times \mathbb{(\mathbb{P}_\alpha)}_n^{\pi_{\alpha,1}}$ is $\kappa_n$-distributive over $V^{\mathbb{T}_n}$, and so $\mathbb{P}_n\times \mathbb{(\mathbb{P}_\alpha)}_n^{\pi_{\alpha,1}}$ forces $\kappa_n=(\kappa_{n-1})^{++}$. Moreover,  as $\mathbb{P}_n\times \mathbb{(\mathbb{P}_\alpha)}_n^{\pi_{\alpha,1}}$ projects to $\mathbb{P}_n$ and the former preserves $\kappa_n$,  it also forces $|\mu|=\cf(\mu)$. Altogether, $\mathbb{P}_n\times \mathbb{(\mathbb{P}_\alpha)}_n^{\pi_{\alpha,1}}$ forces $|\mu|=\cf(\mu)=\kappa_n=(\kappa_{n-1})^{++}$.
To establish that the same configuration is being forced by $\mathbb{S}_n\times(\mathbb{P}_\alpha)_n^{\varpi^\alpha_n}$,
we give a sandwich argument, as follows:
 \begin{itemize}
 \item $\mathbb{P}_n\times(\mathbb{P}_\alpha)_n^{\pi_{\alpha,1}}$ projects to $\mathbb{S}_n\times(\mathbb{P}_\alpha)_n^{\varpi^\alpha_n}$,
as witnessed by $(p,q)\mapsto (\varpi_n(p), q)$;
 \item For any condition $p$ in $(\mathbb{P}_\alpha)_n$, $(\mathbb{S}_n\downarrow \varpi^\alpha_n(p))\times((\mathbb{P}_\alpha)_n^{\varpi^\alpha_n}\downarrow p)$ projects to $(\mathbb{P}_\alpha)_n\downarrow p$,
by Definition~\ref{niceprojection}\eqref{theprojection}.
 \item  $(\mathbb{P}_\alpha)_n$ projects to $\mathbb{P}_n$ via $\pi_{\alpha,1}$.
 \end{itemize}
This completes the proof.
\end{proof}

\begin{lemma}\label{preservegch} Let $n<\omega$ and $0<\alpha<\mu^+$.
Then $(\mathbb{P}_\alpha)_n^{\varpi^\alpha_n}$ preserves $\gch$.
\end{lemma}
\begin{proof} The case $\alpha=1$ is taken care of by Lemma~\ref{GCHmoti}.

Now, let $\alpha\ge2$.
Since $(\mathbb{P}_\alpha)_n^{\varpi^\alpha_n}$ contains a $\sigma_n$-directed-closed dense subset, it preserves $\gch$ below $\sigma_n$.
By the sandwich analysis from the proof of Lemma~\ref{suitableindeed},
in any generic extension by $(\mathbb{P}_\alpha)_n^{\varpi^\alpha_n}$, $|\mu|=\cf(\mu)=\kappa_n=(\sigma_{n})^+$.
So, as $(\mathbb{P}_\alpha)_n^{\varpi^\alpha_n}$ is a notion of forcing of size $\le\mu^+$, collapsing $\mu$ to $\kappa_n$,
it preserves $\gch_\theta$ for any cardinal $\theta>\kappa_n$.

It thus left to verify that $(\mathbb{P}_\alpha)_n^{\varpi^\alpha_n}$ forces $2^\theta=\theta^+$ for $\theta\in\{\sigma_n,\kappa_n\}$.

$\br$ Arguing as in Lemma \ref{suitableindeed}, for any condition $p$ in $(\mathbb P_\alpha)_n$, $(\mathbb{T}_n\downarrow \psi_n(p))\times \left(((\mathbb{P}^{\psi_n})_n\downarrow p)\times (\mathbb{P}_\alpha)_n^{\pi_{\alpha,1}}\right)$ projects onto $(\mathbb{P}_\alpha)_n^{\varpi^\alpha_n}$.
Recall that the first factor of the product is a $\kappa_n$-cc forcing of size $\le\kappa_n$.
By Lemma~\ref{MoreClosednessOfIterates}, the second factor is forcing equivalent to a $\kappa_n$-directed-closed forcing.
Thus, by Easton's lemma, this product preserves $\ch_{\sigma_n}$ if and only if $\mathbb{T}_n\downarrow \psi_n(p)$ does.
And this is indeed the case, as the number of $\mathbb{T}_n$-nice names for subsets of  $\sigma_n$ is at most $\kappa_n^{<\kappa_n}=\kappa_n=(\sigma_n)^+$.

$\br$ Again, arguing as in Lemma \ref{suitableindeed},  $(\mathbb{P}_1)_n\times (\mathbb{P}_\alpha)_n^{\pi_{\alpha,1}}$
projects onto $\mathbb{S}_n\times(\mathbb{P}_\alpha)_n^{\varpi^\alpha_n}$, which projects onto $(\mathbb{P}_\alpha)_n^{\varpi^\alpha_n}$.
Since $(\mathbb{P}_\alpha)_n^{\pi_{\alpha,1}}$ is forcing equivalent to a $\mu$-directed-closed, it preserves $\ch_{\sigma_n}$.
Also, it preserves $\mu$ and so, by Lemma \ref{GCHmoti}(1) and the absolutness of the $\mu^+$-Linked property, $(\mathbb{P}_1)_n$ is also $\mu^+$-Linked in $V^{(\mathbb{P}_\alpha)_n^{\pi_{\alpha,1}}}$. Once again, counting-of-nice-names arguments implies that this latter forcing forces $2^{\kappa_n}\leq \mu^+=(\kappa_n)^+$. Thus, $(\mathbb{P}_1)_n\times (\mathbb{P}_\alpha)_n^{\pi_{\alpha,1}}$  preserves $\ch_{\kappa_n}$ and so does $(\mathbb{P}_\alpha)_n^{\varpi^\alpha_n}$.
\end{proof}

\begin{theorem}\label{TheoremReflection} In $V^{\mathbb{P}_{\mu^+}}$, all of the following hold true:
\begin{enumerate}
\item All cardinals $\geq \kappa$ are preserved;
\item $\kappa=\aleph_\omega$, $\mu=\aleph_{\omega+1}$ and $\lambda=\aleph_{\omega+2}$;
\item $2^{\aleph_n}=\aleph_{n+1}$ for all $n<\omega$;
\item $2^{\aleph_\omega}=\aleph_{\omega+2}$;
\item Every stationary subset of $\aleph_{\omega+1}$ reflects.
\end{enumerate}
\end{theorem}
\begin{proof}
(1) We already know that $\one_{\mathbb P_{\mu^+}}\forces_{\mathbb{P}_\alpha}\check{\mu}=\check\kappa^+$.
By Lemma~\ref{l14}\eqref{l14(3)}, $\kappa$ remains strong limit cardinal in $V^{\mathbb{P}_{\mu^+}}$.
Finally, as Clause~\eqref{c1} of Definition~\ref{SigmaPrikry} holds for $(\mathbb P_{\mu^+},\lh_{\mu^+},c_{\mu^+},\vec{\varpi}_{\mu^+})$,
$\mathbb P_{\mu^+}$ has the $\mu^+$-chain-condition, so that all cardinals $\geq\kappa^{++}$ are preserved.

(2) Let $G\s \mathbb{P}_{\mu^+}$ be an arbitrary generic over $V$.
By virtue of Clause~(1) and Setup~\ref{ReflectionAfterIteration},
it suffices to prove that  $V[G]\models\kappa=\aleph_{\omega}$.
Let $G_1$ the $\mathbb{P}_1$-generic filter generated by $G$ and $\pi_{\mu^+,1}$.
By Theorem~\ref{MotisModel}, $V[G_1]\models\kappa={\aleph_\omega}$.
Thus, let us prove that $V[G]$ and $V[G_1]$ have the same cardinals $\leq \kappa$.

Of course, $V[G_1]\s V[G]$, and so any $V[G]$-cardinal is also a $V[G_1]$-cardinal.
Towards a contradiction, suppose that there is a $V[G_1]$-cardinal $\theta<\kappa$ that ceases to be so in $V[G]$.
Any surjection witnessing this can be encoded as a bounded subset of $\kappa$, hence as a bounded subset of some $\sigma_n$ for some $n<\omega$.
Thus,  Lemma~\ref{l14}\eqref{l14(1)} implies that $\theta$ is not a cardinal in $V[H_n]$,
where $H_n$ is the $\mathbb{S}_n$-generic filter generated by $G_1$ and $\varpi^1_n$. As $V[H_n]\s V[G_1]$,  $\theta$ is not a cardinal in $V[G_1]$, which is a contradiction.

(3) On one hand, by Lemma~\ref{l14}\eqref{l14(1)}, $\mathcal{P}(\aleph_n)^{V^{\mathbb{P}_{\mu^+}}}=\mathcal{P}(\aleph_n)^{V^{\mathbb{S}_m}}$ for some $m<\omega$.
On the other hand, as $\gch_{<\lambda}$ holds (cf. Setup~\ref{ReflectionAfterIteration}), Remark~\ref{RemarkonS}
 shows that  $\mathbb{S}_m$ preserves  $\ch_{\aleph_n}$. Altogether,  $V^{\mathbb{P}_{\mu^+}}\models \ch_{\aleph_n}$.

(4) By Setup~\ref{ReflectionAfterIteration}, $V\models 2^\kappa=\kappa^+$.
In addition, $\mathbb{P}_{\mu^+}$ is isomorphic to a notion of forcing lying in $H_{\mu^+}$ (see \cite[Remark~3.3(1)]{PartII})
and $|H_{\mu^+}|=\lambda$. Thus, $V^{\mathbb{P}_{\mu^+}}\models 2^\kappa\le\lambda$.
In addition,  $\mathbb P_{\mu^+}$ projects to $\mathbb P_1$, which is isomorphic to $\mathbb Q$,
being a poset blowing up $2^\kappa$ to $\lambda$, as seen in Theorem~\ref{MotisModel},
so that $V^{\mathbb{P}_{\mu^+}}\models 2^\kappa\ge\lambda$.
So, $V^{\mathbb{P}_{\mu^+}}\models 2^\kappa=\lambda$.
Thus, together with Clause~(2), $V^{\mathbb{P}_{\mu^+}}\models 2^{\aleph_\omega}=\aleph_{\omega+2}$.

(5)  Let $G$ be $\mathbb P_{\mu^+}$-generic over $V$ and hereafter work in $V[G]$.
Towards a contradiction, suppose  that there exists a stationary set $T\s \mu$ that does not reflect.
By shrinking, we may assume the existence of some regular cardinal $\theta<\mu$ such that $T\s E^\mu_\theta$.
Fix $r^*\in G$ and a $\mathbb P_{\mu^+}$-name $\tau$ such that $\tau_G$ is equal to such a $T$
and such that $r^*$ forces $\tau$ to be a stationary subset of $\mu$
that does not reflect. Since $\mu=\kappa^+$ and $\kappa$ is singular in $V$, by possibly enlarging $r^*$, we may assume that $r^*$ forces $\tau$ to be a subset of $\Gamma_{\ell(r^*)}$ (see the opening of Subsection~\ref{subsection62}).
Furthermore, we may require that $\tau$ be a \emph{nice name}, i.e., each element of $\tau$ is a pair $(\check \xi,p)$ where $(\xi,p)\in \Gamma_{\ell(r^*)} \times P_{\mu^+}$,
and, for each ordinal $\xi\in \Gamma_{\ell(r^*)}$, the set $\{ p\in P_{\mu^+}\mid (\check \xi,p)\in \tau\}$ is a maximal antichain.

As $\mathbb P_{\mu^+}$ satisfies Clause~ \eqref{c1} of Definition~\ref{SigmaPrikry}, $\mathbb P_{\mu^+}$ has in particular the $\mu^+$-cc.
Consequently, there exists a large enough $\beta<\mu^+$ such that $$B_{r^*}\cup\bigcup\{ B_p\mid (\check{\xi},p)\in\tau\}\s\beta.$$
Let $r:=r^*\restriction\beta$ and set $$\sigma:=\{(\check{\xi},p\restriction\beta)\mid (\check{\xi},p)\in\tau\}.$$
From the choice of \blkref{III}, we may find a large enough $\alpha<\mu^+$ with $\alpha>\beta$ such that $\psi(\alpha)=(\beta,r,\sigma)$.
As $\beta<\alpha$, $r\in P_\beta$ and $\sigma$ is a $\mathbb P_\beta$-name,
the definition of our iteration at step $\alpha+1$ involves appealing to \blkref{II}
with $(\mathbb P_\alpha,\lh_\alpha,c_\alpha,\vec{\varpi}_\alpha)$,
$r^\star:=r*\emptyset_\alpha$ and $z:=i^{\alpha}_\beta(\sigma)$.\footnote{Recall  Convention~\ref{conv71}.}
For each ordinal $\eta<\mu^+$, denote $G_\eta:=\pi_{\mu^+,\eta}[G]$.
By our choice of $\beta$ and since $\alpha>\beta$, we have
$$\tau=\{ (\check{\xi},p*\emptyset_{\mu^+})\mid (\check{\xi},p)\in\sigma\}=\{ (\check{\xi},p*\emptyset_{\mu^+})\mid (\check{\xi},p)\in z\},$$
so that, in $V[G]$,
$$T=\tau_{G}=\sigma_{G_\beta}=z_{G_\alpha}.$$
In addition, $r^*=r^\star*\emptyset_{\mu^+}$ and so $\ell(r^*)=\ell(r^\star)$.

As $r^*$ forces that $\tau$ is a non-reflecting stationary subset of $\Gamma_{\lh(r^\star)}$, it follows that $r^\star$ $\mathbb{P}_\alpha$-forces the same about $z$.
\begin{claim} $z$ is $r^\star$-fragile.
\end{claim}
\begin{proof}
Recalling Lemma~\ref{key}, it suffices to prove that
for every $n<\omega$,
$$V^{{(\mathbb{P}_\alpha)}_n}\models \refl(E^{\mu}_{< \sigma_{n-2}}, E^{\mu}_{{<}\sigma_n}).$$
This is trivially the case for $n\le1$. So, let us fix an arbitrary $n\ge2$.
By Lemma~\ref{suitableindeed}, $((\mathbb{P}_\alpha)_n,\mathbb{S}_n,\varpi^\alpha_n)$ is suitable for reflection with respect to $\langle \sigma_{n-2},\kappa_{n-1},\kappa_n,\mu \rangle$.
Since $(\mathbb{P}_\alpha)_n^{\varpi^\alpha_n}$ is forcing equivalent to a $\sigma_n$-directed-closed forcing and (by Lemma \ref{preservegch}) it preserves $\gch$,
$\kappa_{n-1}$ is a supercompact cardinal indestructible under forcing with $(\mathbb{P}_\alpha)_n^{\varpi^\alpha_n}$.
So, recalling Setup~\ref{ReflectionAfterIteration}, $(\mathbb{P}_\alpha)_n^{\varpi^\alpha_n}$ preserves the supercompactness of $\kappa_{n-1}$.
Thus, by Lemma~\ref{ebfreflection},
$V^{{(\mathbb{P}_\alpha)}_n}\models \refl(E^{\mu}_{< \sigma_{n-2}}, E^{\mu}_{{<}\sigma_n})$.
\end{proof}

As $z$ is $r^\star$-fragile and  $\pi_{\mu^+,\alpha+1}(r^*)=r^\star*\emptyset_{\alpha+1}=\myceil{r^\star}{\mathbb P_{\alpha+1}}\in G_{\alpha+1}$,
Clause~\ref{C8onestep} of \blkref{II} implies that
there exists (in $V[G_{\alpha+1}]$) a club subset of $\mu$ disjoint from $T$. In particular, $T$ is nonstationary in $V[G_{\alpha+1}]$ and thus nonstationary in $V[G]$. This contradicts the very choice of  $T$.
 The result follows from the above discussion and the previous claim.
\end{proof}

We are now ready to derive the Main Theorem.

\begin{theorem} Suppose that there exist infinitely many supercompact cardinals.
Then there exists a  forcing extension where all of the following hold:
\begin{enumerate}
\item $2^{\aleph_n}=\aleph_{n+1}$ for all $n<\omega$;
\item $2^{\aleph_\omega}=\aleph_{\omega+2}$;
\item every stationary subset of $\aleph_{\omega+1}$ reflects.
\end{enumerate}
\end{theorem}
\begin{proof} Using Corollary~\ref{cor84}, we may assume that all the blanket assumptions of Setup~\ref{ReflectionAfterIteration} are met. Specifically:
 \begin{itemize}
 \item  $\vec{\kappa}=\langle \kappa_n\mid n<\omega\rangle$ is an increasing sequence of supercompact cardinals that are indestructible under $\kappa_n$-directed-closed notions of forcing that preserve the $\gch$;
\item $\kappa:=\sup_{n<\omega}\kappa_n$, $\mu:=\kappa^+$ and $\lambda:=\kappa^{++}$;
\item  $\gch$ holds. 
\end{itemize}
Now, appeal to Theorem~\ref{TheoremReflection}.
\end{proof}

\section{Acknowledgments}
Poveda was partially supported by a postdoctoral fellowship from the Einstein Institute of Mathematics of the Hebrew University of Jerusalem.
Rinot was partially supported by the European Research Council (grant agreement ERC-2018-StG 802756) and by the Israel Science Foundation (grant agreement 2066/18).
Sinapova was partially supported by the National Science Foundation, Career-1454945 and  DMS-1954117.

\end{document}